\documentclass[12pt]{amsart}
\usepackage{amsmath}
\usepackage{amssymb}
\usepackage{amscd}
\usepackage{graphicx}
\usepackage{calrsfs}
\usepackage{enumitem}

\newtheorem*{TA}{\sc theorem A}
\newtheorem*{TAP}{\sc theorem A'}

\newtheorem{thm}{\sc theorem}[section]
\newtheorem{pro}[thm]{\sc proposition}
\newtheorem{lem}[thm]{\sc lemma}
\newtheorem{cor}[thm]{\sc corollary}

\newtheorem{clm}[thm]{\sc claim}

\theoremstyle{definition}

\newtheorem{dfn}[thm]{\sc definition}
\newtheorem{ntt}[thm]{\sc notation}
\newtheorem{voc}[thm]{\sc vocabulary}

\newtheorem{ccl}[thm]{\sc conclusion}

\theoremstyle{remark}

\newtheorem{rmk}[thm]{\sc remark}
\newtheorem{note}[thm]{\sc note}

\newtheorem{exm}[thm]{\sc example}

\newtheorem{ppt}[thm]{\sc property}

\newcommand{\MM}{\mbox{$\overline{M}$}}

\newcommand{\supp}{\mbox{$\mbox{\rm spt}$}}

\newcommand{\pr}{\mbox{$\mbox{\rm pr}$}}

\def\codim{{\rm codim}}

\def\G{{\Gamma}}

\def\nb{{\mathcal Op}}

\def\Z{{\bf Z}}

\def\R{{\bf R}}
\def\C{{\bf C}}

\def\S{{\bf S}}

\def\D{{\bf D}}
\def\I{{\bf I}}

\def\AA{{\mathcal A}}
\def\BB{{\mathcal B}}
\def\CC{{\mathcal C}}

\def\FF{{\mathcal F}}
\def\GG{{\mathcal G}}
\def\HH{{\mathcal H}}

\def\MM{{\mathcal M}}
\def\NN{{\mathcal N}}
\def\PP{{\mathcal P}}

\def\SS{{\mathcal S}}

\def\XX{{\mathcal X}}

\def\id{{\rm id}}
\def\mun{{^{-1}}}

\def\diff{{\rm Diff}}

\begin{document}

\title[Quasi-complementary foliations]{
Quasi-complementary foliations and the Mather-Thurston
theorem}

\begin{abstract}
We establish a form of the h-principle for the existence of foliations of codimension at least $2$
which are quasi-complementary to a given one.
Roughly, ``quasi-complementary''
means
 that they are complementary but on the boundaries of some kind of Reeb components. The construction involves an adaptation of
 W. Thurston's ``inflation'' process.
 The same methods
also provide a proof of the classical Mather-Thurston theorem.

\end{abstract}
\maketitle

\begin{center}
Ga\"el
Meigniez

\end{center}
\date{\today}
\section{introduction}\label{introduction_sec}

 \subsection{Quasi-complementary foliations}\label{qc_ssec}
 Given, on a manifold $M$, a dimension-$q$ foliation $\FF$, the existence of a
foliation $\GG$ complementary to $\FF$ (that is, $\GG$ is of
 cod\-im\-ension $q$ and transverse to $\FF$) is
 of course in general an intractable problem.
  In this paper, we weaken
  the transversality condition, prescribing a simple (and classical)
   model for the tangentialities between $\FF$ and $\GG$ which implies that $\GG$
  is a limit of plane fields complementary to $\FF$ but themselves not necessarily integrable.
  We establish, when $q\ge 2$, a form of Gromov's h-principle
 for such ``quasi-complementary'' foliations.
 \medbreak
 Here is a very elementary example similar to what we call quasi-compl\-ement\-ar\-ity,
  although $q=1$. Consider
  the Hopf foliation $\FF$ of the $3$-sphere $\S^3$ by circles. The classical geometric
theory of foliations  shows that $\FF$ admits
   no complementary foliation $\GG$: indeed, by the Novikov closed leaf theorem
(see for example \cite{camacho_neto_85}
or \cite{candel_conlon_00}),
  $\GG$ would have a compact leaf
which would separate $\S^3$, in contradiction to the transversality to $\FF$;
alternatively, one can argue that the Hopf fibration would then be a foliated bundle
(\cite{camacho_neto_85} pp. 99-100, or \cite{candel_conlon_00} example 2.1.5) over a simply-connected base $\S^2$;
hence all leaves of $\GG$
 would be diffeomorphic to the base; and by the Reeb global stability theorem
(\cite{camacho_neto_85} ch. IV theorem 4, or \cite{candel_conlon_00} theorem 6.1.5),
the total space would be $\S^2\times\S^1$, not $\S^3$.

However, it is easily verified that
 the sphere has
  a Reeb foliation $\GG$ which is complementary to the circles but on its unique compact leaf,
   which is tangential to them. Moreover, $\GG$ is a limit of $2$-plane fields
  complementary to the circles, provided that one makes its two Reeb components ``turbulize''
 in 
  appropriate directions: precisely, the holonomy of
$\GG$ along any circle fibre in the compact leaf must be contracting on one side of the leaf,
 and expanding
on the other. 
 \medbreak

The models for the tangentialities are classical, being nothing but
W. Thurston's construction to fill holes in codimensions $2$ and more (\cite{thurston_74},
section 4).
To fix ideas, the smooth ($C^\infty$) differentiability class is understood everywhere,
unless otherwise specified. 
 On the interval $\I:=[0,1]$, fix 
a smooth real function $r\mapsto u(r)$ such that $u'(r)>0$ for $0<r<1$, and
 $u(r)$ and all its successive derivates vanish at $r=0$, and $u(r)+u(1-r)=1$.
Write $\D^n$ (resp. $\S^{n-1}$) for the compact unit ball (resp. sphere) in $\R^n$;
endow $\D^2$ with the polar coordinates $\rho, \theta$;
endow $\S^1$ with the coordinate $s$;
 on  $\D^2\times\S^1$, one has
 the smooth $1$-parameter family $(\omega_r)_{r\in\I}$ of
 smooth nonsingular $1$-forms defined for $0\le r\le 1/2$ by
 $$\omega_{r}:=u(1-2\rho)ds+u(2\rho)d\rho$$
 (on $\{\rho\le 1/2\}$) and
 $$\omega_{r}:=u(2-2\rho)d\rho+u(2\rho-1)(ds-u(1-2r)d\theta)$$
 (on $\{\rho\ge 1/2\}$);
 while for $1/2\le r\le 1$:
  $$\omega_r:=u(2-2r)\omega_{1/2}+u(2r-1)ds$$
  For every $r$, the form $\omega_r$ is integrable (it is
   immediately verified that
  $\omega_r\wedge d\omega_r=0$); hence one gets on  $\D^2\times\S^1$
   a $1$-parameter family
  of codimension-$1$ foliations,
seen on Figure \ref{movie_fig}.
\begin{figure}
\includegraphics*[scale=0.45, angle=-90]{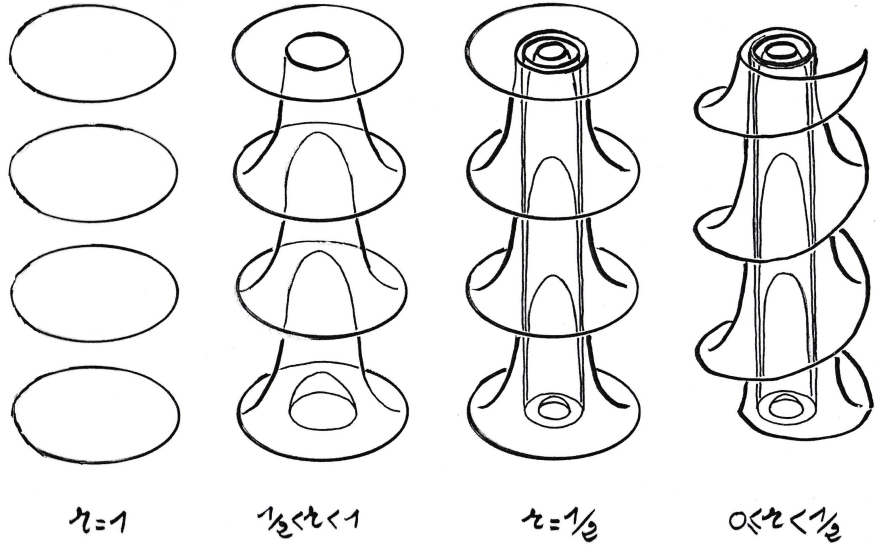}
\caption{W. Thurston's construction to fill holes in codimensions $2$ and more.}
\label{movie_fig}
\end{figure}

Fix $p, q\ge 2$. For every closed $(p-2)$-fold $\Sigma$,
we define the \emph{multifold Reeb component of core} $\Sigma$ as the $(p+q)$-fold
 $$C_\Sigma:=\Sigma\times\D^2\times\D^{q-1}\times\S^1$$
(whose projection to the $i$-th factor,  $1\le i\le 4$, will be denoted by $\pr_i$)
endowed with two foliations:
\begin{itemize}
\item  The dimension-$q$ foliation $\FF_\Sigma$ parallel to the factor
$\D^{q-1}\times\S^1$;
\item  The codimension-$q$ foliation $\GG_\Sigma$ obtained
by endowing, for every $a\in\D^{q-1}$, the fibre $\pr_3\mun(a)$ with
its codimension-$1$ foliation pullback of the $1$-form $\omega_{\vert a\vert}$
under $\pr_2\times\pr_4$.
\end{itemize}

\begin{dfn}\label{multifold_Reeb_dfn} On a $(p+q)$-fold $M$,
the codimension-$q$ foliation $\GG$ is \emph{quasi-complementary}
 to the dimension-$q$ foliation $\FF$ if they are transverse but maybe for finitely many disjoint
multifold Reeb components $C_\Sigma\hookrightarrow M$, in which
 $\FF$ (resp. $\GG$) coincides with $\FF_\Sigma$ (resp. $\GG_\Sigma$).

\end{dfn}

The components may have different cores; one can of course also consider the
union of the components as one component
whose core may be not connected.
Note that $\GG$ is almost everywhere complementary to $\FF$:
precisely, everywhere on $M$, but on the subset
defined in each component $C_\Sigma$ by $\rho\circ\pr_2=1/2$ and
 $\Vert\pr_3\Vert\le 1/2$, which is a hypersurface tangential to $\FF$.
Also, it is easily verified
 that $\GG$ is a limit of codimension-$q$ plane fields complementary to $\FF$
  on $M$.

 For the elements on
 Haefliger structures: normal bundle, \emph{differential,} concordances,
 regularity
 and relations to foliations, see Section
\ref{haefliger_sec} below. Recall the
 Foliation theorem in codimensions $2$ and more,
 also known as ``h-principle for foliations", on closed manifolds
(\cite{thurston_74},
see also \cite{eliashberg_mishachev_98} and \cite{mitsumatsu_vogt_16}).
We use the notation $\nb_X(Y)$ for ``some open neighborhood of $Y$ in $X$''.
  \begin{thm}[Thurston]\label{h_thm}
    On a compact manifold $M$, let $\nu$ be a real vector bundle of dimension $q\ge 2$,
 let $\gamma$ be
 a $\Gamma_q$-structure whose normal bundle is
 $\nu$, and let $\omega$ be a $1$-form valued in $\nu$ of constant rank $q$;
 assume that $d\gamma=\omega$
on $\nb_M(\partial M)$.
 
  Then,
  $M$ admits a \emph{regular} $\Gamma_q$-structure $\gamma'$ of normal bundle $\nu$
   such that:
    \begin{itemize} 
      \item
       $\gamma'=\gamma$ on $\nb_M(\partial M)$, 
  and  $\gamma'$ is concordant to $\gamma$ on $M$
(rel. $\partial M$);
     \item $d\gamma'$ is homotopic to $\omega$ on $M$ (rel. $\partial M$)
     among the  $1$-forms valued in $\nu$ of constant rank $q$.

\end{itemize}

\end{thm}

Our main result is a refinement of this classical one
for a manifold already foliated.

 \begin{TA}
 On a compact manifold $M$, let
 $\FF$ be a  foliation
 of dimension $q\ge 2$,
 let $\gamma$ be
 a $\Gamma_q$-structure whose normal bundle is
 ${\tau\FF}$, and let $\omega$ be a $1$-form valued in $\tau\FF$ such that $\omega\vert\tau\FF$
 is of constant rank $q$;
 assume that $d\gamma=\omega$
on $\nb_M(\partial M)$.

   Then,
  $M$ admits a \emph{regular} $\Gamma_q$-structure $\gamma'$ of normal bundle ${\tau\FF}$
   such that:
    \begin{itemize} 
      \item
       $\gamma'=\gamma$ on $\nb_M(\partial M)$, 
  and  $\gamma'$ is concordant to $\gamma$ on $M$
(rel. $\partial M$);
     \item $d\gamma'$ is homotopic to $\omega$ on $M$ (rel. $\partial M$)
     among the  $1$-forms valued in $\tau\FF$ of constant rank $q$;
     \item The foliation induced by $\gamma'$ is quasi-complementary to $\FF$ on $M$.

\end{itemize}
 \end{TA}

\begin{rmk}
 
  We put no restriction on the position of $\FF$ with respect to $\partial M$.
\end{rmk}
 
\begin{rmk}\label{differentiability_rmk} The construction does not use sophisticated results
on groups of diffeomorphisms. For every integer $r\ge 1$, the
 theorem \emph{A} holds as well in the differentiability class $C^r$,
 with the same proof.
 \end{rmk}

\begin{rmk}
  This result and our proof are also valid for $p=1$, in which case there exist no
 multifold Reeb components
 at all,
 hence ``quasi-complementary'' means ``complementary''. Of course,
 it is not a great deal to produce a foliation complementary to a given codimension-$1$
 foliation; however, the point is that there is one
  \emph{in every concordance class} of $\Gamma_q$-structures whose normal bundle is $\tau\FF$;
  and this goes in every differentiability class
  (Remark \ref{differentiability_rmk}), including $C^{q+1}$.
\end{rmk}

\begin{rmk} 
 By construction, for $p:=\codim(\FF)\ge 2$,
  the core $\Sigma$ of each multifold Reeb component resulting from our construction will be a
 product of two spheres $\S^{i-1}\times\S^{p-i-1}$, with $1\le i\le p-1$: see
 Proposition \ref{core_pro} below. One can if one likes better,
 by a trick due to Thurston, arrange that $\Sigma$ is the $(p-2)$-torus (\cite{thurston_76},
 beginning of section 5);
 or,  for $p\ge 3$, that $\Sigma=\S^1\times\S^{p-3}$  (\cite{meigniez_17},
 sections 3.3.2 and 3.3.3).

 \end{rmk}
 
 Denote,
 as usual, by $B\Gamma_q^r$ (resp. $B\bar\Gamma_q^r$)
  the Haefliger space  classifying the   $\Gamma$-structures
  (resp. parallelized  $\Gamma$-structures)
 of codimension $q$ and differentiability class $C^r$.
 One has in particular the following corollaries of Theorem \emph{A}, 
 since
  $B\bar\Gamma_q^1$ is contractible
   \cite{tsuboi_89} and since
  $B\bar\Gamma_q^\infty$ is $(q+1)$-connected \cite{thurston_74_2}.

\begin{cor}  Let $\FF$ be a $C^\infty$
 foliation of dimension at least $2$ on a closed manifold.

 Then,  $\FF$
  admits a quasi-complem\-ent\-ary foliation of class $C^1$.
  
        If moreover
         $\codim(\FF)=2$, or if the bundle $\tau\FF$ is trivializable,
   then $\FF$  admits a quasi-complem\-ent\-ary foliation of class $C^\infty$.
\end{cor}

\subsubsection{About the proof}
Theorem \emph{A} is better established under a version A' producing
 a \emph{cleft} foliation $\GG$ with fissures (see Section \ref{holes_sec})
instead of multifold Reeb components; the fissures are some kind of discontinuities
in the foliation that
 have a product
structure $\Sigma\times\D^q$ where $\Sigma$ is a compact $(p-2)$-fold;
in that frame, quasi-complementarity means that outside the fissures, $\GG$
 is complementary to $\FF$, and that on the fissures,
  every slice parallel to $\D^q$ is a plaque of $\FF$.
 For $q\ge 2$, the versions A and A' are straightforwardly equivalent
to each other
through  Thurston's method to
  fill the holes; but A'  also holds
 in codimension $q=1$.

 To prove A', the problem  is translated, using the Gromov-Phillips-Haefliger
parametric Foliation theorem on open manifolds, into an extension problem
whose proof falls to
 an adaptation of the original ``inflation'' process that Thurston introduced to prove
 Theorem \ref{h_thm}.
 
 We feel that the present work
 illustrates the power and the accuracy,  in the framework of Gromov's h-principle,
   of the tools that W. Thurston left to us after his early works on foliations.

 \subsection{A proof of the Mather-Thurston theorem}\label{MT_ssec}
 
 A second application of our method deals with the construction of foliated products,
 and more generally of foliated bundles. We begin with discussing foliated products.
 
 In the case where the given foliation $\FF$ is a product foliation, we can
  get full complementarity
 at the price of modifying the base factor of the product by some surgeries; in other words
 we give a proof of the classical Mather-Thurston theorem
 \cite{thurston_74_2}, by means of a geometric
 construction pertaining to the h-principle.

 See  J. Mather's proof in codimension $1$  in
\cite{mather_73}; for the general codimensions,
Mather (\cite{mather_76} pp. 79--80)
mentions that Thurston had three different
 proofs; the first seems to be lost;
see \cite{mather_75}\cite{mather_76}\cite{mcduff_79}\cite{mcduff_80}\cite{segal_78}
for the two other proofs; see also \cite{kupers_19} for a modern one. 

 \medbreak
 Precisely, fix a manifold $X$  of dimension $q\ge 1$, without boundary,
 not necessarily compact. 
   ``At infinity'' means as usual
    ``except maybe on some compact subset''.
For a compact oriented manifold $V$ of dimension $p\ge 0$ maybe with smooth boundary,
on $V\times X$,
consider the slice
foliation $\FF_V$ parallel to $X$; and
 the
 \emph{horizontal}  codimension-$q$  foliation parallel to $V$,
 or equivalently the \emph{horizontal} $\Gamma_q$-structure
 whose differential is
the projection $$\tau(V\times X)\to\tau\FF_V$$ parallely to $\tau V$.

Recall that a \emph{foliated $X$-product} over $V$ means a
codim\-ension-$q$ foliation on
$V\times X$  complementary to $\FF_V$ on $V\times X$ and horizontal at infinity.
Equivalently, this
 amounts to  a $\Gamma_q$-structure on $V\times X$
whose normal bundle is $\tau\FF_V$, whose differential
induces the identity on $\tau\FF_V$, and horizontal at infinity.

We use the notation $\hat V\subset V\times\I$ for the union of
$V\times 0$ with $\partial V\times\I$.
An \emph{oriented cobordism $(V,W,V^*)$ rel. $\partial V$}
 means as usual a compact oriented $p$-fold
$V^*$ together with an oriented diffeomorphism $\partial V^*\cong\partial V$, and
a compact oriented $(p+1)$-fold
$W$ bounded by
$-\hat V\cup_{\partial V} V^*$.
We write 
 $\hat\pr$ for the restriction to $\hat V\times X\subset V\times\I\times X$
 of the projection
$V\times\I\times X\to V\times X$.
 We shall prove:

 \begin{thm}[Mather-Thurston, version ``for geometrically minded topologists'']\label{MT_thm}
  Let $V$ be a compact oriented manifold; let $\gamma$ be
  a $\Gamma_q$-structure on $V\times X$, of normal bundle $\tau\FF_V$,
   horizontal at infinity, and
 restricting to
 a foliated $X$-product over $\partial V$.
  
  Then,
 there are an oriented cobordism $(V,W,V^*)$ rel. $\partial V$ and
a $\Gamma_q$-structure
 on $W\times X$, of normal bundle $\tau\FF_W$, horizontal at infinity, coinciding with $\hat\pr^*(\gamma)$
on $\hat V\times X$, and restricting to a foliated $X$-product over $V^*$.
 
 \end{thm}

 This interpretation of the Mather-Thurston theorem
  was introduced by D.B. Fuchs \cite{fuchs_74}.
Let us now recall how the classical version of the theorem can be deduced.

Consider the group $\diff_c(X)$ of the compactly supported diffeomorphisms of $X$,
endowed with the smooth topology; the same group $\diff_c(X)_\delta$ with the discrete topology;  the identity map
$$\diff_c(X)_\delta\to\diff_c(X)$$
and  the topological group $\overline{\diff_c(X)}$ that is the
 homotopy-theoretic fibre  of this map.
Recall that its classifying space $B\overline{\diff_c(X)}$
also classifies the foliated $X$-products.

 Realize
 $B\Gamma_q$ as a fibred space over
$BO_q$ with fibre $B\bar\Gamma_q$.  Denote by $\XX$
the foliation of $X$ by points (regarded as a $\Gamma_q$-structure
on $X$ whose normal bundle is $\tau X$ and whose differential is the identity
of $\tau X$).
Consider the space $\G_c(\tau X)$  of the maps $$f:X\to B\Gamma_q$$ such that $f$
 lifts the map
$X\to BO_q$ classifying $\tau X$, and such that $f$ classifies $\XX$ 
at infinity.
One can view $\G_c(\tau X)$ as the space of the
$\Gamma_q$-structures on $X$ whose normal bundle is $\tau X$
and coinciding with $\XX$ at infinity.
For $X=\R^q$, the space $\G_c(\tau\R^q)$ coincides with the
$q$-th loop space $\Omega^q(B\bar\Gamma_q)$.

Consider the map 
$$B\overline{\diff_c(X)}\overset{\alpha}{\longrightarrow}\G_c(\tau X)$$
 adjoint to the map $$X\times B\overline{\diff_c(X)}\to B\Gamma_q$$
 that classifies the $\Gamma_q$-structure of the total space of the universal
foliated $X$-product
 \cite{mather_73}\cite{thurston_74_2}.

Theorem \ref{MT_thm} amounts to say that $\alpha$ induces an isomorphism in
{oriented bordism.}
Equivalently, by
 the ``Hurewicz theorem for bordism groups''
(\cite{atiyah_hirzebruch_61}, see also \cite{eliashberg_galatius_mishachev_11},
appendix B), \emph{$\alpha$ induces
 an isomorphism in integral homology.}
 That last wording is the classical one. The map
  $\alpha$ is actually a homology equivalence:
   indeed, $B\bar\Gamma_q$ being $(q+1)$-connected
 in class $C^\infty$ \cite{thurston_74_2}, the space
 $\Gamma_c(\tau X)$ is simply connected.
 
 \medbreak There is also a well-known generalization
 of the Mather-Thurston theorem for foliated bundles
 rather than foliated products.
 
 Still consider a manifold $X$  of dimension $q\ge 1$; to simplify, assume that $X$ is closed.
As usual, an \emph{$X$-bundle} $(E,\pi)$
over a compact manifold $V$
is a smooth locally trivial bundle map  $\pi:E\to V$
 whose fibres are diffeomorphic with $X$; 
a \emph{foliated $X$-bundle} means moreover a
codim\-ension-$q$ foliation on
$E$  complementary to the fibres of $\pi$.
Equivalently, this
 amounts to  a $\Gamma_q$-structure on $E$
whose normal bundle is $\ker(d\pi)$ and whose differential
induces the identity on $\ker(d\pi)$.

  The generalized Mather-Thurston theorem can be stated as 
 an isomorphism in integral homology between
 the space $B\diff^\infty(X)_\delta$ that classifies the foliated $X$-bundles,
 and the classifying space for $\Gamma_q$-struc\-tures on $X$-bundles
 whose normal bundle is tangential to the fibres. For a construction of this
 last classifying space and a proof of the generalized theorem
 as a corollary of the usual Mather-Thurston theorem, see
 \cite{nariman_17}, section 1.2.2.

  Here is the version ``for geometrically minded topologists'' of the generalized
  Mather-Thurston theorem.
The notation $\hat V$ has the same meaning as above.
 \begin{thm}\label{GMT_thm}
  Let $(E,\pi)$ be an $X$-bundle over a compact oriented manifold
  $V$; let $\gamma$ be
  a $\Gamma_q$-structure on $E$, of normal bundle $\ker(d\pi)$,
 and
 restricting to
 a foliated $X$-bundle over $\partial V$.
  
  Then,
 there are
 \begin{itemize}
 \item An oriented cobordism $(V,W,V^*)$ rel. $\partial V$;
 \item An $X$-bundle $(E_W,\pi_W)$ over $W$, coinciding over $\hat V$
 with the pullback of $(E,\pi)$ through the projection $V\times\I\to V$;
 \item A $\Gamma_q$-structure
 on $E_W$, of normal bundle $\ker(d\pi_W)$, coinciding over $\hat V$ with
 the pullback of $\gamma$ through the projection $E\times\I\to E$,
  and restricting to a foliated $X$-bundle over $V^*$.
 \end{itemize}
 \end{thm}
 Our method to prove
  Theorem \ref{MT_thm} actually also gives \ref{GMT_thm}
 after a very few straightforward changes which are left to the reader.
 
   \begin{note}[Other differentiability classes]
   The Mather-Thurston theorem and its generalization for foliated bundles are well-known to
   hold for every differentiability class $C^r$, $r$ integer $\ge 1$. Our proof
   goes without change in these classes. Three other interesting classes are $C^0$,
   Lipschitz and $PL$. The validity of the
    Mather-Thurston theorem is classical in the classes $C^0$ and
   Lipschitz, and also in the class $PL$ in codimension $1$;
    but it remains open in class $PL$ for $q\ge 2$. Our methods cannot afford the $C^0$ class,
   but do afford the Lipschitz and $PL$ classes after a little precisions and adaptations; this material 
   will be covered in a forthcoming paper.
   \end{note}

Let us illustrate Theorem
\ref{GMT_thm} by two known
 corollaries.
Since
  $B\bar\Gamma_q^1$ is contractible
   \cite{tsuboi_89} and since
  $B\bar\Gamma_q^\infty$ is $(q+1)$-connected \cite{thurston_74_2}:
  \begin{cor}\
  
 i)  Every $X$-bundle over a closed oriented manifold is cobordant to
  a $C^1$-foliated $X$-bundle.
  
ii)   Every $X$-bundle over a closed oriented \emph{surface}
   is cobordant to
  a $C^\infty$-foliated $X$-bundle.
  \end{cor}
  For (ii), when $X$ itself is a surface, see also
   \cite{nariman_17}, remark 3.16.
 
\medbreak
\emph{Jenseits des Homotopieprinzips ---}
One could maybe speak of a ``c-principle'', with a
\emph{c} for ``cobordism". Recall that in situations where
 Gromov's famous \emph{h-principle}
 holds,
every ``formal'' object is \emph{homotopic} to a genuine object through the formal objects.
In the same way, let us say that the \emph{c-principle} holds when
every formal object is \emph{cobordant}
 to a genuine object through the formal objects.
Examples of results pertaining to the c-principle are
the Mather-Thurston theorem for foliated products or bundles,
 the Madsen-Weiss theorem
for fibrations whose fibres are surfaces
(see \cite{eliashberg_galatius_mishachev_11}), and the realization of taut compactly generated pseudogroups
by foliations of dimension $2$ and codimension $1$ \cite{meigniez_16}.
In the same spirit, see Fuchs' early paper \cite{fuchs_74} and Kupers' recent one
\cite{kupers_19}.

\medbreak
It is a pleasure to thank Mike Freedman,
 Fran\c cois Laudenbach, Nikolai Mish\-achev, Sam Nariman, Larry Siebenmann,
 and the anonymous referees,
for constructive exchanges which have benefited this work.

 \section{Haefliger structures}\label{haefliger_sec}
 In this section, we recall A. Haefliger's notion of $\Gamma_q$-\emph{structure}
   \cite{haefliger_58}\cite{haefliger_69} 
\cite{haefliger_70}\cite{haefliger_71}, 
   under the form of microfoliated bundle (this form was introduced in \cite{haefliger_69};
   see also \cite{milnor_70}). We fix some vocabulary, point out a few elementary facts,
 and prove (again) the two parametric forms
of the classical Foliation theorem on open manifolds.

   We write every real vector bundle $\nu$ over a manifold $M$
    under the form $\nu=(E,\pi,Z)$ where $E$ is the total space,
$\pi:E\to M$ is the projection, and $Z:M\to E$ is the zero section.
\subsubsection{Definition}\label{definition_sssec}

A $\Gamma_q$-{structure} $\gamma$ on $M$ is given by
 \begin{itemize}
 \item A real vector bundle $\nu=(E,\pi,Z)$ of dimension $q$ over $M$;
 
 \item An open neighborhood $U$ of $Z(M)$ in $E$;
 
\item On $U$, a cod\-imens\-ion-$q$ foliation $\MM$
   transverse to every fibre.
\end{itemize}

One calls $\nu$ the \emph{normal bundle,} and $\MM$ the  \emph{microfoliation}
(Figure \ref{gamma_fig}).
\begin{figure}
\includegraphics*[scale=0.45, angle=-90]{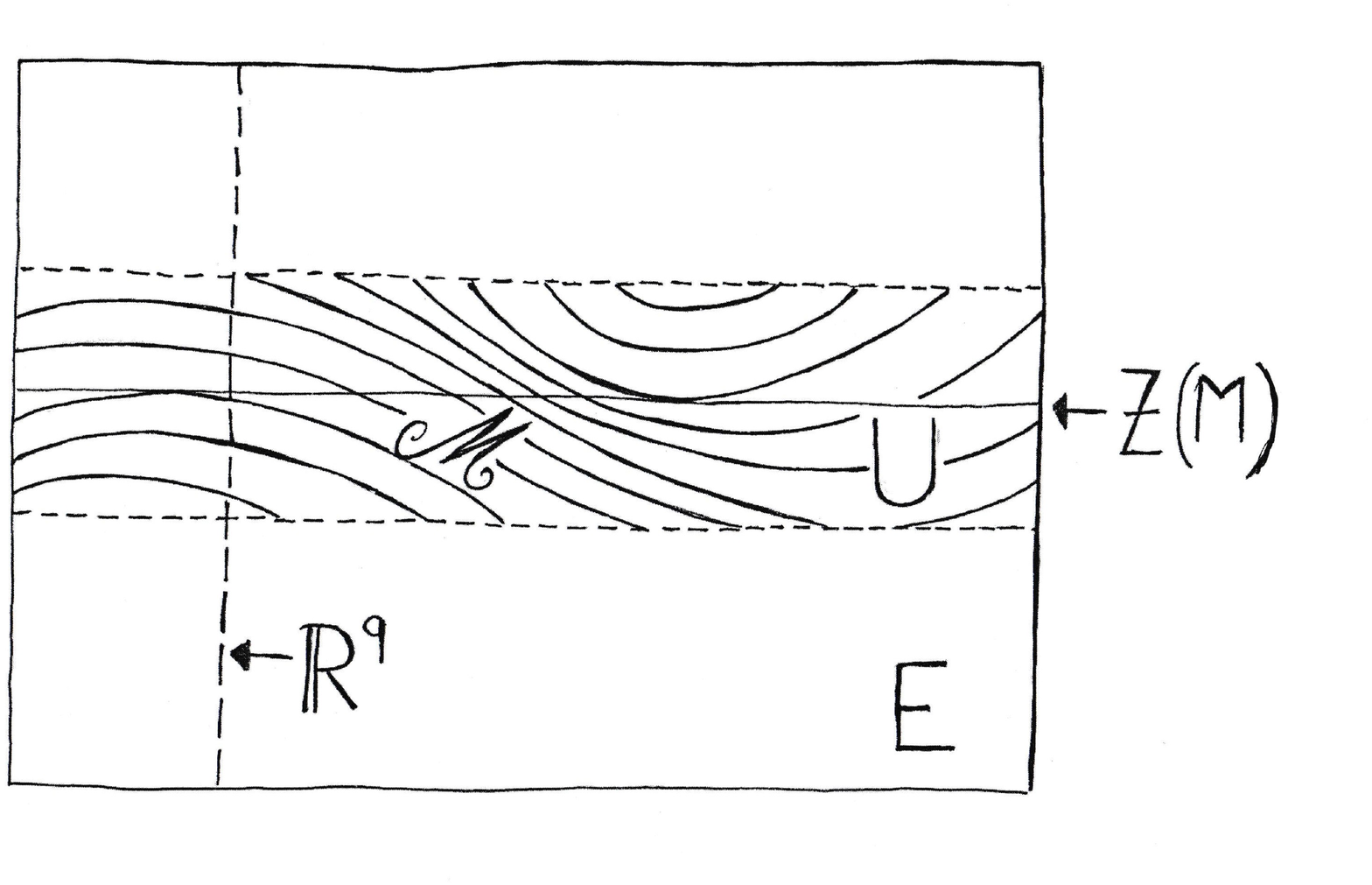}
\caption{A $\Gamma_q$-structure.}
\label{gamma_fig}
\end{figure}
 One regards two $\Gamma_q$-structures 
 as equal if they have the same normal bundle and if their microfoliations
coincide on some open
 neighborhood of the zero section; in other words, $\gamma$ is
 actually the germ of
 $\MM$ along $Z(M)$; we shall also denote $\gamma$ by $[\nu,U,\MM]$.

 A \emph{$\bar\Gamma_q$-structure} means
a $\Gamma_q$-structure whose normal bundle is $M\times\R^q$.

 \subsubsection{Canonical form and differential}\label{differential_sssec}
 Consider a $\Gamma_q$-structure $\gamma=[\nu,U,\MM]$ on a manifold $M$,
where $\nu=(E,\pi,Z)$.

On the manifold $U$, let $\Omega$ be the
differential $1$-form  valued in $\pi^*(\nu)$, defined at every point $v\in U$  as
 the projection of $\tau_vE$ onto $\ker(d_v\pi)=\nu_{\pi(v)}$ parallelly to $\tau_v\MM$.
If one likes better, $\Omega$ is the unique $1$-form defining the
foliation $\MM$ (in the sense that $\tau\MM=\ker\Omega$)
 and whose restriction to every fibre of $\pi$ is the identity.
We call $\Omega$ the \emph{canonical} form of the $\Gamma_q$-structure $\gamma$.
Let us define the \emph{differential}
 $d\gamma$ of the $\Gamma_q$-structure $\gamma$ as $Z^*(\Omega)$: a differential
 $1$-form on $M$ valued in $\nu$.

\begin{note} The notion of differential for a Haefliger structure does not
seem to appear in the literature.
In the case where $d\gamma$ is of rank $q$, of course $d\gamma$ admits
a convex set of left inverse vector
 bundle monomorphisms $\nu\hookrightarrow\tau M$, which are the objects
 that the authors have
considered instead. The differential exists for every Haefliger structure at every point,
 not only at
the regular ones (see subsection \ref{regular_sssec} below). It is functorial with respect to pullbacks (see subsection \ref{pull_sssec} below).
 We feel that speaking in terms of $d\gamma$,
 the analogy between the h-principle for foliations
and several other \emph{avatars} of Gromov's h-principle becomes more transparent.
From our viewpoint, the Foliation theorems \ref{h_thm} above and
\ref{parametric_thm} below deal with homotoping a given differential form of maximal rank
to one which is somehow integrable.
\end{note}

\subsubsection{Pullback}\label{pull_sssec}
 Given a $\Gamma_q$-structure $\gamma=[\nu,U,\MM]$ on $M$, and given
  a smooth mapping $f:N\to M$, one defines a pullback $\Gamma_q$-structure
  $f^*(\gamma)$ on $N$ whose normal bundle is the usual pullback vector bundle
   $f^*(\nu)$, and
  whose microfoliation is the pullback of the foliation $\MM$ under
  the canonical bundle morphism $f^*(\nu)\to\nu$.
  \begin{exm}\label{tautological_exm}
We give, as an example, a useful construction allowing one
 to regard the microfoliation of any Haefliger structure as another Haefliger structure.

 Given $\gamma=[\nu,U,\MM]$ as above, where $\nu=(E,\pi,Z)$,
 consider the vector bundle $\hat\nu:=\pi^*(\nu)$ of rank $q$ over the base $E$; consider
  its total space $$\hat E=E\times_ME=\{(v,w)\in E\times E/\pi(v)=\pi(w)\}$$
 and the exponential map $$\exp:\hat E\to E:(v,w)\mapsto v+w$$
  One has \emph{on the manifold $U$}
 the $\Gamma_q$-structure $\hat\gamma$ whose normal bundle is $\hat\nu\vert U$
  and whose microfoliation is the pullback of the foliation $\MM$ under $\exp$.
  Clearly,
  \begin{equation}\label{tautological_eqn}
   Z^*(\hat\gamma)=\gamma
   \end{equation}
 \end{exm}
  
  \subsubsection{Isomorphisms}\label{isomorphism_sssec}
   J. Milnor's notion of \emph{microbundle} (here smooth)
  is the natural one for the bundle normal to a Haefliger
  structure. For simplicity, one rather speaks in terms of vector bundle; but this underlying fact
  is reflected in the morphisms that one admits. It is enough for us to
  consider isomorphisms.
   \begin{dfn}
 A
 \emph{microisomorphism} between two real vector bundles $\nu=(E,\pi,Z)$, $\nu'=(E',\pi',Z')$
 over a same manifold $M$ is 
  a germ of diffeomorphism $\varphi:E\to E'$ along
  $Z(M)$ such that $Z'=\varphi\circ Z$
  and $\pi=\pi'\circ\varphi$.

   If moreover $\nu'=\nu$ and if the differential
 of $\varphi$ at every point of $Z(M)$ is the identity, we call $\varphi$ a
 \emph{special microautorphism} of $\nu$.
  \end{dfn}

Clearly, the group $Maut(\nu)$ of the microautomorphisms of $\nu$
splits as the semidirect product of the group $Smaut(\nu)$ of the special
microautomorphisms by the group $Aut(\nu)$ of the linear automorphisms.

Note that $Smaut(\nu)$ is convex, hence contractible.

 \begin{dfn}
Given two $\Gamma_q$-structures $\gamma:=[\nu,U,\MM]$
 and  $\gamma':=[\nu',U',\MM']$ on a same manifold $M$,
  an
 \emph{isomorphism} between them
  is a microisomorphism $\varphi:\nu\to\nu'$ such that $\MM=\varphi^*(\MM')$
  (as germs of foliations along $Z(M)$).

  \end{dfn}

  \subsubsection{Concordance}\label{concordance_sssec}
  
  Let $\gamma_0$, $\gamma_1$ be   two $\Gamma_q$-structures  on a same manifold $M$
  which coincide in restriction to some submanifold $N\subset M$ (maybe empty). 
  
 \begin{dfn} A \emph{concordance,} also known as a \emph{homotopy,}
  between $\gamma_0$ and $\gamma_1$ (rel. $N$), is a
   $\Gamma_q$-structure $\gamma$ on $M\times{\I}$
  such that
  \begin{itemize}
  \item $\gamma\vert(M\times i)=\gamma_i$,
  for $i=0,1$;
  \item $\gamma$ coincides with $\pr_1^*(\gamma_0)$ in restriction to $N\times\I$.
  \end{itemize}
  \end{dfn}
  
    \begin{lem}\label{concordance_lem}
  
  If $\gamma_0$ and $\gamma_1$ are
   isomorphic (rel. $N$), then they
  are concordant  (rel. $N$).

     \end{lem}
     
     \begin{proof} Put $\gamma_i=[\nu_i,U_i,\MM_i]$ 
      ($i=0, 1$).
     Let $\varphi$ be a microisomorphism between $\gamma_0$ and $\gamma_1$
     such that $\varphi\vert N=\id$.
     Without loss of generality, shrinking $U_0$ and $U_1$, arrange that $U_1=\varphi(U_0)$.
     
     Endow $\bar U_0:=U_0\times\I$
     with the foliation $\bar\MM_0:=\pr_1^*(\MM_0)$;
     in the disjoint union $\bar U_0\sqcup U_1$, identify each point $(v,1)\in\bar U_0$
     with $\phi(v)\in U_1$. The quotient space is diffeomorphic to an open neighborhood
     of the zero section in a vector bundle over $M\times\I$; and on this quotient space, the image
     of $\bar\MM_0$ is
      the microfoliation of
      a concordance (rel. $N$) between $\gamma_0$ and $\gamma_1$.
 \end{proof}

\begin{pro}[Concordance extension property]\label{cep_pro}
 Given
 a $\Gamma_q$-struc\-ture $\gamma$ on $M$ and a submanifold
$N\subset M$, every concordance
of $\gamma\vert N$ extends to a concordance of $\gamma$.
\end{pro}
 This is obvious
from Haefliger's original viewpoint on the $\Gamma$-structures \cite{haefliger_58},
obvious if one sees $\Gamma_q$-structures as maps to $B\Gamma_q$,
and almost as obvious from the geometric viewpoint adopted here. See also
\cite{mather_73} page 199.

 \subsubsection{Regular Haefliger structures and foliations}\label{regular_sssec}

 The following remarks are intended
  to make clear the relations between two notions: foliations on the one hand,
  regular Haefliger structures on the other hand.
  
\begin{voc}\label{nd_voc}
  Given a codimension-$q$ foliation $\GG$ on a manifold $M$,
  we define as usual the bundle  \emph{normal} to $\GG$ as $\nu\GG:=\tau M/\tau\GG$;
  and we define the \emph{differential} of $\GG$ as the canonical projection
  $d\GG:\tau M\to\nu\GG$.
  \end{voc}
  
  If $\GG$ is moreover complementary to a dimension-$q$ foliation $\FF$ on $M$,
  we often identify $\nu\GG$ with $\tau\FF$; then, $d\GG$ becomes
  a
\emph{linear retraction}
 $\tau M\to\tau\FF$.
 
A $\Gamma_q$-structure $\gamma=[\nu,U,\MM]$ on $M$, where $\nu=(E,\pi,Z)$,
 is called \emph{regular} at a point $x\in M$ if $d\gamma$ is of maximal rank $q$
at $x$; in other words, $Z$ is transverse to $\MM$ at $x$.

\begin{dfn}\label{special_dfn}
 If $\gamma$ is  regular on $M$, one says that $\gamma$ \emph{induces} the
 codimension-$q$ foliation
  $\GG:=Z^*(\MM)$ on $M$.
   If moreover
   \begin{itemize}
   \item The bundle
   $\nu$ normal to $\gamma$ \emph{equals} the bundle $\nu\GG$
   normal to $\GG$;
 \item The differential
  $d\gamma$ \emph{equals}
   the {differential} $d\GG$ of the foliation $\GG$;
  \end{itemize} then
  we say that $\gamma$ \emph{specially} induces $\GG$, and that $\gamma$
  is a \emph{special} regular Haefliger structure.
  \end{dfn}

  \begin{lem}\label{isomorphism_lem}\
  
  a) Every codimension-$q$
   foliation is specially induced by some regular $\Gamma_q$-structure.
 
   b) Two {regular} $\Gamma_q$-structures
    are isomorphic if and only if they induce the same
   foliation.
   
   c) The identity
   is the unique automorphism of a regular $\Gamma_q$-structure.
 
   \end{lem}
   \begin{proof}
      Fix an auxiliary Riemannian metric on $M$;
      denote by $D(x,r)\subset M$ the open ball of center $x$ and radius $r$.

   a)
 Given a codimension-$q$ foliation
  $\GG$ on $M$,
 let  $\exp:\nu\GG\to M$ be the exponential map
$\tau M\to M$ restricted to the subbundle
   $\tau\GG^\bot\cong\nu\GG$. Clearly, on a neighborhood of the zero section,
   $\exp^*(\GG)$
 is the microfoliation of a $\Gamma_q$-structure specially inducing $\GG$.
 
 b)    One can assume that $M$ is compact.
   Let 
   $\gamma:=[E,\pi,Z,U,\MM]$ and $\gamma':=[E',\pi',Z',U',\MM']$
   be two regular $\Gamma_q$-structures on $M$
    inducing
   the same foliation $\GG=Z^*(\MM)=Z'^*(\MM')$.

    Since $\MM'$ is complementary in $U'$ to the fibres of $\pi'$, 
    there is a $r>0$ so small that for every $x\in M$, the projection $\pi'$
   admits a local section
   $\sigma'_x:D(x,r)\to U'$ tangential to $\MM'$ and such that $\sigma'_x(x)=Z'(x)$.
   Clearly, $\sigma'_x(y)$ depends smoothly on the pair $(x,y)$.
   
   On the other hand, 
   since $\gamma$ is regular, after shrinking $U$, one has for every $v\in U$
   a point $Z(\pr(v))$ of the zero section in the leaf
    of $\MM$ through $v$; and one can arrange that $\pr:U\to M$ is smooth, and
    that $\pr\circ Z=\id$. One has
   $\MM=\pr^*(\GG)$. Moreover, after shrinking $U$ again, one can arrange that
   for every $x\in M$, the diameter of 
   $\pi(\pr\mun(x))$ in $M$ is less than $r$. One defines for $v\in U$:
   $$\varphi(v):=\sigma'_{\pr(v)}(\pi(v))$$
   So, $\varphi$ is smooth, and maps
   every leaf of $\MM$ into a leaf of $\MM'$, and
    $\pi'\circ\varphi=\pi$, and $\varphi\circ Z=Z'$.
   
   It remains to verify that $\varphi$ induces a diffeomorphism between $\nb_E(Z(M))$
   and $\nb_{E'}(Z'(M))$.
   To this end, permuting the roles of $\gamma$ and $\gamma'$, one makes a similar map
   $$\varphi':\nb_{E'}(Z'(M))\to\nb_E(Z(M))$$
   and one considers $$\psi:=\varphi'\circ\varphi:\nb_{E}(Z(M))\to\nb_E(Z(M))$$ 
   So,  $\psi$ maps
   every leaf of $\MM$ into a leaf of $\MM$, and $\pi\circ\psi=\pi$, and $\psi\circ Z=Z$.
   We claim that such a map $\psi$ is
   necessarily the identity
   on some neighborhood of $Z(M)$.

    Indeed, shrinking $U$ again, one can arrange that $\psi$ is defined on $U$.
     Since $\gamma$ is regular, shrinking $U$ again, the saturation
     of $Z(M)$ with respect to the foliation $\MM$
     is $U$. Consider any leaf $L$ of $\MM$. In particular, $L$ is connected.
     Since $L$ meets $Z(M)$ and $\psi$ is the identity there, one has $\psi(L)\subset L$.
      Since $\pi\vert L$ is an etale map preserved by $\psi$, the fix point set of $\psi\vert L$
      is open and closed in $L$, hence $\psi\vert L=\id$.
   The claim
    is proved.

   c) is a particular case of the claim.
      \end{proof}

In other words: (i) A codimension-$q$ foliation on $M$ 
amounts to an \emph{isomorphism class} of regular $\Gamma_q$-structures on $M$;

(ii) For every codimension-$q$ foliation $\GG$ on $M$, 
  the  {regular}
  $\Gamma_q$-struc\-tures on $M$ \emph{specially} inducing $\GG$
form a principal space under the group $Smaut(\nu\GG)$.

\begin{ccl}\label{induction_ccl}

a)
In view of Lemma \ref{concordance_lem} and of  (i) above,
  one can speak of the concordance class
of a codimension-$q$
 foliation $\GG$, defined as
 the concordance class
of any $\Gamma_q$-structure 
 inducing $\GG$.

b)
In view of (ii) above
 and $Smaut(\nu\GG)$ being contractible,
 there is no inconvenience
  in identifying a codimension-$q$ foliation $\GG$ with any regular
 $\Gamma_q$-structure
\emph{specially} inducing $\GG$.
 \end{ccl}

  \subsubsection{The parametric Foliation theorem on open manifolds}
\label{parametric_sssec}
The classical Foliation theorem on open manifolds admits
 two parametric
versions, that we respectively call ``nonintegrable'' and ``integrable''.
The second will be used repeatedly in the proof of Theorem \emph{A'} below.
It does not seem to appear explicitly in the literature,
although all proofs of the nonintegrable version
(for example the one in \cite{laudenbach_meigniez_16}) actually prove the
integrable one as well.

 The space of parameters will be a compact manifold ${A}$; fix 
 a compact submanifold $B\subset A$, maybe empty. Consider, over
 a manifold $M$,
 a real vector bundle
$\nu=(E,\pi,Z)$ of dimension $q$; and   its pullback $\tilde\nu$ over $M\times A$.

\begin{dfn} By a \emph{family}
 of Haefliger structures $(\gamma(a))_{a\in{A}}$ whose normal bundle is $\nu$, one means,
  for every
$a\in{A}$, a $\Gamma_q$-structure $\gamma_a$ on $M$ whose normal bundle is $\nu$;
denote its microfoliation by $\MM(a)$;
such that the (germ of)
plane field $\tau\MM(a)$ on $\nb_E(Z(M))$
 depends smoothly on $a$.

Call the family $(\gamma(a))_{a\in{A}}$ \emph{integrable}
(with respect to the parameter)
if moreover,
 there is a global $\Gamma_q$-structure $\tilde\gamma$ on $M\times{A}$
whose normal bundle is $\tilde\nu$ and whose restriction
to the slice $M\times a$ is $\gamma(a)$,
for every $a\in{A}$.
\end{dfn}

Fix a compact submanifold $N\subset M$ (maybe empty)
such that the pair $(M,N)$ is \emph{open,} in the sense,
classical in the h-principle,
 that
 every connected component of the complement $M\setminus N$
which is relatively compact in $M$
meets $\partial M$.
Assume that $M$ carries
a (smooth) parametric family $(\omega(a))_{a\in{A}}$ of $\nu$-valued differential $1$-forms
 of constant rank $q$ such
 that $d\gamma(a)=\omega(a)$ holds on $\nb_M(N)$ for every $a\in{A}$,
and on $M$ for every $a\in\nb_{{A}}({B})$.
 Consider the projection $\pr_1:(x,t)\mapsto x$
and the embedding $\iota_t:x\mapsto(x,t)$ ($x\in M$, $t\in{\I}$).

\begin{thm}\label{parametric_thm} Let  $(\gamma(a))_{a\in{A}}$
be a
 family  of $\Gamma_q$-structures on $M$.
 
 Then, under the above hypotheses,

(i) There is a smooth
family
$(\bar\gamma(a))_{a\in{A}}$ of $\Gamma_q$-structures on $M\times{\I}$
such that for every $a\in{A}$:
\begin{enumerate}
\item $\iota_0^*(\bar\gamma(a))=\gamma(a)$;
\item $\bar\gamma(a)=\pr_1^*(\gamma(a))$ on $\nb_M(N)\times{\I}$;
\item If $a\in{B}$, then $\bar\gamma(a)=\pr_1^*(\gamma(a))$
 on $M\times{\I}$;
\item $\gamma'(a):=\iota_1^*(\bar\gamma(a))$ is regular.
\end{enumerate}

(ii) Moreover, the family $(d\gamma'(a))_{a\in{A}}$ is homotopic on $M$,
 relatively to $(N\times A)\cup(M\times B)$,
 to
the family $(\omega(a))_{a\in{A}}$ among the families of $\nu$-valued
$1$-forms of constant rank $q$.

(iii) If the family $(\gamma(a))_{a\in{A}}$ is moreover integrable, then one
can choose the family $(\bar\gamma(a))_{a\in{A}}$ to be also integrable.
\end{thm}

\begin{proof}[Proof of Theorem \ref{parametric_thm}]
Haefliger's original
proof of the nonparametric Foliation theorem on open manifolds
\cite{haefliger_69}\cite{haefliger_70} is a direct application of the Gromov-Phillips
transversality theorem in the frame of  $\Gamma$-structures:
the transversality theorem
 is applied, in the total space of
the normal bundle of the structure, to the zero section,
and provides a homotopy that puts it transverse to the microfoliation.
This argument
goes with parameters and thus
proves both parametric versions, the nonintegrable and the integrable.
Here are the details, for the sake of completeness.

Choose a smooth plateau function $\chi$ on $M\times{A}$ such that
\begin{itemize}
\item $\chi=1$ on
a neighborhood of $P:=(N\times{A})\cup(M\times{B})$;
\item
 $\omega(a)_x=d\gamma(a)_x$
 for every $(x,a)$ in the support $\supp(\chi)$.
\end{itemize}

 Let $V\subset E$ be an open neighborhood of the zero section $Z(M)$,
so small that $\MM(\gamma(a))$ is defined on $V$ for every parameter $a$,
and that the
plane field $\tau\MM(a)$ on $V$
 depends smoothly on $a$.
Let $\Omega(a)$ be on $V$ the canonical, $\pi^*(\nu)$-valued
$1$-form defining $\MM(a)$ (recall Paragraph \ref{differential_sssec} above).

For a fixed $a\in{A}$, consider, over the map $Z:M\to E$, the
 bundle morphism $$\zeta(a):\tau M\to\tau E$$
defined for every $x\in M$ and $u\in\tau_xM$ as
 $$\zeta(a)_xu:=\chi(x,a)u\oplus(1-\chi(x,a))\omega(a)_xu$$
(where the tangent space $\tau_{Z(x)}E$ is decomposed as $\tau_xM\oplus\nu_x$).
In view of the choice of $\chi$ and of the definitions of $d\gamma$ and $\zeta$,
\begin{equation}\label{zeta_eqn}
\Omega(a)_{Z(x)}\circ\zeta(a)_x=\omega(a)_x
\end{equation}
Hence, $\zeta(a)_x$ is transverse to $\tau_{Z(x)}\MM(a)$ in $\tau_{Z(x)}E$.
By the Gromov-Phillips transversality theorem 
(\cite{gromov_69}, \cite{phillips_70}, \cite{gromov_86} pp. 53, 102) (which,
if one likes, one can today obtain as an immediate application of the Eliashberg-Mishachev
Holonomic Approximation theorem \cite{eliashberg_mishachev_02}), 
  one has a map
$$H:M\times{A}\times{\I}\to V$$
and, over $H$, a homotopy of parametric families of
 bundles morphisms $$\eta:\tau M\times{A}\times{\I}\to\tau E$$ 
such that for every $x\in M$, $u\in\tau_xM$, $a\in{A}$, $t\in{\I}$:
\begin{enumerate}[label=(\Roman*)]
\item $H(x,a,t)=Z(x)$ if $(x,a)\in P$
or $t=0$;
\item $\eta(u,a,t)=\zeta(a)_xu$ if $(x,a)\in P$
or $t=0$;
\item
The map $u\mapsto \eta(u,a,t)$ maps linearly $\tau_xM$ into $\tau_{H(x,a,t)}E$
 transversely to $\tau_{H(x,a,t)}\MM(a)$;
\item $\eta(u,a,1)=(\partial H/\partial x)(x,a,1)u$.
\end{enumerate}

For every parameter $a$, define the $\Gamma_q$-structure
 $\bar\gamma(a)$ on $M\times{\I}$ as the pullback of $\hat\gamma(a)$
 (recall Example \ref{tautological_exm}) through the map
$(x,t)\mapsto H(x,a,t)$. The properties (1), (2), (3) of Theorem \ref{parametric_thm}
 follow from (I) above and from Equation (\ref{tautological_eqn}) in
 Example \ref{tautological_exm}; while (4) follows from (III) and (IV).

 (ii): Consider on $M$ the $1$-parameter
family of ${A}$-parametrized families of rank-$q$, $\nu$-valued $1$-forms:
 $$\varpi(a,t)_xu:=\Omega_{H(x,a,t)}\eta(u,a,t)$$
(where $a\in{A}$, $t\in{\I}$, $x\in M$, $u\in\tau_xM$).
By Equation (\ref{zeta_eqn}) above and by (I) and (II), the identity
 $\varpi(a,t)_x=\omega(a)_x$ is met for $(x,a)\in P$ or $t=0$.
On the other hand, by the definitions of $\Omega$ and $\hat\gamma$ and by
 (IV), one has $\varpi(a,1)=d\gamma'(a)$.

(iii): Assume moreover that
 every $\gamma(a)$ is the restriction to the slice $M\times a$
 of a global $\Gamma_q$-structure
$\tilde\gamma$ on $M\times{A}$. Then, every $\bar\gamma(a)$ is the restriction
 to the slice $M\times a\times{\I}$ of
the global $\Gamma_q$-structure
 on $M\times{A}\times{\I}$ that is the pullback of $\hat{\tilde\gamma}$
through the map $$H\times\pr_2:M\times{A}\times{\I}\to V\times{A}$$
\end{proof}

\section{Cleft Haefliger structures and cleft foliations}\label{holes_sec}
We shall actually prove first a form \emph{A'} of Theorem \emph{A} dealing
with ``fissures'' instead of multifold Reeb components; this form
goes as well for $q=1$. Our fissures are some kind of discontinuities in a foliation
which in the compact case amount to the same as the classical
{``holes'',} ubiquitous
in the theory of foliations since Thurston's works
\cite{thurston_74}\cite{thurston_76}. The advantage of fissures over holes
 is that fissures are more functorial.
 
 Consider a compact manifold $Q$ of dimension $q\ge 1$
  with smooth boundary (we are essentially interested
in the case $Q=\D^q$).

\begin{ntt}\label{diff_ntt}
One denotes by $\diff(Q)$
the group of the diffeomorphisms of $Q$ \emph{whose support is contained in the interior
of $Q$.}

One denotes  by
 $\diff(Q)^\I$ the topological group (for the smooth topology)
  of the smooth pointed paths in $\diff(Q)$: in other words, the families
 of diffeomorphisms
$\varphi:=(\varphi_t)_{t\in\I}$ such that
\begin{itemize}
\item $\varphi_t\in\diff(Q)$;
\item $\varphi_0=\id$;
\item
 $\varphi_t(y)$ depends smoothly on the pair $(t,y)\in\I\times Q$.
 \end{itemize}
It is convenient to add
 that
 \begin{itemize}
 \item $\varphi_t$ coincides with $\id$ (resp. $\varphi_1$)
for $t$ close to $0$ (resp. $1$).
\end{itemize}
\end{ntt}

Then, one considers as usual on $\S^1\times Q$
 the $1$-dimensional foliation $\SS_\varphi$, called the \emph{suspention} of $\varphi$,
 spanned by the vector field $(\partial/\partial t,\partial\varphi/\partial t)$
  (here $\S^1$ is the quotient of the interval $\I$ by $0=1$).

\medbreak

 Fix a 
$\varphi\in\diff(\D^q)^\I$. Denote by $\theta$ the angle coordinate
$z\mapsto z/\vert z\vert$ on $\C^*$.
The \emph{model cleft of monodromy $\varphi$}
 is the manifold $\C\times\R^q$
endowed on the complement of $0\times\D^q$
 with the codimension-$q$
 foliation $\CC_\varphi$ equal to
 $(\theta\times\id_{\D^q})^*(\SS_\varphi)$
 on $\C^*\times\D^q$ and horizontal (parallel to $\C$)
 on $\C\times(\R^q\setminus\D^q)$.
 
 Let $M$ be a manifold of dimension $\ge 0$.
  \begin{dfn}\label{fissure_dfn}
A ($q$-) \emph{fissure} (or cleft) in $M$
is a pair $(C,[c])$ where
\begin{enumerate}
\item $C\subset M$ is a proper (i.e. topologically closed) codimension-$2$ submanifold, which may have a boundary;
\item $[c]$ is the germ along $C$ of
a submersion $$c:\nb_M(C)\to\C\times \R^q$$
 such that $C=c\mun(0\times\D^q)$;
 \item  At every point of $C\cap\partial M$, the rank of $c\vert\partial M$
  is also $q+2$.
 
 \end{enumerate}
  \end{dfn}
 Notes:
 
 --- The image $c(C)$ is not necessarily the all of $0\times\D^q$;
 
 ---  We admit the case where $\dim(M)<2$, in which case $C$ is necessarily empty;
   
---   Sometimes,
 we regard the connected components of the fissure as several distinct fissures;
   
---   After (3), if $C$ intersects $\partial M$, then $(C,[c])$
 induces in $\partial M$
  a fissure
  $(C\cap\partial M,[c\vert\partial M])$.

 \begin{dfn}\label{holed_dfn}
A \emph{cleft $\Gamma_q$-structure of monodromy $\varphi$ and normal
bundle $\nu$} on $M$
is a triple $\Gamma:=(C,[c],\gamma)$ such that
\begin{itemize}
\item
 $(C,[c])$ is a fissure in $M$;
 \item $\nu=(\pr_2\circ c)^*(\tau\R^q)$ on $\nb_M(C)$;
\item $\gamma$ is   on
$M\setminus C$ a $\Gamma_q$-structure of normal bundle $\nu\vert(M\setminus C)$;
\item  $\gamma$ specially induces $c^*(\CC_\varphi)$ on $\nb_M(C)\setminus C$.
\end{itemize}
\end{dfn}

(recall Definition \ref{special_dfn}). The fissure may be empty; and must be empty if
$\dim(M)<2$.

 In case $M$ has a boundary, $\Gamma$
 induces on $\partial M$
  a cleft $\Gamma_q$-structure
  $$\Gamma\vert\partial M:=(C\cap\partial M,[c\vert\partial M],\gamma\vert\partial M)$$

\begin{dfn}\label{holed_foliation_dfn} 
Let $\Gamma:=(C,[c],\gamma)$ be a cleft $\Gamma_q$-structure on $M$.
If moreover $\gamma$ is regular on $M\setminus C$, and \emph{specially} induces a foliation
$\GG$ there, then one calls
 $\Gamma$, or  (recalling Conclusion \ref{induction_ccl} (b))
 the triple $G:=(C,[c],\GG)$,
 a \emph{cleft foliation}.
 \end{dfn}
 
 The two following notions
 will be crucial in the proof of Theorem \emph{A'}. Let $G$ be a cleft foliation
 on a manifold $M$.

\begin{dfn}\label{tangent1_dfn}
One calls a vector field $\nabla$ on $M$ \emph{tangential} to 
 $G$ if
 \begin{itemize}
 \item $\nabla(x)\in\tau_x\GG$
  at every point $x\in M\setminus C$;
 \item $\nabla(x)\in\ker(d_xc)$
  at every $x\in\nb_M(C)$.
 \end{itemize}
 \end{dfn}
 
 \begin{dfn}\label{tangent2_dfn}
One calls a foliation $\XX$ on $M$ \emph{tangential} to 
 $G$ if
 \begin{itemize}
 \item $\tau_x\XX\subset\tau_x\GG$
  at every point $x\in M\setminus C$;
 \item  $\tau_x\XX\subset\ker(d_xc)$
  at every $x\in\nb_M(C)$.
 \end{itemize}
 \end{dfn}

Now, let $\FF$ be a dimension-$q$ foliation on $M$.

\begin{dfn}\label{position_dfn} The $q$-fissure $(C,[c])$ in $M$
 is
 \emph{in standard position} with respect to $\FF$ if for
 every leaf $L$ of $\FF\vert\nb_M(C)$ there is a $z\in\C$ such that
   $c$ induces a rank-$q$ map from $L$
   into the slice $z\times\R^q$.
 
 In particular, $C$ is then tangential to $\FF$, in the sense that $\tau_x\FF\subset\tau_x C$
  at every $x\in C$. One thus gets a dimension-$q$ foliation $\FF_C$ on $C$.

\end{dfn}

\begin{dfn}\label{qc_dfn}
The cleft foliation $G:=(C,[c],\GG)$ on $M$ is
 \emph{quasi-compl\-em\-ent\-ary} to $\FF$ if
   $(C,[c])$ is in standard position with respect to $\FF$, and if
  $\GG$ is complementary to $\FF$ on $M\setminus
C$.
\end{dfn}

Recall that $\bar M$ denotes $M\times\I$ and that $\hat M\subset\bar M$
is the union of $M\times 0$ with $\partial M\times\I$; also, $\pr_1$
denotes the first projection $\bar M\to M$.

 \begin{TAP} Fix a $\varphi\in\diff(\D^q)^\I$
such that $\varphi_1\neq\id$.
 On a compact manifold $M$, let
 $\FF$ be a  foliation
 of dimension $q\ge 1$,
 and $\gamma$ be
 a $\Gamma_q$-structure whose normal bundle is
 ${\tau\FF}$.  Let $\bar\FF$ 
 denote the dimension-$q$ foliation on $\bar M$
parallel to $\FF$.
 
 Assume that over  $\nb_M(\partial M)$, 
 the differential  $d\gamma$ is a linear retraction $\tau M\to\tau\FF$.
  Then,

  (i)
  $\bar M$ admits a  cleft
 $\Gamma_q$-structure $\Gamma:=(C,[c],\bar\gamma)$
  of  normal bundle ${\tau\bar\FF}$ and of monodromy $\varphi$
  such that

  \begin{itemize}
 
  \item $(C,[c])$ is disjoint from $\hat M$ and
   in standard position w. r. t. $\bar\FF$;
   
     \item $\bar\gamma=\pr_1^*(\gamma)$ on $\nb_{\bar M}(\hat M)$;

  \item $\Gamma\vert(M\times 1)$ is a cleft foliation quasi-complementary to $\FF$.
 
\end{itemize}

(ii) Moreover, in the case where $M$ is a product $V\times X$ and where
 $\FF$ is its slice foliation parallel to $X$, one can arrange that the projection
 $V\times X\times\I\to V\times\I$
   is one-to-one in restriction to each connected component of $c\mun(0,0)$.

 \end{TAP}

\begin{note}\label{retraction_note} The hypothesis that 
  over  $\nb_M(\partial M)$, 
   the differential  $d\gamma$ is a linear retraction $\tau M\to\tau\FF$,
   amounts to say that the $\Gamma_q$-structure $\gamma$
 is
  regular there, complementary to $\FF$, and special
    (recall  the note after Vocabulary \ref{nd_voc},
  and Definition
  \ref{special_dfn}).
 \end{note}
      
 The proof of Theorem \emph{A'} will be given in Section \ref{proof_sec} below. 
 The proof of Theorem \emph{A} as a corollary of
 \emph{A'} will be given in Subsection \ref{AAP_ssec} below.
 The proof of the Mather-Thurston Theorem \ref{MT_thm} as a corollary of \emph{A'}
 will be given in Section
\ref{MT_sec} below.

\subsection{Toolbox} We shall use the following tools, among which
 (\ref{filling_sssec}) is classical;
the others are more or less obvious.

  \subsubsection{Structure of compact fissures}
   
 Let $(C,[c])$ be a $q$-fissure
in a manifold $M$, and $\Sigma:=c\mun(0,0)$.
Note that if $M$ is compact, then $C$ is compact.

\begin{lem}\label{complete_lem} Assume
that $C$ is compact. Then:

(a) There are a small $\epsilon>0$ and an equidimensional embedding
 $$h:\Sigma\times\D^2\times\D^q\hookrightarrow M$$
whose image $H$ is a thin compact neighborhood of $C$ in $M$,
and such that $$c\circ h:(\sigma,z,y)\mapsto(\epsilon z,
(1+\epsilon)y)$$

(b) If  $(C,[c])$ is moreover in standard position (Definition \ref{position_dfn})
with respect
to a dimension-$q$ foliation $\FF$, then one can moreover arrange that
 $h^*(\FF)$ is the slice foliation parallel
  to $\D^q$.
\end{lem}

 After the lemma, the compact fissure $C$ is diffeomorphic with $\Sigma\times\D^q$.
 We call $\Sigma$ the \emph{core} of the fissure.

\begin{proof} (a) Since $C$ is compact and $\D^q$ is compact connected,
 the submersion $c$ induces a proper map from $\nb_M(C)$
onto 
 $\nb_{\C\times\R^q}(0\times\D^q)$.
For $\epsilon>0$ small enough,
after Ehresmann's lemma, 
 $c$ induces a locally trivial fibration of
some neighborhood $H$ of $C$ in $M$ over the bidisk
 $$\{(z,y)\in\C\times\R^q/\vert z\vert\le\epsilon, \vert y\vert\le 1+\epsilon\}$$
Of course, the bidisk being contractible, the fibration is actually trivial.

(b) If 
$(C,[c])$ is moreover in standard position
with respect
to a dimension-$q$ foliation $\FF$,  consider the (trivial) fibration
$c\vert C$ of $C$ over $\D^q$.
The leaves of $\FF_C$ (Definition \ref{position_dfn}), being complementary in $C$ to the
fibres of this fibration, are mapped diffeomorphically onto the simply connected base.
 By the Reeb local stability theorem (\cite{camacho_neto_85} ch. IV parag. 4),
shrinking $H$ if necessary, all leaves of $\FF\vert H$ are compact $q$-disks, and
they are the fibres of
a fibration $p:H\to\Sigma\times\D^2$. The wanted embedding $h$ is
the inverse of the diffeomorphism
 $$p\times(\pr_2\circ c):H\to\Sigma\times
\D^2\times\D^q$$
\end{proof}

\subsubsection{Horizontal perturbation
 of a compact fissure}\label{reparametrizing_sssec}
 Let $(C,[c])$ be a compact
$q$-fissure in a manifold $M$.

Choose a neighborhood $H=h(\Sigma\times\D^2\times\D^q)$ of $C$ in $M$ as in Lemma \ref{complete_lem} (a).
 For
 any self-diffeomorphism $f\in\diff(\Sigma\times\D^2)$ (Notation \ref{diff_ntt} above),
consider the self-diffeomorphism $\bar f$ of $H$ defined by
$$\bar f(h(\sigma,z,y)):=h(f(\sigma,z),y)$$
Then, $(\bar f\mun(C),[c\circ \bar f])$ is a new $q$-fissure in $M$.

a) If moreover $(C,[c])$ is in standard position with respect to a foliation
$\FF$ on $M$, then after  Lemma \ref{complete_lem} (b), for a proper choice of
the embedding $h$, for every $f\in\diff(\Sigma\times\D^2)$,
 the perturbed fissure $(\bar f\mun(C),[c\circ\bar f])$  is also
 in standard position with respect  to
$\FF$.

b) Given  a
cleft $\Gamma_q$-structure $\Gamma:=(C,[c],\gamma)$ on $M$,
provided that $H$ is thin enough, $\bar f$
preserves the germ of $c^*(\gamma)$ along $\partial H$.
We thus get on $M$ a \emph{perturbed} cleft $\Gamma_q$-structure
$$\Gamma_{h,f}:=(\bar f\mun(C),[c\circ\bar f],\gamma_{h,f})$$
 where
  $\gamma_{h,f}$ is the $\Gamma_q$-structure 
  on $M\setminus\bar f\mun(C)$ defined on
   $H\setminus\bar f\mun(C)$ by
    $\gamma':=\bar f^*(\gamma)$,
 and on $M\setminus Int(H)$ by $\gamma':=\gamma$.

c) If moreover $\Gamma$ is a cleft foliation quasi-complementary to
$\FF$, then after  Lemma \ref{complete_lem} (b), for a proper choice of
 $h$,  for every $f\in\diff(\Sigma\times\D^2)$,
  the perturbed cleft foliation $\Gamma_{h,f}$ is also
 quasi-complementary to
$\FF$.

 \subsubsection{Pulling back a fissure, a cleft Haefliger structure or a cleft foliation}
 \label{pullback_sssec}
  Let  $u:M'\to M$ be a smooth map between two manifolds, and 
  $(C,[c])$ be a $q$-fissure in $M$.

\begin{dfn}\label{pullable_dfn} The fissure $(C,[c])$ is \emph{pullable} through $u$
if at every point of $u\mun(C)$, the map $c\circ u$ is of maximal rank $q+2$.
\end{dfn}

Then,  
one gets in $M'$ a \emph{pullback}
 fissure $(u\mun(C),[c\circ u])$.
  In case $u$ is an inclusion $M'\subset M$, we say ``restrictable'' instead of ``pullable'' and
 ``restriction'' instead of ``pullback''. The following remarks (a) -- (f) will be useful.
 
 \medbreak
 
a) If moreover
\begin{enumerate}
\item One has on $M$ a dimension-$q$ foliation $\FF$ and on $M'$ a
dimension-$q$  foliation $\FF'$;
\item $u$ is \emph{leafwise etale} with respect to
$\FF'$ and $\FF$, in the sense that at every $x\in M'$, the differential $d_xu$ maps isomorphically $\tau_x\FF'$
 onto $\tau_{u(x)}
\FF$;
\item The fissure $(C,[c])$ is in standard position with
respect
to $\FF$ (Definition \ref{position_dfn});
\end{enumerate}
 then obviously the pullback fissure $(u\mun(C),[c\circ u])$ is in standard position with respect
to $\FF'$.

b) The pullability condition is actually generic in this frame, in the following sense.
  
Under the hypotheses of Definition \ref{pullable_dfn} and (1), (2), (3) above,
assume moreover that $C$ is compact.
Choose a neighborhood $H=h(\Sigma\times\D^2\times\D^q)$ of $C$
in $M$ as in Lemma \ref{complete_lem} (a) and (b).
\begin{lem}\label{generic_lem}
 Then, for $f\in\diff(\Sigma\times\D^2)$ generic,
  the perturbed fissure
  $(\bar f\mun(C),[c\circ\bar f])$
 (Tool \ref{reparametrizing_sssec} above) is pullable through $u$.

\end{lem}

\begin{proof} Application of the Thom transversality theorem
under, if one likes, Poenaru's version ``with constraints''
 (see \cite{laudenbach_11}
paragraph 5.4).

\end{proof}

c) Under the hypothesis \ref{pullable_dfn}, any cleft $\Gamma_q$-structure
 $\Gamma:=(C,[c], \gamma)$
 on $M$ admits a \emph{pullback}
 cleft $\Gamma_q$-structure
  of the same monodromy $$u^*(\Gamma):=(u\mun(C),[c\circ u],u^*(\gamma))$$
  In case $u$ is an inclusion $M'\subset M$, 
we rewrite $u^*(\Gamma)$ as $\Gamma\vert M'$.

d)
  Let $\nabla$ be on $M$ a vector field tangential to a cleft foliation $G$
  (Definition \ref{tangent1_dfn}).

 Then, the (maybe
 only partially defined) flow of $\nabla$ preserves $G$, in the following sense:
  for every $t\in\R$, $x\in Int(M)$, if $\nabla^t(x)$ is defined and also lies in $Int(M)$, then
 $G=(\nabla^t)^*(G)$ on $\nb_M(x)$.

e) Under the hypothesis \ref{pullable_dfn} and (1), (2) above,
the pullback through $u$ of any cleft
foliation quasi-complementary to $\FF$ on $M$ is a 
 cleft
foliation quasi-complementary to $\FF'$ on $M'$.

\subsubsection{Holes}\label{fh_sssec}

\begin{dfn}\label{hole_dfn} For a compact
manifold $\Sigma$, a compact manifold $Q$ of dimension $q\ge 1$
and an element 
$\varphi\in\diff(Q)^\I$,
 the \emph{hole} of core $\Sigma$, fibre $Q$ and monodromy
$\varphi$ consists of
 the manifold $H:=\Sigma\times\D^2\times Q$ together
with, along $\partial H$, the germ of a codimension-$q$ foliation $\HH_\varphi$
 such that
\begin{itemize}
\item $\HH_\varphi$
 is the slice foliation parallel to $\D^2$ 
   on some neighborhood of
 $\Sigma\times\D^2\times\partial Q$;
\item $\HH_\varphi$ restricts, on $\Sigma\times\partial\D^2\times Q$,
 to the preimage of the suspension $\SS_\varphi$ under the projection
$(\sigma,z,y)\mapsto(z,y)$.
 \end{itemize}

 \end{dfn}
 Let $(C,[c],\gamma)$ be a
cleft $\Gamma_q$-structure  of monodromy $\varphi$ on a manifold $M$,
 whose fissure is compact.
Choose
 a neighborhood $H$ as in Lemma \ref{complete_lem}. One can then
 forget
 $\gamma$ in the interior of $H$, remembering
 only its germ along $\partial H$. One has thus \emph{enlarged} the fissure into a hole
 of the same core, of fibre $\D^q$ and of the same monodromy.

 Conversely, 
  any hole $(\Sigma\times\D^2\times\D^q,\HH_\varphi)$ of fibre $\D^q$
  can be \emph{horizontally shrunk}
 into a fissure, by extending radially the germ of foliation $\HH_\varphi$ from the vertical
 boundary $\Sigma\times\partial\D^2\times\D^q$  and thus
 foliating the subset
  $\Sigma\times(\D^2\setminus 0)\times\D^q$.

\subsubsection{Filling a hole whose monodromy is a multirotation}\label{filling_sssec}
 (After \cite{thurston_74}, section 4).

\begin{dfn}\label{multirotation_dfn} Let $Q$ be
a compact manifold of dimension $q\ge 2$. One calls
$\varphi:=(\varphi_t)_{t\in\I}\in\diff(Q)^\I$
 a \emph{multirotation} if
 there are
 a function $u$ on $\I$ as in the subsection \ref{qc_ssec} and
 an equidimensional
 embedding $$F:\D^{q-1}\times\S^1\hookrightarrow Int(Q)$$ such that
for every $t\in\I$:
\begin{itemize}   \item $\varphi_t(F(a,s))=F(a,s+tu(1-2\vert a\vert))$
      for every $a\in\D^{q-1}$ such that $\vert a\vert\le 1/2$ and every $s\in\S^1\cong\R/\Z$;
\item
   $\varphi_t$ is the identity on the rest of $Q$.
\end{itemize}
\end{dfn}
  
     For any (compact)
      manifold $\Sigma$, on $H:=\Sigma\times\D^2\times Q$, denote by $\FF$
     (resp. $\HH_0$)  the slice foliation parallel to $Q$ (resp. to $\Sigma\times\D^2$).
    Assuming that $\varphi$ is the above multirotation,
 consider the foliation $\HH$ of
  $\Sigma\times\D^2\times Q$ equal to the foliation
     $\GG_\Sigma$ of the subsection \ref{qc_ssec}
      in $\Sigma\times(\D^2\times Im(F))\cong C_\Sigma$ (see Figure \ref{movie_fig}),
     and equal to $\HH_0$ in the complement.
     The germ of $\HH$ along $\partial H$
      matches Definition \ref{hole_dfn}. One has thus  \emph{filled,}
       by means of the foliation $\HH$,
 the  hole of  core $\Sigma$, fibre $Q$ and monodromy $\varphi$.
       
\begin{rmk}\label{differential_rmk} (on the differential of the filling foliation)
Recall the definitions (Vocabulary \ref{nd_voc})
 of the normal bundle $\nu\HH$
and of the differential $d\HH$ for the foliation
 $\HH$.
The $1$-parameter family $(\omega_r)$ of the  subsection \ref{qc_ssec}
 provides a homotopy between
 the codimension-$q$
  plane fields
   $\tau\HH_0$ and $\tau\HH$  on $H$, rel. $\partial H$.
   Hence, the family $(\omega_r)$ also provides, uniquely up to a homotopy,
   a linear isomorphism of  vector bundles
    $$h:\nu\HH\to\nu\HH_0=\tau\FF$$
    which is the identity
  on $\partial H$, and such that
    the image $h\circ d\HH$ of  the differential $d\HH$
 is homotopic to $d\HH_0$ among the $(\tau\FF)$-valued
  $1$-forms of rank $q$ on
 $H$, rel. $\partial H$.
\end{rmk}

\subsubsection{Vertical reparametrization of a hole}\label{reparametrization_sssec}

\begin{dfn}\label{vi_dfn}
 Given a locally trivial fibration between some manifolds $\pi:E\to B$ and
  a $1$-parameter family $(f_s)_{s\in\I}$ of self-diffeomorphisms
  of $E$ such that $f_0=\id$ and that $\pi\circ f_s=\pi$ for every $s$,
  one calls the family $(f_s)_{s\in\I}$ (or $f_1$ alone)  an isotopy \emph{vertical}
  with respect to $\pi$.
  \end{dfn}
  The vertical isotopies enjoy an obvious extension property.

  Now, let $Q$ be a compact $q$-manifold; consider the identity component
  $\diff(Q)_0$ and its universal cover $\widetilde\diff(Q)_0$,
   which can be realized as
the quotient of $\diff(Q)^\I$ by the
 relation of homotopy rel. $\partial\I$. We denote by $[\varphi]$
 the class of the $1$-parameter family of diffeomorphisms $\varphi$.
 The projection $\widetilde\diff(Q)_0\to\diff(Q)_0$ is 
 $[\varphi]\mapsto\varphi_1$.
 
  The following is classical; we recall the short proof for completeness.
  
  \begin{lem}\label{suspension_lem}
  The suspension foliation $\SS_\varphi$ only depends, up to
  an isotopy in $\S^1\times Q$ vertical with respect to the projection to $\S^1$
  and which is the identity close to $\S^1\times\partial Q$,
   on the conjugation class of
  $[\varphi]$ in the group $\widetilde\diff(Q)_0$.
  \end{lem}
  \begin{proof} (Here $\S^1$ is the quotient of the interval $\I$ by the relation $0=1$.)
  
  One the one hand, a homotopy of the family $(\varphi_t)_{t\in\I}$
  rel. $\partial\I$ obviously amounts to a vertical  isotopy of $\SS_\varphi$ rel.
  $0\times Q$.

On the other hand, let $\varphi'=\psi\varphi\psi\mun$ be a conjugation 
 in the group $\diff(Q)^\I$. Put
 for $s\in\I$:
$$ 
 f_s:\S^1\times Q\to\S^1\times Q:
 (t,y)\mapsto(t,\varphi'_{st}(\psi_s(\varphi_{st}\mun(y))))
$$
  This family $(f_s)$
   is a vertical isotopy of $\S^1\times Q$,
  and $f_1^*(\SS_{\varphi'})=\SS_{\varphi}$.
  \end{proof}
  It follows that the notion of hole of monodromy $\varphi$
does  actually not depend on $\varphi$ itself,
but only on the conjugation class of $[\varphi]$ in
  $\widetilde\diff(Q)_0$. Indeed, any isotopy in $\Sigma\times\S^1\times Q$
  vertical with respect to the projection to $\Sigma\times\S^1$, and which
  is the identity close to  $\Sigma\times\S^1\times\partial Q$,
  extends into an isotopy in $\Sigma\times\D^2\times Q$
  vertical with respect to the projection to $\Sigma\times\D^2$, and which
  is the identity close to  $\Sigma\times\D^2\times\partial Q$. (The like holds of course for
  the monodromy of cleft Haefliger structures).

\subsubsection{Splitting a hole into two holes}\label{splitting_sssec}
 Given  a compact manifold $Q$ and a factorization $\varphi=\varphi'\varphi''$ in the group
 $\diff(Q)^\I$,
 let $S$ be the $2$-sphere minus the interiors of three disjoint compact disks.
 Hence, $\pi_1(S)$ is the free group on two generators $\gamma'$,
 $\gamma''$, which one chooses such that the three components of $\partial S$,
 oriented as such,
 are respectively conjugate to $\gamma'\mun$ and $\gamma''\mun$ and $\gamma'\gamma''$.
Suspending the representation $\pi_1(S)\to\diff(Q)^\I$
that maps $\gamma'$ to $\varphi'$ and $\gamma''$ to $\varphi''$,
one gets on $S\times Q$ a codimension-$q$ foliation $\GG$ such that
\begin{enumerate}
\item $\GG$ is
complementary to the slice foliation parallel to $Q$;
\item $\GG$ is
parallel to $S$ on a neighborhood of $S\times\partial Q$;
\item $\GG$
induces the suspensions of $\varphi\mun$, $\varphi'\mun$, $\varphi''$ on
the three components of $\partial S\times Q$.
\end{enumerate}
Given also a manifold $\Sigma$,
pulling back $\GG$ into $\Sigma\times S\times Q$,
one obtains a partial filling of the hole $H$ of  core $\Sigma$,
fibre $Q$ and monodromy $\varphi$,
leaving two holes $H'\subset H$, $H''\subset H$
 of core $\Sigma$,  fibre $Q$ and respective monodromies $\varphi'$, $\varphi''$. Note that the
 core of $H'$ (resp. $H''$) \emph{embeds}
  into the base $\Sigma\times\D^2$ of $H$
through the projection $H\mapsto\Sigma\times\D^2$.

\subsubsection{Vertical shrinking of a hole}\label{vertical_sssec}
 If $Q'\subset Q$ is a domain containing in its interior the support of every $\varphi_t$,
where $\varphi=(\varphi_t)_{t\in\I}\in\diff(Q)^\I$, then in
the hole $\Sigma\times\D^2\times Q$
 of core $\Sigma$, fibre $Q$ and monodromy $\varphi$,
  we can foliate 
the subset $\Sigma\times\D^2\times(Q\setminus Q')$
 by the horizontal slice foliation parallel to
$\Sigma\times\D^2$, leaving
a smaller hole
 $\Sigma\times\D^2\times Q'$
  of the same core, fibre $Q'$ and monodromy $\varphi\vert Q'$.

 \subsection{Proof of Theorem \emph{A} as a corollary of
 \emph{A'}}\label{AAP_ssec}
    Let $M$, $\FF$, $\gamma$ and $\omega$ be as in the hypotheses of Theorem $A$.
    After pushing the microfoliation of $\gamma$ by the linear automorphism
    $(\omega\vert\tau\FF)\mun$ of the vector bundle $\nu\gamma=\tau\FF$,
     one is reduced to the case where
    $ \omega\vert\tau\FF=\id$.
 Choose
 $\varphi$ to be a multirotation (Definition \ref{multirotation_dfn} above).
 Applying $A'$ (i), one obtains
  a cleft $\Gamma_q$-structure $(C,[c],\bar\gamma)$.
  Enlarge the fissure into a hole $H\subset\bar M$ of fibre $\D^q$ and monodromy $\varphi$ (Tool \ref{fh_sssec}). Fill up
   this hole according to Paragraph \ref{filling_sssec} above;
there results on $\bar M$
 a (non cleft) $\Gamma_q$-structure $\tilde\gamma$ whose normal bundle is
 $\tau\bar\FF$, which coincides with 
  $\bar\gamma$
 on the complement of $H$,
  and whose restriction $\gamma':=\tilde\gamma\vert(M\times 1)$ to $M\times 1\cong M$
 is a (non cleft) foliation quasi-complementary to $\FF$.
 Finally, $\gamma'$ being special (Definition \ref{special_dfn}) on $(M\times 1)\setminus(H\cap(M\times 1))$,
 and after the above remark \ref{differential_rmk} applied in $H\cap(M\times 1)$,
 the differential
$d\gamma'$ is homotopic to $\omega$, rel. $\partial(M\times 1)$,
 among the $(\tau\FF)$-valued
$1$-forms on $M$ whose restriction to $\tau\FF$ is of constant rank $q$.

\begin{note}[on the hypothesis $q\ge 2$]
 In the proof of Theorem A, the hypothesis that $q\ge 2$ is crucial, allowing one
to choose  the diffeomorphism $\varphi$ of $\D^q$ to be a multirotation,
 and consequently, after  enlarging
the fissure, to fill the resulting hole of monodromy $\varphi$ (following Thurston)
 with the foliation described in the introduction. On the contrary, in codimension $q=1$, there exist no multirotations, since $\S^1$ does not embed in the interval $\D^1$!
 Filling a hole in codimension $1$
  requires in general a complicated construction where one first has to enlarge the hole
 by a ``worm gallery'' \cite{thurston_76} (see also \cite{meigniez_17}).
 For this reason, there seems to be no general result for the existence of codimension-$1$ quasicomplementary
  foliations, like Theorem A in the higher codimensions.
\end{note}

\section{Proof of Theorem \emph{A'}.}\label{proof_sec}

\subsection{Foliating a neighborhood of a codimension-$1$ skeleton
transversely to $\FF$}
Following a classical scheme,
a first part of the proof of Theorem \emph{A'}
will solve the problem on a small
neighborhood of the codimension-$1$ skeleton of a triangulation of a large domain in $M$;
this part is somehow standard, pertaining to Gromov's h-principle for $\diff$-invariant open
differential relations on open manifolds,
with the help of Thurston's jiggling lemma.
\medbreak

Let $M$, $\FF$, $\gamma$ be as in Theorem \emph{A'}.
Put $p:=\dim(M)-q$.
By the Jiggling lemma (\cite{thurston_74} section 5), one has
\begin{itemize}
\item A compact domain $D\subset Int(M)$, large enough that
$d\gamma$ is a linear retraction $\tau M\to\tau\FF$
over some open neighborhood of $M\setminus Int(D)$;
\item A triangulation $K$ of $D$
 which is
 in \emph{general position,} in Thurston's sense, with respect
to $\FF$.
\end{itemize} From Thurston's definition of ``general position'',
 we only need to recall that every cell $S$ of $K$ is transverse to $\FF$
(when $\dim(S)<p$, ``transverse'' means that there is no tangency)
and that, when $\dim(S)\ge p$, the foliation $\FF\vert Int(S)$ is
conjugate to the standard linear codimension-$p$ foliation on $\R^{\dim(S)}$.
Consider  the  $(p+q-1)$-skeleton $K^{(p+q-1)}$ of $K$.

\begin{lem}\label{skeleton_lem}
After a first concordance of $\gamma$
on $M$
rel. $\partial M$  (no cleft is necessary
 at this stage),
 one can moreover arrange that
 $d\gamma$ is also a linear retraction $\tau M\to\tau\FF$
 over some open neighborhood $U$ of $K^{(p+q-1)}$
  in $M$.
 \end{lem}
(About the differential being a linear retraction, recall Note \ref{retraction_note}).

\begin{proof} Choose a filtration $(K_n)$ ($0\le n\le N$) of  $K^{(p+q-1)}$
 by subcomplexes,
such that $K_0=K\vert\partial D$, that $K_N=K^{(p+q-1)}$,
and that $K_n$ is the union of $K_{n-1}$ with a single cell $S_n$.
By induction on $n$, assume that
 $d\gamma$ is already a linear retraction $\tau M\to\tau\FF$ over
  some open neighborhood  $U_{n-1}$ of $K_{n-1}$ in $M$. For convexity reasons,
  $d\gamma\vert U_{n-1}$ extends over $M$ to a global linear retraction $
  \omega:\tau M\to\tau\FF$.
There are two cases, depending on $d:=\dim(S_n)$.

In case $0\le d\le p$, one has an equidimensional
 embedding $$e:S_n\times\D^{p-d}\times\D^q
\hookrightarrow M$$
whose image is contained in a small neighborhood of $S_n$, and such that
\begin{itemize}
\item $e(x,0,0)=x$ for every $x\in S_n$;
\item $e$ maps $\partial S_n\times\D^{p-d}\times\D^q$
into $U_{n-1}$;
\item
 $e^*(\FF)$ is the slice foliation on
$S_n\times\D^{p-d}\times\D^q$ parallel to the $\D^q$ factor.
\end{itemize} One applies the integrable parametric
Foliation theorem on open manifolds (Theorem \ref{parametric_thm} of Paragraph
\ref{parametric_sssec})
 to the manifold $\D^q$, with $N:=\emptyset$,
the space of parameters being
$A:=S_n\times\D^{p-d}$, while $B:=\partial S_n\times\D^{p-d}$.

In case $p\le d\le p+q-1$, one has an equidimensional embedding
$$e:\D^p\times\D^{d-p}\times\D^{p+q-d}\hookrightarrow M$$
whose image is contained in a small neighborhood of $S_n$, and
 such that
\begin{itemize}
\item
 $e\mun(S_n)=\D^p\times\D^{d-p}\times 0$;
\item
 $e$ maps $\partial(\D^p\times\D^{d-p})\times\D^{p+q-d}$ into $U_{n-1}$;
 \item
 $U_{n-1}\cup e(\D^p\times\D^{d-p}\times\D^{p+q-d})$ is a neighborhood of $S_n$;
\item
 $e^*(\FF)$ is the slice foliation on $\D^p\times\D^{d-p}\times\D^{p+q-d}$
 parallel to the $\D^{d-p}\times\D^{p+q-d}$ factor.
\end{itemize}
One applies Theorem \ref{parametric_thm} to the manifold $\D^{d-p}\times\D^{p+q-d}$,
with $N:=\partial\D^{d-p}\times\D^{p+q-d}$,
the space of parameters being
$A:=\D^{p}$, while $B:=\partial\D^p$.

In both cases, one gets on some open neighborhood $U_n$
of $K_n$ in $M$ a regular $\Gamma_q$-structure $\gamma'$ such that
\begin{enumerate}
\item $\nu\gamma'=\tau\FF\vert U_n$;
\item $\gamma'$ is concordant with $\gamma$ rel. $\nb_M(K_{n-1})$;
\item The restriction $d\gamma'\vert\tau\FF$ is of constant rank $q$ over $U_n$;
\item $d\gamma'\vert\tau\FF$ is homotopic with $\id_{\tau\FF}$ among the
linear automorphisms of the vector bundle $\tau\FF\vert U_n$.
\end{enumerate}

Then, over $U_n$,
 push the microfoliation of $\gamma'$ through the linear automorphism
 $(d\gamma'\vert\tau\FF)\mun$ of the vector bundle $\tau\FF$, and thus
 get an isomorphic $\Gamma_q$-structure $\gamma''$ whose differential
  is a linear retraction of $\tau U_n$
 onto $\tau\FF\vert U_n$.  To verify that $\gamma''$ is concordant with $\gamma'$,
 and hence with $\gamma$, rel. $\nb_M(K_{n-1})$, one can refer to the general
 Lemma \ref{concordance_lem};
  but here it is even more
 obvious, because of (4). Finally, one calls
 to the Concordance Extension Property (Proposition \ref{cep_pro} above)
  to complete the induction.
\end{proof}

\subsection{Preparing the niche for inflation}\label{niche_ssec}
The next step in the proof of Theorem \emph{A'} is apparented to Po\'enaru's
flexibility for folded equidimensional maps
\cite{poenaru_66},
to Gromov's proof of the microcompressibility lemma
 (\cite{gromov_86} p. 81), to Lemma 2.3A in \cite{eliashberg_mishachev_97}, etc.
 
According to Lemma \ref{skeleton_lem},
one can define $\bar\gamma$ on $U\times\I$ as $\pr_1^*(\gamma)$.
 Then, vaguely, there remains, in the wall $M\times\I$, to fill the niches $Int(S)\times(0,1]$,
 where $S$ runs over the $(p+q)$-cells of $K$.  The present step will
  somehow normalize the Haefliger structure along the boundary of the niche,
   in order that  further down, we can
 solve the new extension problem by the ``inflation'' method.
 
  Precisely, for each $(p+q)$-cell $S$ of $K$,
 one has a large equidimensional embedding $j:\I^p\times\D^q\hookrightarrow Int(S)$
 containing $S\setminus(S\cap U)$ in its interior,
  and such that $j^*(\FF)$ is the slice
 foliation on $\I^p\times\D^q$ parallel to $\D^q$.
In this way, the proof of Theorem \emph{A'}
 is reduced to the case where $M:=\I^p\times\D^q$,
and where $\FF$ is the slice foliation parallel to $\D^q$. We restrict the attention to
this case.

Note --- The choice of the cube $\I^p$ here, rather than the ball $\D^p$, will be of no
importance for most of the construction, but it will make
civilization easier below in Subsection \ref{civilization_sssec}.

Recall that
 $\bar M:=M\times\I$; that $\hat M:=(M\times 0)\cup(\partial M\times\I)\subset\bar M$;
that $\bar\FF$ is on $\bar M$ the dimension-$q$ slice foliation  parallel
to $\D^q$; and that  $\pr_1:M\times\I\to M$ denotes the first projection.
Consider on $\bar M$
 the $\bar\Gamma_q$-structure $\pr_1^*(\gamma)$.

Put
$$X:=(\partial\I^{p}\times\D^q\times\I)\cup(\I^p\times\partial\D^{q}\times 1)
\subset\hat M\subset\bar M$$
Note that $X\cap(M\times 1)=\partial M\times 1$.
\begin{lem}\label{niche_lem}
There is on $\nb_{\bar M}(\hat M)$ a foliation $\GG$ such that
\begin{itemize}
\item $\GG$ is complementary
to $\bar\FF$ on $\nb_{\bar M}(\hat M)$;
\item $\GG$ is specially  induced by $\pr_1^*(\gamma)$ on
$\nb_{\bar M}(X)$;
\item $\GG$ is  concordant to $\pr_1^*(\gamma)$ on  $\nb_{\bar M}(\hat M)$
rel. $\nb_{\bar M}(X)$ (recall Conclusion \ref{induction_ccl} (a)).
\end{itemize}
\end{lem}

\begin{proof} 

Let $\omega$ be over $\bar M$ the identity automorphism
of the bundle $\tau\bar\FF$.

 First, regard $\gamma\vert(M\times 0)$
as an integrable parametric family of $\Gamma_q$-structures on $\D^q$,
the space of parameters being $\I^p$.
 The integrable parametric Foliation theorem
\ref{parametric_thm} is first applied to $\gamma$ and $\omega$ on the open manifold
  $\D^q$, the pair of parameters spaces being $$(A,B):=(\I^p,\partial\I^{p})$$
 One gets on $M\cong M\times 0$
 a foliation $\GG_0$ complementary to $\FF$ there, specially induced by $\gamma$ on
$\nb_{M\times 0}(\partial\I^{p}\times\D^q\times 0)$, and concordant to $\gamma$
(rel. $\partial\I^{p}\times\D^q\times 0$).

Next, by the concordance extension property for $\Gamma_q$-structures (see Paragraph
 \ref{concordance_sssec}),
 $\GG_0$ can be extended over
$\nb_{\bar M}(\hat M)$ by a $\bar\Gamma_q$-structure $\hat\gamma$,
coinciding with $\pr_1^*(\gamma)$ on
$\nb_{\bar M}(X)$, and concordant to $\pr_1^*(\gamma)$  rel. $X$.

Consider some small collar neighborhood $T\cong\S^{q-1}\times\I$ of 
$\partial\D^q$ in $\D^q$, and
 $$N:=\I^p\times T\times\I\subset\bar M$$
Regard $\hat\gamma\vert N$
  as an integrable family of $\Gamma_q$-structures on $T$,
  parametrized by $\I^{p+1}$. One obtains $\GG$ on $N$ by again applying
 the integrable parametric Foliation theorem \ref{parametric_thm}, this time
 to $\hat\gamma$ and $\omega$ on the open manifold
 $T$; the pair of parameters spaces
being $$(A,B):=(\I^{p+1},\partial\I^{p+1})$$
 \end{proof}

Fix a compact collar neighborhood $Col$ of $\partial
M$ in $M$,
 small enough
that $d\gamma$ is a linear retraction $\tau M\to\tau\FF$
 over $Col$. Put $\Delta:=M\setminus Int(Col)$ and
$$W:=(M\times[0,1/2])\cup(Col\times[1/2,1])\subset\bar M$$
$$\partial_1W:=(\Delta\times(1/2))\cup
(\partial\Delta\times[1/2,1])$$
As a paraphrase of Lemma \ref{niche_lem}, there is on $\nb_{\bar M}(W)$
a $\bar\Gamma_q$-structure $\bar\gamma$
 such that:
\begin{itemize}
\item On $\nb_{\bar M}(\hat M\cup(Col\times 1))$, one has
  $\bar\gamma=\pr_1^*(\gamma)$;
\item  On $\nb_{\bar M}(\partial_1W)$, the
$\bar\Gamma_q$-structure $\bar\gamma$ is regular and
 specially induces a foliation  $\GG$ complementary
to $\bar\FF$.
\end{itemize}

To prove Theorem \emph{A'}
(i), it remains to extend $\bar\gamma$ through the smaller niche
 $\bar M\setminus W$.
 Hence, one is reduced to Proposition \ref{inflation_pro} below.

\subsection{Inflation}\label{inflation_ssec}
Recall that $M=\I^p\times\D^q$ ($p\ge 0$, $q\ge 1$), that $\bar M=M\times\I$,
that $\hat M=(M\times 0)\cup(\partial M\times \I)\subset\bar M$,
 that
 $\bar\pi$ is the projection $$\bar\pi:\bar M=\I^p\times\D^q\times\I\to\I^{p+1}:
 (x,y,t)\mapsto(x,t)$$ and
that $\bar\FF$ is the $q$-dimensional slice
 foliation
of $\bar M$ parallel to the $\D^q$ factor. Let
$\varphi$ be as in Theorem \emph{A'}.

 \begin{pro}\label{inflation_pro}
Let  $\GG$ be along $\hat M$ a germ of foliation complementary to
$\bar\FF$ there.

 Then, $\GG$ extends to all of $\bar M$ as a
 \emph{cleft} foliation of monodromy $\varphi$ quasi-complementary
to $\bar\FF$.
\end{pro}

Note --- Here, the fact that cubes and disks are cubes and disks is actually unimportant:
the same result would go for any two compact manifolds instead of $\I^p$ and $\D^q$,
with almost the same proof.
\medbreak
 The rest of the present subsection \ref{inflation_ssec} is to prove Proposition
\ref{inflation_pro}.
 The difficulty lies of course
  in the arbitrary position of $\GG$ with respect to $\I^{p}\times\partial\D^{q}\times\I$. The ``inflation''
method introduced by Thurston \cite{thurston_74}\cite{thurston_76} to prove the Foliation theorem
on closed manifolds will fit to solve this difficulty, after some
adjustments.    We give the details for three reasons. First, our framework
 is not exactly the same as Thurston's:
  he foliated simplices, we foliate prisms (by which we mean the
 product of
 a simplex by a smooth disk).
  Second, some think that
   the argument in \cite{thurston_74} is difficult; some have even believed
   that it was not
   fully convincing --- of course, it is not the case. Third, we feel that ``inflation'' deserves to
   be more widely used
    as a general method in the h-principle, which it has not been, since the
    fundamental papers \cite{thurston_74}\cite{thurston_76},
 but in \cite{laudenbach_meigniez_16}.

\subsubsection{Building the prisms}\label{prisms_sssec}
Let $Dom(\GG)$ be a small open neighborhood of $\hat M$ in $\bar M$,
on which $\GG$ is defined and complementary to $\bar\FF$.
We shall first decompose the domain to be foliated, precisely
 $\bar M$ minus some smaller open neighborhood of $\hat M$, into many thin vertical prisms
  (Figure \ref{prisms_fig}).

 Endow $\I^{p+1}$ with a triangulation $K$;
  choose it linear for convenience. The choice of $K$ will be made more precise further down.
 
 \begin{lem}\label{prism_lem}
 Provided that $K$ is fine enough, one can choose, for each cell $\alpha$ of $K$,
 an embedding
  $$e_\alpha:\alpha\times\D^q\hookrightarrow
 \bar M$$
such that, for every $x\in\alpha$:
 \begin{itemize}
 \item  $e_\alpha(x\times\D^q)\subset Int(\bar\pi\mun(x))$;
 \item $\bar\pi\mun(x)\subset e_\alpha(x\times Int(\D^q))\cup Dom(\GG)$;
    \item If $x$ lies on some proper face $\beta\varsubsetneq\alpha$,
     then $$e_\beta(x\times\D^q)\subset e_\alpha(x\times Int(\D^q))$$
     \end{itemize}
     and that, for every $y\in\partial\D^q$,
  $e_\alpha(\alpha\times y)$
   is contained in a leaf of $\GG$.
 \end{lem}

  \begin{proof}

Regard $\bar M$ as $\I^{p+1}\times\D^q$; endow both factors with the
 Euclidian metrics.
   Fix $\epsilon>0$ so small that $(x,y)\in Dom(\GG)$ holds for every $x\in\I^{p+1}$
   and every $y\in\D^q$ such that $$\vert y\vert\ge 1-(p+3)\epsilon$$
  Put $r_d:=1-(p+2-d)\epsilon$ for every integer $0\le d\le p+1$.
\begin{figure}
\includegraphics*[scale=0.45, angle=-90]{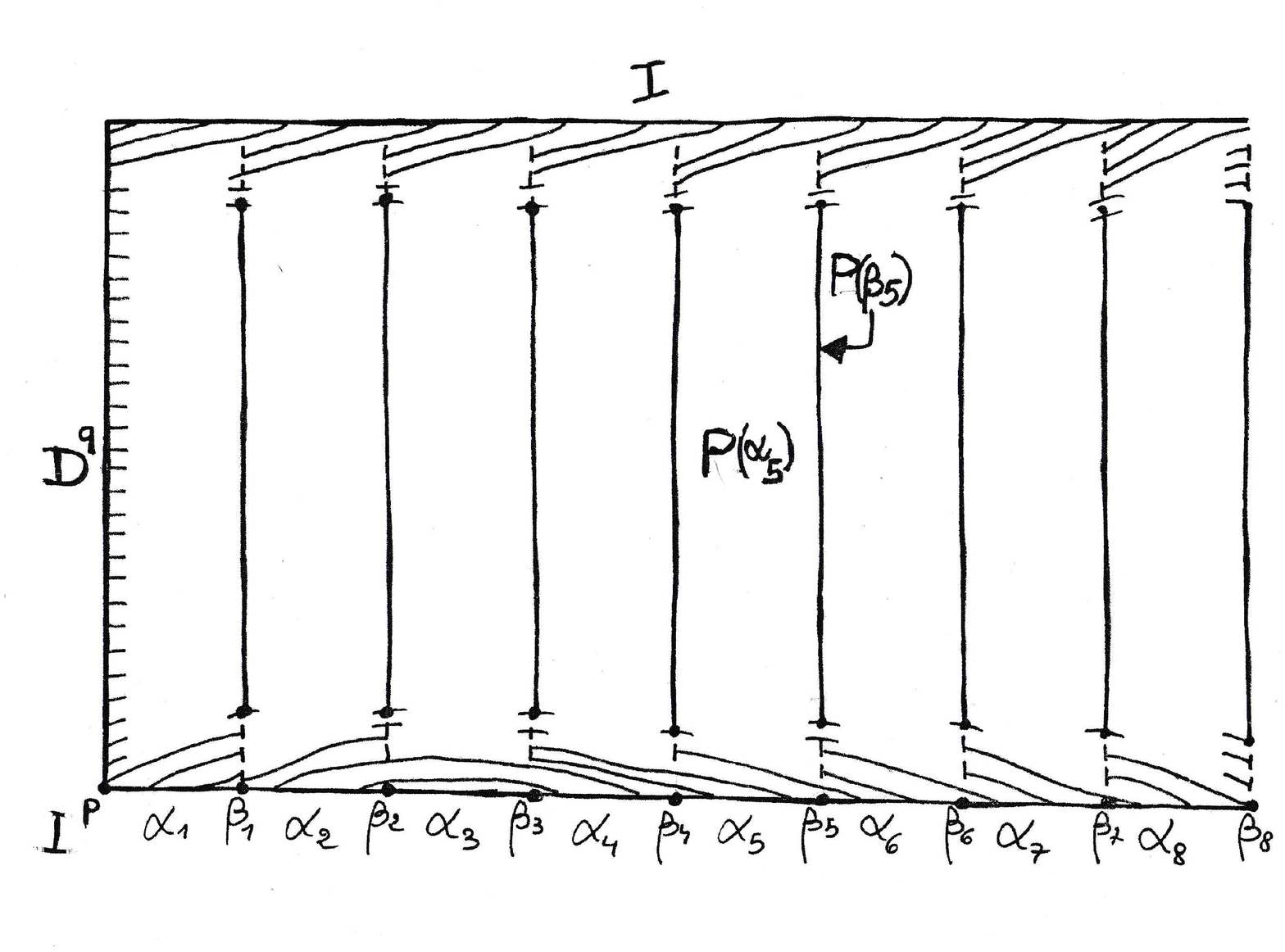}
\caption{Decomposition of the most part of $\bar M$ into prisms,
in the case $p=0$, $q=1$.}
\label{prisms_fig}
\end{figure}

 Consider the partially defined monodromy of $\GG$ close to $\partial M\times\I$.
  Precisely, for a linear
 path $u:\I\to{\I^{p+1}}$ and for $y\in\D^q$, denote by $(u(1),h_u(y))$
 the extremity of the path (if any) in 
 $\bar M$  originated at $(u(0),y)$, covering $u$
 through $\bar\pi$, 
 and tangential to $\GG$.  
   Fix $\delta>0$ small enough that
for every linear path $u$
whose length $\vert u\vert$ is at most $\delta$,
 and every $y\in\D^q$ with $r_0\le\vert y\vert\le r_{p+1}$, the
 monodromy $h_u(y)$ is defined, and $$\vert h_u(y)-y\vert<\epsilon/2$$
 
Choose the linear triangulation $K$ such
   that every cell  $\alpha$ of $K$ is of diameter  at most $\delta$.
 For every point $x\in\alpha$ and every
 $y\in\partial\D^q$, put
  $$e_\alpha(x,y):=h_{[\flat(\alpha),x]}(r_{\dim(\alpha)}y)$$
This is  a smooth family of embeddings $\partial\D^q\hookrightarrow\D^q$,
  parametrized by the points of $\alpha$. By the isotopy
  extension property, this family extends to
    a smooth family of embeddings $\D^q\hookrightarrow\D^q$, also denoted by $e_\alpha$.
    The demanded properties hold by definition or follow
    from the triangle inequality.
\end{proof}
 
   We write $$P(\alpha):=e_\alpha(\alpha\times\D^q)$$
  $$R(\alpha):=(\alpha\times\D^q)\setminus e_\alpha(\alpha\times Int(\D^q))$$

    Let $R\subset\bar M$
be the union of the $R(\alpha)$'s, for all cells $\alpha$ of $K$.
 Hence, $R$
  is   a compact neighborhood of $\I^p\times{\partial\D^q}\times\I$ in $\bar M$.
  The foliation $\GG$ is defined
   on $\nb_{\bar M}(R\cup\bar M_0)$ and complementary to $\bar\FF$ there.

  \subsubsection{Civilizing the prisms}\label{civilization_sssec}
  The tool of ``civilization'' was introduced by Thurs\-ton \cite{thurston_74}
  in order to guarantee that some microextensions of foliations defined on the cells
  of a triangulation are compatible and thus give a global smooth foliation;
  we slightly depart from his vocabulary.

  \begin{dfn}\label{civilization_dfn}
   Let $X$ be a manifold
  and $Y\subset X$ be a compact submanifold, 
   maybe with boundary
  and corners.
 A \emph{civilization} for $Y$ in $X$ is
  a foliation $\CC$
  of codimension $\dim(Y)$ on $\nb_X(Y)$,
  transverse to $Y$.
  
\end{dfn} 
 Actually, the object of interest is the \emph{germ} of $\CC$ along $Y$.
 The utility of civilization lies in the following obvious remark.
   Consider
   the local retraction along the leaves of $\CC$:
 $$\pr:\nb_X(Int(Y))\to Int(Y)$$

 \begin{rmk}\label{tangent2_rmk}
  (Recall Tool \ref{pullback_sssec} (c) and Definitions
  \ref{tangent2_dfn} and \ref{pullable_dfn}.)
 \begin{itemize}
 \item Let $G$ be a cleft foliation on $Int(Y)$. Then, $\pr^*(G)$ is a cleft foliation on
 $\nb_X(Int(Y))$ and $\CC$ is tangential to $\pr^*(G)$ there;
 \item Conversely, let $G$ be a cleft foliation on $\nb_X(Y)$, to which $\CC$ is tangential.
  Then, $G$ is restrictable to $Int(Y)$, and $G=\pr^*(G\vert Int(Y))$ on $\nb_X(Int(Y))$.
  \end{itemize}
 \end{rmk}
\begin{ntt}\label{inclusion_ntt} Given two foliations $\AA$, $\BB$ on a same manifold,
we briefly write $\AA\subset\BB$ to mean that $\tau_x\AA\subset\tau_x\BB$
at every point $x$.
\end{ntt}

  The advantage of having chosen especially a \emph{polyhedral} base 
  $\I^{p+1}\subset\R^{p+1}$ and a \emph{linear} triangulation $K$ of this base
  is that each cell $\alpha$ of $K$ admits in $\R^{p+1}$ an obvious civilization $\NN_\alpha$,
  namely, the parallel linear foliation of $\R^{p+1}$ orthogonal to $\alpha$ for the
  Euclidian metric, such that for every face $\beta\subset\alpha$,
  the compatibility condition $\NN_\alpha\subset\NN_\beta$ holds on
   $\nb_{\R^{p+1}}(\beta)$.

Consider the inclusion $$\bar M=\I^{p+1}\times\D^q\subset \R^{p+1}\times\D^q$$ and
 the first projection $\pi:\R^{p+1}\times\D^q\to\R^{p+1}$.

\begin{lem}\label{civilization_lem}
 One can choose, for every cell $\alpha$ of $K$, a civilization
 $\CC_\alpha$
for $\alpha\times\D^q$ in $\R^{p+1}\times\D^q$ such that
\begin{enumerate}
 \item  $\CC_\alpha\subset\GG$ on
  $\nb_{\bar M}(R(\alpha))$
  and, if $\alpha\subset K_0$, on
  $\nb_{\bar M}(\alpha\times\D^q)$;
 \item $\CC_\alpha\subset\pi^*(\NN_\alpha)$ on $\nb_{\R^{p+1}}(\alpha)\times\D^q$;
\item $\CC_\alpha\subset\CC_\beta$ on $\nb_{\R^{p+1}}(\beta)\times\D^q$,
 for every face $\beta\subset\alpha$.

  \end{enumerate}
  \end{lem}
  \begin{proof} Denote by $\FF$
  the slice foliation  of $\R^{p+1}\times\D^q$
  parallel to $\D^q$. It is easy to extend $\GG$ to a codimension-$q$
   foliation, still denoted by $\GG$,
  complementary to $\FF$ on $\nb_{\R^{p+1}\times\D^q}(R\cup\bar M_0)$.
   Then, make a second, non necessarily integrable extension:
    let $\xi\subset\tau(\R^{p+1}\times\D^q)$ be a $(p+1)$-plane field
  complementary to $\FF$, and coinciding with $\tau\GG$ over
   $\nb_{\R^{p+1}\times\D^q}(R\cup\bar M_0)$.
    Consider a cell $\alpha$ of $K$. On
     $\nb_{\R^{p+1}\times\D^q}(R(\alpha))$,
    define $\CC_\alpha$ as the pullback of $\NN_\alpha$ through $\pi$
    tangentially to $\GG$. Then, extend $\CC_\alpha$ over
     $\nb_{\R^{p+1}\times\D^q}(\alpha\times\D^q)$
    as follows.   Denote by $\langle\alpha\rangle\subset\R^{p+1}$
     the affine subspace spanned by $\alpha$.
      At every point $(x,y)\in\langle\alpha\rangle\times\D^q$,
   define the leaf of $\CC_\alpha$ through $(x,y)$
   as the union of the paths in $\R^{p+1}\times\D^q$ starting from $(x,y)$, tangential
   to $\xi$, and whose projection through $\pi$ is a linear path in $\R^{p+1}$
   orthogonal to $\langle\alpha\rangle$ at $x$.
    The properties (1) through (3) are immediate.
  \end{proof}

\subsubsection{Foliating the prisms}\label{foliation_sssec}
   
     We can choose the fine linear triangulation $K$ of $\I^p\times\I$
 such that moreover, $K$
 collapses onto
 its subcomplex $$K_0:=K\vert((\I^p\times 0)\cup({\partial\I^p}\times\I))$$
 Such a collapse
 amounts to a filtration of $K$ by subcomplexes $(K_n)$ ($0\le n\le N)$
  such that
   $K_N=K$, and such that for every $1\le n\le N$, exactly two cells $\alpha_n$,
  $\beta_n$ lie in $K_n$ but not in $K_{n-1}$; moreover $\beta_n$
  is a hyperface of $\alpha_n$. We write
    $$\bar M_n:=K_n\times\D^q\subset\bar M$$
 
 The inflation process is an induction:
    the prisms will be foliated (by cleft foliations quasi-complementary to the 
    fibres of $\bar\pi$) one after the other, in the order given by
   the collapse.

  \begin{voc}\label{system_voc}
  (a) A \emph{smooth partial triangulation} in a manifold $\Sigma$ is a
  geometric simplicial complex which is topologically embedded in $\Sigma$, and such that
   each cell is smoothly embedded  in $\Sigma$.
  
  (b) By a \emph{system of spheres} in $\Sigma$, we mean a finite disjoint union
  of smoothly embedded spheres, whose dimensions are not necessarily
  the same and may vary between $0$ and $\dim(\Sigma)$, and whose
  normal bundles are trivial.
  
  \end{voc}
   The induction hypotheses are the following.
  Recall Definitions \ref{qc_dfn} and \ref{tangent2_dfn}.

\begin{ppt}\label{induction_ppt} There is
  on $\nb_{\bar M}(R\cup\bar M_{n})$ a cleft foliation
 $$G_n:=(C_n,[c_n],\GG_n)$$ of monodromy $\varphi$,  
such that
\begin{enumerate}
\item $G_n$ is quasi-complementary to $\bar\FF$
  on $\nb_{\bar M}(R\cup\bar M_{n})$;
\item
 $G_n$
  coincides with $\GG$
on $\nb_{\bar M}(R\cup\bar M_0)$;
\item For each cell $\gamma$ of $K_n$, the civilization $\CC_\gamma$
is tangential to $G_n$ on $\nb_{\bar M}(\gamma\times\D^q)$;
\item $\bar M_n\cap\Sigma_n$
 (where $\Sigma_n:=c_n\mun(0,0)$) admits a triangulation  $\Delta_n$
which is a smooth partial triangulation
 of $\Sigma_n$ and collapses onto a 
system of spheres in $\Sigma_n$;
\item $\bar\pi$ is one-to-one in restriction to each connected component of
 $\bar M_n\cap\Sigma_n$.
\end{enumerate}
\end{ppt}

The object of interest is actually the \emph{germ} of $G_n$ along $R\cup\bar M_n$.

 For $n=0$, the (non cleft) foliation $G_0:=\GG$ satisfies the
  above
   Property \ref{induction_ppt} (1) through (5), in view of Lemma \ref{civilization_lem} (1).

 Now, for some $1\le n\le N$,
  assume that the above induction hypotheses are met for $n-1$.
  
 For short, write $\alpha, \beta$, $i$, 
  instead of $\alpha_n$, $\beta_n$,
    $\dim(\beta_n)$.  
     We work in $\alpha\times\D^q$. ``Horizontal'', for a vector field or a foliation,
     means parallel to the $\alpha$ factor. Denote
   by $\hat\partial\alpha$ the union of the hyperfaces of $\alpha$
  other than $\beta$, by
     $\pi_\alpha$ the first projection $\alpha\times\D^q\to\alpha$, and by
   $\FF_\alpha$ the slice foliation of $\alpha\times\D^q$
  parallel to $\D^q$.
  
  The extension of the cleft foliation
   $G_{n-1}$, defined on $\nb_{\bar M}(R\cup\bar M_{n-1})$,
    to a cleft foliation $G_n$ on $\nb_{\bar M}(R\cup\bar M_n)$,
   will be in two steps: first, an extension only to $P(\alpha)$; then, a microextension
   from $P(\alpha)$ to $\nb_{\bar M}(P(\alpha))$, by means of the civilizations
   $\CC(\alpha)$ and $\CC(\beta)$.
  \begin{lem}\label{pullable_lem} For each cell $\gamma$ of $K$, the cleft foliation
  $G_{n-1}$ is pullable through the map $e_\gamma$ restricted to
    $\nb_{\gamma}(\gamma\cap K_{n-1})\times\D^q$,
  and defines on such a neighborhood a pullback cleft foliation $e_\gamma^*(G_{n-1})$
  quasi-complementary to the slices parallel to $\D^q$.
  \end{lem}
  \begin{proof} Consider any point $(x,y)\in(\gamma\cap K_{n-1})\times\D^q$.
  Let $\delta\subset\gamma\cap K_{n-1}$ be the smallest face containing $x$.
  By Property \ref{induction_ppt} (3) applied to $\delta$  at order $n-1$
   and by (3) of Lemma
  \ref{civilization_lem}, $\CC_\gamma$ is tangential to $G_{n-1}$ on
  $\nb_{\bar M}(x,y)$. Since the embedding
  $e_\gamma$ is transverse to $\CC_\alpha$ at the point $(x,y)$,
    the pullability and the quasi-complementarity follow (see Tool \ref{pullback_sssec} (e)).
  \end{proof}

  \begin{lem}\label{extension_lem}  There is on $\alpha\times\D^q$ a cleft foliation
  $G_\alpha$ such that
\begin{enumerate}
\item $G_\alpha$ is quasi-complementary to $\FF_\alpha$
 on $\alpha\times\D^q$;
 \item $G_\alpha$ coincides with $e_\alpha^*(G_{n-1}=\GG)$ on
$\nb_{\alpha\times\D^q}(\alpha\times\partial\D^q)$;
\item $G_\alpha$ coincides with $e_\alpha^*(G_{n-1})$ on
 $\nb_{\alpha\times\D^q}(\hat\partial\alpha\times\D^q)$;
\item $G_\alpha$ coincides with $e_\alpha^*(G_{n-1}=\GG)$ on
$\nb_{\alpha\times\D^q}(e_\alpha\mun(R(\beta))$;
\item $e_\alpha^*(\CC_\beta)$ is tangential to $G_\alpha$
 on $\nb_{\alpha\times\D^q}(\beta\times\D^q)$.

\end{enumerate}
\end{lem}

\begin{figure}
\includegraphics*[scale=0.45, angle=0]{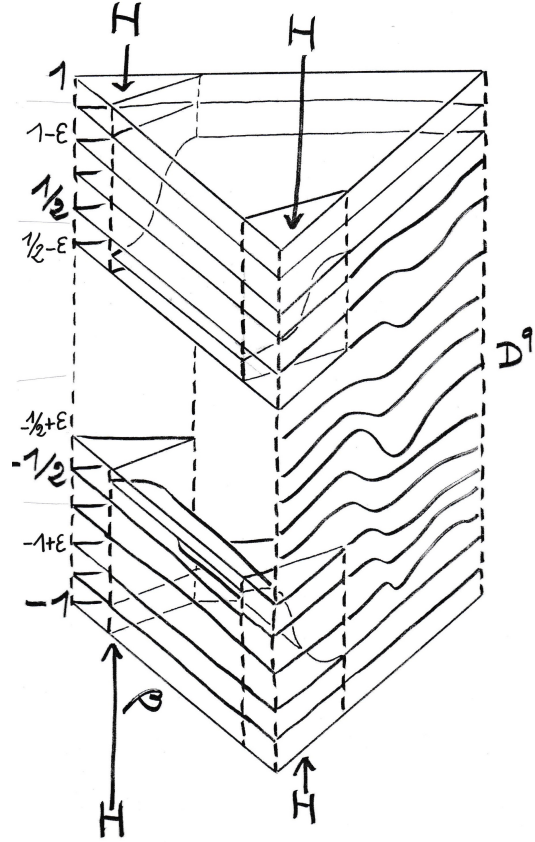}
\caption{Foliating the prism, in the case $i=q=1$.}
\label{inflation_fig}
\end{figure}

\begin{proof} The construction of $G_\alpha$ is the main part of the inflation
process (Figure \ref{inflation_fig}). Informally,
after Thurston's image \cite{thurston_74},
we can think of the prism $\alpha\times\D^q$, which we have to foliate,
 as a room. The back wall is $\hat\partial\alpha\times\D^q$,
 the
ceiling and floor are $\alpha\times\partial\D^q$ (here of course
 the image is less realistic for $q>1$);
 the front wall $\beta\times\D^q$
splits as the union of a window $e_\alpha\mun(P(\beta))$ and a corona
$e_\alpha\mun(R(\beta))$; there are no side walls in our prismatic version. 

   A cleft foliation
  $e_\alpha^*(G_{n-1})$ is already defined
close to the back wall, the ceiling and floor, and the corona.
Close to the ceiling, floor and corona, the foliation $e_\alpha^*(G_{n-1})$
 is not cleft; actually it is essentially horizontal (parallel to $\alpha$) there.
  On the contrary, close to the back wall, 
  $e_\alpha^*(G_{n-1})$ is in general cleft and complicated, being produced by the
previous steps of the induction.

The general idea is to pull back this foliation through
the room by means of a flow crossing the room
from the front wall to the back wall and tangential to the ceiling and to the floor.

The difficulty
 is that in general, it is not possible to make the pullback foliation match the horizontal
 foliation on the corona. This is solved at the price of a new cleft (or hole), which will allow us to
 somehow divert the complicated part of the back wall foliation
 from the corona, and make it exit, instead, through the window.
 
 Some technicalities arise from the fact that the foliation
 to be pulled back is cleft, that the flow must be tangential to this foliation where it is already
 defined,
  that the
 extended cleft foliation must be quasi-complementary to the verticals
  (condition (1) above)
  and tangential to the
 civilization of the front wall $\beta\times\D^q$ (condition (5));
  this last condition being crucial in view of the subsequent
 microextension from $\alpha\times\D^q\cong P(\alpha)$ to $\nb_{\bar M}(P(\alpha))$.
 
 \medbreak
 
 First, some normalizations
   will fix the ideas and simplify the notations. We have some choice in
   the parametrization $e_\alpha$ of $P(\alpha)$: we can change
   $e_\alpha$ to $e_\alpha\circ F$, where $F$ is
    any isotopy of $\alpha\times\D^q$ vertical with respect to $\pi_\alpha$
    (Definition \ref{vi_dfn}).

 I) The boundary component
   $\alpha\times{\partial\D^q}$ is saturated for the foliation
   $e_\alpha^*(\GG)$ (recall Lemma \ref{prism_lem}),
  which on $\alpha\times{\partial\D^q}$
  restricts to the slice foliation parallel to $\alpha$. Hence,
   after a first vertical isotopy, 
   we can arrange,
     on $\nb_{\alpha\times\D^q}(\alpha\times{\partial\D^q})$, that
   $e_\alpha^*(\GG)$ (which coincides with $e_\alpha^*(\GG_{n-1})$ there
   by Property \ref{induction_ppt} (2) at order $n-1$),
    is horizontal there. 
   
  II)  In $\R^q$,
  denote by $D_r$ (resp. $A_r$) the compact disk
  (resp. corona) defined by $\vert y\vert\le r$
  (resp. $r\le\vert y\vert\le 1$). After a second vertical isotopy, relative to 
  $\nb_{\alpha\times\D^q}(\alpha\times{\partial\D^q})$, 
   we can arrange that
   $$e_\alpha\mun(P(\beta))=\beta\times D_{1/2}
  \ \text{and}\ e_\alpha\mun(R(\beta))=\beta\times A_{1/2}$$

    III) Since  $e_\alpha^*(\GG)\vert(\beta\times A_{1/2})$
  is a foliated product whose base $\beta$ is simply connected, this foliated product is trivial. Hence,
   after a third vertical isotopy, relative to $\beta\times D_{1/2}$ and to
    $\nb_{\alpha\times\D^q}(\alpha\times{\partial\D^q})$,
   we can moreover arrange that
   $e_\alpha^*(\GG_{n-1}=\GG)$ is horizontal
   on $\nb_{\alpha\times\D^q}(\beta\times A_{1/2})$.
   \medbreak

 In the case $i=0$, Lemma  \ref{extension_lem}
  is trivial: $\alpha$ is an interval whose endpoints are
  $\beta$ and $\hat\alpha$. The foliation $\CC(\alpha)$ 
  being of codimension $1$ in $\bar M$, the tangentiality
  property \ref{induction_ppt} (3) above at order $n-1$ simply means that the prism
   $P(\alpha)$
   does not meet $C_{n-1}$. Hence, $\alpha\times\D^q$ is endowed along
 $\hat\alpha\times\D^q$, along $\alpha\times\partial\D^q$ and along
  $\beta\times A_{1/2}$ 
 with
 a germ of $1$-dimensional foliation $e_\alpha^*(\GG_{n-1})$ (not cleft),
 complementary to $\FF_\alpha$, and horizontal on $\alpha\times\partial\D^q$
  (Property \ref{induction_ppt} (1) and (2)
 at order $n-1$).
  One has an obvious extension of
  $e_\alpha^*(\GG_{n-1})$
 by a $1$-dimensional foliation complementary to $\FF_\alpha$ on the all of $\alpha\times\D^q$.
 
  For the rest of the proof of Lemma
  \ref{extension_lem} , we assume that $i\ge 1$.

 \begin{clm}\label{crossing_clm} There are
 \begin{itemize}
 \item  On $\alpha$, a nonsingular vector field $\nabla$;
 \item    On $\alpha\times\D^q$, a nonsingular vector field $\tilde\nabla$;

 \end{itemize}
  such that
 \begin{enumerate}
 \item $\nabla$ is transverse to every hyperface of $\alpha$;
\item Every orbit of $\nabla$ enters $\alpha$ through
$\beta$ and exits $\alpha$ through  $\hat\partial\alpha$;
 \item $\tilde\nabla$ lifts
 $\nabla$ through $\pi_\alpha$;
 \item $\tilde\nabla$ is tangential to $e_\alpha^*(\GG)$ on
  $\nb_{\alpha\times\D^q}(\alpha\times\partial\D^q$);
   \item $\tilde\nabla$ is tangential to $e_\alpha^*(\GG)$ on
  $\nb_{\alpha\times\D^q}(\beta\times A_{1/2}$);

  \item $\tilde\nabla$ is tangential to $e_\alpha^*(G_{n-1})$  on $\nb_{\alpha}(\hat\partial
  \alpha)\times\D^q$.

 \end{enumerate}

 \end{clm}
 
  For (6), recall
  Definition \ref{tangent1_dfn}.
  
The construction requires a little care because
 $\alpha$ is an arbitrary linear $(i+1)$-simplex
  in the Euclidian space $\R^{p+1}$; for example,  some of the dihedral angles
  of $\alpha$
  along the hyperfaces of $\beta$
  can be obtuse.

 \begin{proof}[Proof of Claim \ref{crossing_clm}]
 For simplicity, in the proof of this claim, denote by $\CC_\gamma$, $\GG$, $G_{n-1}$
the pullbacks of $\CC_\gamma$, $\GG$, $G_{n-1}$ in $\alpha\times\D^q$ through $e_\alpha$.

 For each face $\gamma\subseteq\alpha$, its barycenter $\flat(\gamma)$ is
 a vertex of the first barycentric subdivision
 $Bar(\alpha)$  of $\alpha$. Let $S_\gamma
 \subset\alpha$ be the open star of  $\flat(\gamma)$
 with respect to $Bar(\alpha)$ (the interior of the union of the
 cells of $Bar(\alpha)$ containing $\flat(\gamma)$).
  The interest of these stars is that for every two faces
  $\gamma, \delta\subseteq\alpha$, one obviously has
 \begin{enumerate}[label=(\alph*)]
 \item $Int(\gamma) \subset S_\gamma$;
 \item $S_\gamma$ intersects $\delta$ iff $\gamma\subset\delta$;
  \item  $S_\gamma$ intersects $S_\delta$ iff $\gamma\subset\delta$
  or $\delta\subset\gamma$.
 \end{enumerate}

\medbreak

   Subclaim A: \emph{
   There is, for every face $\gamma\subseteq\alpha$,
   an open subset $U_\gamma\subset\alpha$ such that 
    for every two faces
  $\gamma, \gamma'\subseteq\alpha$:}
  \begin{enumerate}[label=(\roman*)]
\item $Int(\gamma)\subset U_\gamma\subset S_\gamma$;

\item $\CC_\gamma$ is defined on $U_\gamma\times\D^q$, and
$\CC_\gamma\subset\GG$ on $U_\gamma\times\nb_{\D^q}(\partial\D^q)$;

\item 
Every leaf of $\CC_\gamma\vert(U_\gamma\times\D^q)$ is mapped
diffeomorphically onto a leaf of $\NN_\gamma\vert U_\gamma$ through $\pi_\alpha$;

\item If $\gamma\neq\alpha,\beta$, then $G_{n-1}$ is defined on $U_\gamma\times\D^q$,
and $\CC_\gamma$
is tangential to $G_{n-1}$
on $U_\gamma\times\D^q$;

\item $\CC_\beta\subset\GG$
on $U_\beta\times\nb_{\D^q}(A_{1/2})$;

\item If $\gamma\subset\gamma'$ then
 $\CC_{\gamma'}\subset\CC_\gamma$ on $(U_\gamma\cap U_{\gamma'})\times\D^q$.
 \end{enumerate}

 The proof of the subclaim A is
   straightforward by means of a descending induction on $\dim(\gamma)$,
   using
  Lemma \ref{civilization_lem}, and (a) above, and the induction properties \ref{induction_ppt} (2) and (3)
  at the order $n-1$.

 \medbreak
We now resume the proof of Claim \ref{crossing_clm}.
By (i), the family
  $(U_\gamma)_{\gamma\subseteq\alpha}$
  is an open cover of $\alpha$.
    Let
 $(u_\gamma)_{\gamma\subseteq\alpha}$ be on $\alpha$ a
  partition of the unity subordinate to this cover.

 Definition of $\nabla$:
  consider the vector $v_\alpha:=\flat(\alpha)-\flat(\beta)$; and,
   for each $\gamma\subset\partial\alpha$,
    the orthogonal projection $v_\gamma$ of $v_\alpha$ into the linear subspace $\gamma^\bot
    \subset\R^{p+1}$
 orthogonal to $\gamma$.
 Put
 $$\nabla:=\sum_{\gamma\subseteq\alpha}u_\gamma v_\gamma$$
 
Verification of (1):
  given a hyperface $\eta\subset\alpha$,
    let  $n_\eta\in\R^{p+1}$ be the unit
  vector parallel to $\alpha$ and normal to $\eta$,
 pointing inwards $\alpha$ if $\eta=\beta$, and outwards if $\eta\neq\beta$. 
   Since $\alpha$ is a linear simplex,
    $\langle v_\alpha,n_\eta\rangle>0$ (Euclidian scalar product in $\R^{p+1}$).
Consider any point $x\in\eta$.
 For every face $\gamma\subseteq\alpha$ such that $x\in U_\gamma$,
by (i) and (b) one has
 $\gamma\subset\eta$,
hence $n_\eta\in\gamma^\bot$, hence
 $\langle v_\gamma,n_\eta\rangle=\langle v_\alpha,n_\eta\rangle$.
Finally:
$$\langle\nabla(x),n_\eta\rangle=\langle v_\alpha,n_\eta\rangle>0$$

 Verification of (2): for every face $\gamma\subset\partial\alpha$, clearly $v_\alpha$ is not
 parallel to $\gamma$, hence $v_\gamma\neq 0$ and
  $$\langle v_\alpha,v_\gamma\rangle=\vert v_\gamma\vert^2>0$$ It follows, in view of
  the
  definition of $\nabla$, that $\langle
 v_\alpha, \nabla\rangle>0$ at every point of $\alpha$. So, every orbit of $\nabla$
  is proper in $\alpha$ (a compact segment with endpoints on $\partial\alpha$).
 Since $\langle\nabla(x),n_\eta\rangle>0$ at every
  every point $x$ of every hyperface $\eta$, necessarily every orbit goes across $\alpha$
  from $\beta$ to
$\hat\partial\alpha$.
 
 Definition of $\tilde\nabla$:
  the constant vector field on $\alpha$ parallel to $v_\alpha$
  lifts through $\pi_\alpha$ to a vector field $\tilde v_\alpha$ on $\alpha\times\D^q$,
 which we can choose tangential to $\GG$ close to $\alpha\times\partial\D^q$.
 On the other hand,  for each face $\gamma\subset\partial\alpha$,
  by  (iii),
 the constant vector field on $U_\gamma$ parallel to $v_\gamma$, being
  tangential to $\NN_\gamma$,
  lifts in $\alpha\times\D^q$ through $\pi_\alpha$ to a unique vector field
   $\tilde v_\gamma$ on $U_\gamma\times\D^q$
  tangential to $\CC_\gamma$.
 Put
 $$\tilde\nabla:=
 \sum_{\gamma\subseteq\alpha}(u_\gamma\circ\pi_\alpha)\tilde  v_\gamma$$
 
 (3) holds by definition; (4) follows from
(ii).
 We are left to verify (5) and (6).
 
First, note that 
 for every $x\in\alpha$, the set of the faces  $\gamma\subseteq\alpha$ such that $x\in\supp(u_{\gamma})$ has a
 smallest element
  $\delta(x)$: indeed, this finite set  is \emph{totally} ordered by the inclusion relation,
  because of (i) and (c). Moreover,
 {for every
  $x'\in\alpha$ close enough to $x$,
 one has $\delta(x)\subset\delta(x')$}
  (since the supports $\supp(u_{\gamma})$, $\gamma\subseteq\alpha$, are
 compact and in finite number.)
 Consequently, for
 $x\in\partial\alpha$, the set $$N_x:=\{x'\in U_{\delta(x)}\ /\ x'\notin\supp(u_\alpha)
\  \text{and}\ \delta(x)\subset\delta(x')\}$$
 is an open neighborhood
 of $x$ in $\alpha$.
  \medbreak
 Subclaim B:\emph{
 For every $x\in\partial\alpha$, $x'\in N_x$ and $y\in\D^q$,}
  $$\tilde\nabla(x',y)\in\tau_{(x',y)}\CC_{\delta(x)}$$
  
Indeed, for each proper face $\gamma\subset
 \partial\alpha$ such that $x'\in\supp(u_\gamma)$, one has $$
 \delta(x)\subset\delta(x')\subset\gamma$$
 hence,
  thanks to (vi) $$\tilde v_\gamma(x',y)\in\tau_{(x',y)}\CC_\gamma\subset
 \tau_{(x',y)}\CC_{\delta(x)}$$
 Subclaim B follows by definition of $\tilde\nabla$.
 
 \medbreak
 
 Verification of (5): Fix $x\in\beta$. There are two cases.
 
   In case $\delta(x)=\beta$, for every $x'\in N_x$, by definition of $N_x$,
    necessarily $\delta(x')=\beta$ also;
     hence by definition of $\tilde\nabla$ and $\tilde v_\beta$,
  for every $y\in\D^q$:
  $$\tilde\nabla(x',y)=\tilde v_\beta(x',y)\in\tau_{(x',y)}\CC_\beta$$
  Thus,  provided that $y$ is close enough to $A_{1/2}$ in $\D^q$, after (v):
  $$\tilde\nabla(x',y)\in\tau_{(x',y)}\GG$$
  
 Now, consider the second case $\delta(x)\neq\beta$. 
  On the one hand, by  Subclaim B and by (iv) applied to
 $\gamma=\delta(x)$, the vector field
 $\tilde\nabla$ is tangential to $G_{n-1}$ on $N_x\times\D^q$. On the other hand,
 by the induction property \ref{induction_ppt}  (2),
 $G_{n-1}$ coincides with $\GG$ on $\nb_{\alpha\times\D^q}(\beta\times A_{1/2})$.
 So, we also get (5) in that case.
 
Verification of (6): For every $x\in\hat\partial\alpha$,
 by Subclaim B and by (iv) applied to
 $\gamma:=\delta(x)$, the tangency property (6) holds on $N_x\times\D^q$.
 The proof of Claim \ref{crossing_clm} is complete. \end{proof}

 Let us resume the proof of Lemma  \ref{extension_lem}.
By (4) and (5) of Claim \ref{crossing_clm}, the vector field $\tilde\nabla$ is horizontal
on $\nb_{\alpha\times\D^q}(\alpha\times\partial\D^q)$
and on $\nb_{\alpha\times\D^q}(\beta\times A_{1/2})$.
By a fourth vertical isotopy relative to 
 $\nb_{\alpha\times\D^q}(\alpha\times\partial\D^q)$
and to $\nb_{\alpha\times\D^q}(\beta\times A_{1/2})$ and to
 $\beta\times\D^q$,
we can moreover arrange that
\begin{enumerate}[label=(\alph*)]
\item\label{hor_item} $\tilde\nabla$ is horizontal everywhere on $\alpha\times\D^q$.
\end{enumerate}

  Denoting   by $x_0$ the barycentric coordinate function on $\alpha$
  that restricts to $0$ on $\beta$ and to $1$ on the opposite vertex,
 fix   $\epsilon>0$ small enough that
\begin{enumerate}[label=(\alph*), resume]
\item\label{n_item} The domains $\alpha\times A_{1-\epsilon}$
and  $x_0\mun([0,\epsilon])\times A_{1/2-\epsilon}$ do not intersect
$C_{n-1}$; and on some neighborhood of these domains,
    the foliations $e_\alpha^*(\GG_{n-1})$ and $e_\alpha^*(\GG)$ are both defined,
    coincide, and are horizontal.

    \end{enumerate}

Since $\nabla$ enters $\alpha$ through $\beta$ (recall  (1) and (2) of Claim \ref{crossing_clm}),
 provided that $\epsilon$ is small enough,
 after rescaling $\nabla$ (and accordingly rescaling $\tilde\nabla$, so that (3) of
 Claim \ref{crossing_clm} still holds), we can moreover arrange that
    \begin{enumerate}[label=(\alph*), resume]
  \item\label{gradient_item}  The derivate   $\nabla\cdot x_0$ equals $\epsilon$
 on $x_0\mun[0,\epsilon]$. 
   \end{enumerate}
   
  Let $\pr_\nabla:\alpha\to\beta$ 
 denote the projection along the flowlines of $\nabla$. 
Decompose $\alpha$ into three subsets $\alpha'$, $\alpha''$, $\alpha'''$, and $\beta$ into
two subsets $\beta'$, $\beta''$, as follows (Figure \ref{simplex_fig}).
\begin{figure}
\includegraphics*[scale=0.35, angle=0]{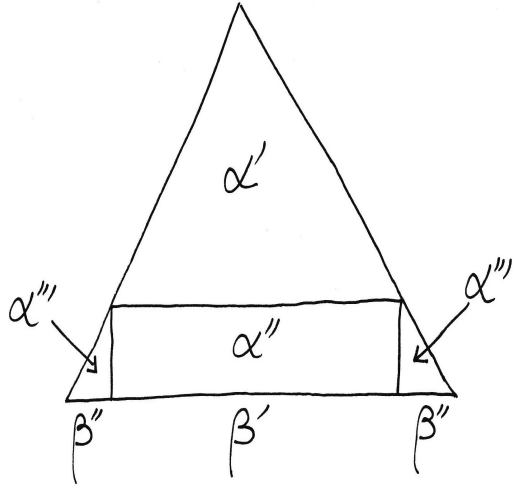}
\caption{Decomposition of the simplices $\alpha$ and $\beta$.}
\label{simplex_fig}
\end{figure}
\begin{itemize}
\item
$\alpha':=x_0\mun[\epsilon,1]$;
\item $\beta'=\pr_\nabla(\alpha')$;
 \item $\alpha'':=x_0\mun[0,\epsilon]\cap\pr_\nabla\mun(\beta')$;
 \item $\beta'':=$ topological closure of $\beta\setminus\beta'$;
\item $\alpha''':=\pr_\nabla\mun(\beta'')$
\end{itemize}

Fix a self-diffeomorphism
$v$ of $\I$, supported in the open interval $({1/2-\epsilon}, 1)$ and such that

\begin{enumerate}[label=(\alph*), resume]
\item\label{v_item} $v(1-\epsilon)<1/2$
\end{enumerate}
Define the self-diffeomorphism $\psi_1$ of $\D^q$ as
$$\psi_1:y\mapsto v(\vert y\vert)y/\vert y\vert$$
for $y\neq 0$, and $\psi_1(0)=0$.
Fix a $1$-parameter family
 $\psi:=(\psi_t)_{t\in\I}$ of self-diffeomorphisms
 of $\D^q$,
 supported in the interior of $A_{1/2-\epsilon}$, such that $\psi_t=\id$ for $t$
 close to $0$ and $\psi_t=\psi_1$ for $t$ close to $1$. Let $(Y_t)_{t\in\I}$
 be the $1$-parameter family of vector fields on $\D^q$ whose flow is $(\psi_t)$.
 Define a nonsingular vector field $V$
on $\alpha\times\D^q$ by:
 
\begin{enumerate}[label=(\alph*), resume]
\item\label{V1_item} $V(x,y):=\nabla(x)\oplus 0$ if
 $x\in\alpha'$;
\item\label{V2_item}
$V(x,y):=\nabla(x)\oplus -Y_{1-\epsilon\mun x_0(x)}(y)$ if $x\in\alpha''\cup\alpha'''$.
 \end{enumerate}

Note that
\begin{enumerate}[label=(\alph*), resume]
\item\label{V3_item} $V$ lifts $\nabla$ through $\pi_\alpha$.
 \end{enumerate}

The vector fields $\tilde\nabla$ and $V$ being both horizontal on
$\nb_\alpha(\hat\partial\alpha)\times\D^q$ (after \ref{hor_item}, \ref{V1_item} and \ref{V2_item}),
 and both of them lifting
$\nabla$ there (Claim \ref{crossing_clm} (3), \ref{V3_item}), let $U$ be
an open neighborhood
of $\hat\partial\alpha$
 in $\alpha$ so small that  $V=\tilde\nabla$ on $U\times\D^q$.
 Provided that $U$ is small enough, one has on $U\times\D^q$
 (after the induction property \ref{induction_ppt} (1) at the order $n-1$
 and after Claim \ref{crossing_clm} (6)):

\begin{enumerate} [label=(\alph*), resume]
\item\label{G1_item} $e_\alpha^*(G_{n-1})$ is quasi-complementary to $\FF_\alpha$
 on $U\times\D^q$;

\item\label{G3_item}
 $V$ is tangential to $e_\alpha^*(G_{n-1})$ on $U\times\D^q$. 
\end{enumerate}

 Recalling Claim \ref{crossing_clm} (1) and (2),
 let $\chi$ be on $\alpha$
a nonnegative real function such that 
\begin{itemize}
\item\label{stop_item} $\chi$ vanishes on
 a neighborhood of ${\hat\partial}\alpha$;

\item\label{embed_item}  The time $t=1$ of the flow $(\nabla'^t)$ of the vector field
 $\nabla':=\chi\nabla$
embeds $\alpha$ into $U$.
\end{itemize}
 Put $V':=(\chi\circ\pi_\alpha)V$ and consider the time $t=1$ of the flow $(V'^t)$.

Obviously, since $V'^1(\alpha\times\D^q)\subset U\times\D^q$ and since $V$ is tangential to $e_\alpha^*(G_{n-1})$ there
(\ref{V3_item}, \ref{G3_item}),
 the cleft foliation
 $e_\alpha^*(G_{n-1})$ is pullable (Tool \ref{pullback_sssec} (c))
  through $V'^1$. 
Consider on $\alpha\times\D^q$ the pullback cleft foliation
$$G^0:=(e_\alpha\circ V'^1)^*(G_{n-1})$$
and write $G^0=(C^0,[c^0],\GG^0)$. Put
$H:=\alpha'''\times A_{1-\epsilon}$.

\begin{clm}\label{G0_clm}\
\begin{enumerate}

\item\label{G01_item} The cleft foliation $G^0$ is quasi-complementary to $\FF_\alpha$
 on $\alpha\times\D^q$;
 
  \item\label{G00_item} The fissure $C^0$ does not intersect
$H$;
 
 \item\label{G02_item} The fissure $C^0$ does not intersect
$\alpha'\times A_{1-\epsilon}$ nor
$\alpha\times\partial\D^q$, and $\GG^0$ is horizontal
on  $\nb_{\alpha\times\D^q}(\alpha'\times A_{1-\epsilon})$ and on
$\nb_{\alpha\times\D^q}(\alpha\times\partial\D^q)$;

\item\label{G03_item}  $G^0$ and $e_\alpha^*(G_{n-1})$ 
coincide on
 $\nb_\alpha(\hat\partial\alpha)\times\D^q$; 
 
 \item\label{G04_item} $V$ is tangential to $G^0$ on $\alpha\times\D^q$;
 \item\label{G07_item} $G^0$ is invariant by the flow of $V$;

\item\label{G05_item} The fissure
 $C^0$ does not meet $\beta'\times  A_{1/2}$, and
 $\GG^0$ is horizontal 
  on
$\nb_{\alpha\times\D^q}(\beta'\times A_{1/2})$;

\item\label{G06_item} The foliation $e_\alpha^*(\CC_\beta)$ is tangential to $G^0$
 on $\nb_\alpha(\partial\beta)\times\D^q$.
 \end{enumerate}
 \end{clm}

\begin{proof}[Proof of the claim] (\ref{G01_item}): the quasi-complementarity that
holds \ref{G1_item} on $U\times\D^q$
 extends to the all of $\alpha\times\D^q$
 since, the vector field $V'$ being projectable \ref{V3_item} through
 $\pi_\alpha$ , its flow preserves the foliation
$\FF_\alpha$.

(\ref{G00_item}): Since by \ref{V1_item}, \ref{V2_item} the domain $H$
 is saturated for the flow of $V$, and since $e_\alpha^*(G_{n-1})$ is not cleft in $H$ 
 (recall \ref{n_item}).

(\ref{G02_item}): Clear by \ref{n_item}, \ref{V1_item} and \ref{V2_item}.

(\ref{G03_item}), (\ref{G04_item}) and  (\ref{G07_item}): after
 the very definition of $G^0$.

(\ref{G05_item}):
On the one hand, by (\ref{G02_item}) of the present Claim \ref{G0_clm},
for $\eta<\epsilon$ close enough to $\epsilon$,
in restriction to $x_0\mun(\eta)\times
 A_{1-\epsilon}$,
the foliation $G^0$ is not cleft, and $\GG^0$
is horizontal (i.e. parallel to $\beta)$ there.
On the other hand, by \ref{gradient_item}  and \ref{V2_item},
 the projection $\pr_V$ 
 of the hypersurface
 $x_0\mun(\eta)\times\D^q$ into $\beta\times\D^q$
 along the flowlines of $V$ is 
 $$
\pr_V: (x,y)\mapsto(\pr_\nabla(x),
\psi_1(y))$$
Since moreover $\GG^0$ is $V^t$-invariant (\ref{G07_item}), it follows that
 $\GG^0$ is horizontal in restriction to the image 
$\pr_V(x_0\mun(\eta)\times A_{1-\epsilon})$, which
 is an open neighborhood of $\beta'\times A_{1/2}$ in $\beta\times\D^q$, since  $v(1-\epsilon)<1/2$ \ref{v_item}.
 The horizontality holds in fact on $\nb_{\alpha\times\D^q}(\beta'\times A_{1/2})$, since
 $V$
is horizontal  on $\nb_{\alpha}(\beta)\times\D^q$ (in view of \ref{V1_item} and \ref{V2_item}).

(\ref{G06_item}): Fix any point $x\in\partial\beta$. Let
$\gamma$ be the smallest face of $\beta$ containing $x$. By Lemma \ref{civilization_lem}
(3), one has
$e_\alpha^*(\CC_\beta)\subset e_\alpha^*(\CC_\gamma)$ on $\nb_\alpha(x)\times
\D^q$. By the induction property \ref{induction_ppt} (3) at the order $n-1$, the civilization
 $ e_\alpha^*(\CC_\gamma)$ is tangential to
$e_\alpha^*(G_{n-1})$ on $\nb_\alpha(x)\times
\D^q$. By 
(\ref{G03_item}), the cleft foliations
$G_{n-1}$ and $G^0$ coincide on $\nb_\alpha(x)\times\D^q$.
\end{proof}

Note that $\GG^0$ is \emph{not} horizontal on $\beta''\times A_{1/2}$, in other words,
it does \emph{not} match $\GG_{n-1}$ there. This is why we need to modify $G^0$ in $H$,
 introducing a new fissure there.

Decompose $\partial\alpha''$ as $$\partial\alpha''=\partial_1\alpha''\cup\partial_2\alpha''\cup\beta'$$
where $$\partial_1\alpha'':=\alpha''\cap\alpha'$$ 
$$\partial_2\alpha'':=\alpha''\cap\alpha'''$$
 In a first time, we define a cleft foliation $G^1=(C^1,[c^1],\GG^1)$
  on $(\alpha\times\D^q)
\setminus H$, and we also define it as a germ along $\partial H$ inside $H$. Let
\begin{enumerate}[label=(\alph*), resume]
\item \label{G11_item} $G^1$ coincide with $G^0$ on  $(\alpha\times\D^q)
\setminus H$;
\item \label{G12_item} $G^1$ coincide with $G^0$ on $\nb_{\alpha\times\D^q}(\partial_1
\alpha''\times A_{1/2-\epsilon})$ and on  $\nb_{\alpha\times\D^q}
(\alpha''\times\partial  A_{1/2-\epsilon})$; in other words $\GG^1$ is horizontal there;
\item \label{G13_item} $G^1$ coincide with $G^0$ on  $\nb_{\alpha\times\D^q}(\partial_2
\alpha''\times A_{1/2-\epsilon})$;
\item \label{G14_item} $C^1$ don't intersect $\beta''\times A_{1/2-\epsilon}$, and $\GG^1$
be horizontal on  $\nb_{\alpha\times\D^q}(\beta''\times A_{1/2-\epsilon})$.
\end{enumerate}

 This turns $H$ into
a hole (Definition \ref{hole_dfn}) of core $\S^{i-1}$, 
 fibre $A_{1/2-\epsilon}$ and monodromy $\psi$ 
 (in view of \ref{V2_item}).

 \begin{clm}\label{filling_clm}
 Provided that we take some extra care in the choice of $\psi$, the cleft foliation $G^1$
extends over $H$, giving on  $\alpha\times\D^q$ a cleft foliation,
still denoted $G^1=(C^1, [c^1],\GG^1)$, of monodromy
 $\varphi$ and quasi-complementary to $\FF_\alpha$. Moreover, $C^1$ has
a component
interior to $H$, whose core is a finite disjoint union
of spheres $\S^{i-1}$, each of which
embeds into $Int(\alpha)$ through $\pi_\alpha$.
\end{clm}

\begin{proof} In every codimension $q\ge 2$:
consider the diffeomorphism $\psi'_1:=\psi_1\vert A_{1/2-\epsilon}$.
Fix
 an arbitrary equidimensional embedding of $\D^q$ in the interior of $A_{1/2-\epsilon}$.
 Thus, $\varphi$
 becomes a family of  self-diffeomorphisms of $A_{1/2-\epsilon}$.
  It is easy, in the above construction of $V$,
   to choose the diffeomorphism $v\vert[1/2-\epsilon,1]$ in
    the group $\diff([1/2-\epsilon,1])$ as a
product of commutators. Consequently,
  so is $\psi'_1$ in the group $\diff(A_{1/2-\epsilon})_0$.
 Since $\varphi_1$ is
not the identity, and since $A_{1/2-\epsilon}$ is connected,
by Epstein's perfectness theorem (Proposition 1.2 and Theorem 1.4 in \cite{epstein_70}), $\psi'_1$ belongs to the normal subgroup of 
$\diff(A_{1/2-\epsilon})_0$ generated by  $\varphi_1$;
hence some lift $[\psi']\in\diff(A_{1/2-\epsilon})^\I$ of $\psi'_1$
(recall Notation \ref{diff_ntt})
  lies in the normal subgroup of $\diff(A_{1/2-\epsilon})^\I$
   generated by  $\varphi$.
    In other words, one has in the group $\diff(A_{1/2-\epsilon})^\I$
    a splitting 
    $$\psi'=\chi_1\dots\chi_\ell$$ where $\chi_1$, \dots, $\chi_\ell$
are conjugate to $\varphi$ in $\diff(A_{1/2-\epsilon})^\I$.
 In the above construction of $V$,
we choose the family $(\psi_t)$ to represent this particular lift $\psi'$.

  Then, by Tool \ref{splitting_sssec} 
   and an obvious induction
  on $\ell$, we partially fill
   the hole $H$
  by a foliation complementary to $\FF_\alpha$ (thanks to the property (1)
  in \ref{splitting_sssec}), leaving some smaller
 holes $H_j$ ($1\le j\le\ell$) of core $\S^{i-1}$, fibre $A_{1/2-\epsilon}$, and respective monodromies
  $\chi_1$, \dots, $\chi_\ell$.
 Equivalently, each hole $H_j$ is of monodromy $\varphi$ (Tool \ref{reparametrization_sssec}).
 Then,  each hole $H_j$
 is vertically shrunk (Tool \ref{vertical_sssec}) into a hole $H'_j$
of core $\S^{i-1}$, fibre $\D^q$ and monodromy $\varphi$.
Finally (Tool \ref{fh_sssec}),
  each $H'_j$ is horizontally shrunk into a fissure of 
  core $\S^{i-1}$, fibre $\D^q$ and monodromy $\varphi$.

This core does embed into $Int(\alpha)$ through $\pi_\alpha$, after the note
that ends  Tool \ref{splitting_sssec}.

The case $q=1$ is alike, but we need \emph{two}
arbitrary equidimensional embeddings of $\D^1$ in the interior of $A_{1/2-\epsilon}\cong\S^0\times\I$:
one in each connected component.
\end{proof}

In view of Claim \ref{G0_clm} (\ref{G01_item}), (\ref{G02_item}),
 (\ref{G03_item}) and (\ref{G05_item}),
  and of \ref{G11_item}, \ref{G13_item} and \ref{G14_item},
the cleft foliation $G^1$ on $\alpha\times\D^q$ matches the properties
(1) through (4) of Lemma \ref{extension_lem} with $G^1$ instead of $G_\alpha$.
However (Claim \ref{G0_clm} (\ref{G06_item})), the property
 (5), namely the tangentiality of $\CC_\beta$ to $G^1$,
 only holds on $W\times\D^q$ for some open neighborhood $W$ of $\partial\beta$
in $\alpha$. 
To rectify this, consider on $\alpha$ the unit vector field $B$ normal to $\beta$ and that
enters $\alpha$ through $\beta$; consider on $\nb_\alpha(\beta)\times\D^q$ the lift
$\tilde B$ of $B$
tangentially to $e_\alpha^*(\CC_\beta)$ (Lemma \ref{civilization_lem}, (2)).
 Let $(u,1-u)$ be on $\alpha$ a partition of
the unity subordinate to the open cover $(W,\alpha\setminus\partial\beta)$.
 Define a vector field $Z$ on $\alpha$
 (resp. $\tilde Z$ on $\alpha\times\D^q$) by $$Z:=(1-u)\nabla+u B$$
$$\tilde Z:=(1-u\circ\pi_\alpha)V+(u\circ\pi_\alpha)\tilde B$$

\begin{clm}\label{Z_clm}
 $\tilde Z$ is tangential to
$G^1$ on  $\nb_\alpha(\beta)\times\D^q$.
\end{clm}
\begin{proof} On the one hand,
$V$ is tangential to
$G^1$ on $\nb_\alpha(\beta)\times\D^q$: indeed, outside $H$, this follows from 
(\ref{G04_item}) of Claim \ref{G0_clm} together with \ref{G11_item};
 while on $\nb_{\alpha\times\D^q}
(\beta''\times A_{1/2-\epsilon})$, both $V$ and $\GG^1$ are horizontal
(\ref{V1_item} and \ref{G14_item}).
On the other hand,
 $\tilde B$ is tangential to $G^1$ on $U\times\D^q$, by the above definition of $U$.
\end{proof}

Since $B$ and $Z$
both enter $\alpha$ through $\beta$ and coincide on $\nb_\alpha(\partial\beta)$,
 there is an isotopy $f=(f_t)_{t\in\I}$ of self-diffeomorphisms of $\alpha$
 (with   $f_0=\id_\alpha$) such that
 for every $t\in\I$:
 \begin{itemize}
 \item  The support of $f_t$ is contained in a small neighborhood
 of $\beta$ and disjoint from $\hat\partial\alpha$;
 \item $f_t$ induces the identity on $\beta$;
 \item
 $f_t^*(Z)=(1-t)Z+tB$ on $\nb_\alpha(\beta)$.
 \end{itemize}
 
\begin{clm}\label{lift_clm} There is an isotopy $\tilde f=(\tilde f_t)_{t\in\I}$
of self-diffeomorphisms of $\alpha\times\D^q$ (with $f_0=\id_{\alpha\times\D^q}$)
 such that for every $t\in\I$:
\begin{enumerate}
\item \label{lift1_item} $\pi_\alpha\circ\tilde f_t=f_t\circ\pi_\alpha$;

    \item \label{lift2_item} The vector field
     $\partial\tilde f_t/\partial t$ is at the time $t$
      horizontal, i.e. tangential to $e_\alpha^*(\GG)$,
  on $\alpha\times\nb_{\D^q}(\partial\D^q)$ and on $
  \nb_{\alpha\times\D^q}(\beta\times A_{1/2})$;
  \item \label{lift3_item} $\tilde f_t$ is supported in $(\alpha\setminus\hat\partial\alpha)
  \times\D^q$;
\item \label{lift4_item} $\tilde f_t$ induces the identity on $\beta$;
  \item \label{lift5_item} $\tilde f_t^*(\tilde Z)=(1-t)\tilde Z+t\tilde B$ on $\nb_\alpha(\beta)\times\D^q$.
  \end{enumerate}
  \end{clm}
   \begin{proof} On $\nb_\alpha(\beta)\times\D^q$, the isotopy $\tilde f$ is uniquely
   determined by (\ref{lift4_item}) and (\ref{lift5_item}). There,
   (\ref{lift1_item}) is satisfied, since $\tilde Z$ and $\tilde B$ respectively lift $Z$ and $B$ through $\pi_\alpha$. Also, on $\nb_\alpha(\beta)\times\D^q$,  (\ref{lift2_item}) is satisfied, since
   $V$, $\tilde B$ and hence $\tilde Z$ are horizontal there
     (after \ref{V1_item}, \ref{V2_item}, Lemma \ref{civilization_lem} (1) applied to $\beta$,
     and after (III)). Also, on $\nb_\alpha(\beta)\times\D^q$,  (\ref{lift3_item}) is satisfied,
     since $\tilde Z=\tilde B$ on $\nb_\alpha(\partial\beta)$. Finally,
     by the standard extension
     property for vertical isotopies, $\tilde f$
     admits an extension
     to the all of $\alpha\times\D^q$ which obeys
     the conditions  (\ref{lift1_item}),  (\ref{lift2_item})
     and  (\ref{lift3_item}).
   \end{proof}

  Define the cleft foliation $G_\alpha$ on $\alpha\times\D^q$ as 
  $\tilde f_1^*(G^1)$; let us verify that the conditions (1) through (5) of Lemma \ref{extension_lem}
  are obeyed.
  
  (1): by Claim \ref{filling_clm} and by Claim \ref{lift_clm} (\ref{lift1_item});
  
  (2): by (I),
  by Claim \ref{G0_clm} (\ref{G02_item}), by \ref{G11_item}, by \ref{G12_item}, and
  by Claim \ref{lift_clm} (\ref{lift2_item});
  
  (3): by Claim \ref{G0_clm} (\ref{G03_item}), by \ref{G11_item}, by \ref{G12_item}, and
  by Claim \ref{lift_clm} (\ref{lift3_item});
  
  (4): by Claim \ref{G0_clm} (\ref{G05_item}), by \ref{G11_item}, by \ref{G14_item},
  and by Claim \ref{lift_clm} (\ref{lift2_item});
  
  (5): by Claim \ref{Z_clm} and by Claim \ref{lift_clm} (\ref{lift5_item})
  applied for $t=1$.

The lemma  \ref{extension_lem} is proved. \end{proof}

Recall that the notations $\alpha, \beta$, $i$ hold for $\alpha_n$, $\beta_n$,
$\dim(\beta_n)$.

It remains to us, in order to complete the inflation induction, to microextend $G_\alpha$
from $\alpha\times\D^q\cong P(\alpha)$
to $\nb_{\bar M}(P(\alpha))$ so as to match $G_{n-1}$ on the intersection
with
$\nb_{\bar M}(R\cup\bar M_{n-1})$, and to satisfy the
properties \ref{induction_ppt} (1) through (5).
 To this end, we use the civilizations of the prisms.

For every cell $\gamma$ of $K$,
denote by $\langle\gamma\rangle\subset\R^{p+1}$ the
 affine subspace spanned by $\gamma$,
 and consider the
local retraction along the leaves of $\CC_\gamma$:
$$\tilde\pr_\gamma:\nb_{\bar M}(P(\gamma))\to
\nb_{\langle\gamma\rangle\times\D^q}(P(\gamma))$$
Actually, the object of interest is the germ of $\tilde\pr_\gamma$ along
$P(\gamma)$. 
 Also, denote by  $\pr_\gamma$ the orthogonal projection
   $\R^{p+1}\to\langle\gamma\rangle$. Obviously, by Lemma \ref{civilization_lem}:
   
 \begin{cor}\label{civilization_cor}\
\begin{enumerate}
\item $\GG=\tilde\pr_\gamma^*(\GG\vert(\langle\gamma\rangle\times\D^q))$
on $\nb_{\bar M}(e_\gamma(\gamma\times\partial\D^q))$;
\item $\bar\pi\circ\tilde\pr_\gamma=\pr_\gamma\circ\bar\pi$ on $\nb_{\bar M}(P(\gamma))$;
\item $\tilde\pr_\delta\circ\tilde\pr_\gamma=\tilde\pr_\delta$
on $\nb_{\bar M}(P(\delta))$, for every face $\delta\subset\gamma$.
 \end{enumerate}
  \end{cor}
  Also,  Lemma \ref{pullable_lem} admits the following complement.
  \begin{lem}\label{pullback_lem} For each cell $\gamma$ of $K$,
  the cleft foliation
  $G_{n-1}$ coincides with the pullback
   $\tilde\pr_\gamma^*(G_{n-1}\vert P(\gamma))$
  on $$\nb_{\bar M}(Int(P(\gamma))\cap
  \nb_{\bar M}\big((R\cup\bar M_{n-1})\cap P(\gamma)\big)$$
  \end{lem}
  \begin{proof} Consider any point $(x,y)\in (R\cup\bar M_{n-1})\cap P(\gamma)$.
  Let $\delta\subset\gamma$ be the smallest face containing $x$.
  We claim that  $\CC_\delta$ is tangential to $G_{n-1}$ on $\nb_{\bar M}(x,y)$.

   Indeed,
  $\delta\subset K_{n-1}$ or $y\in R(\delta)$.
   In the case where
  $\delta\subset K_{n-1}$,
  the claim follows from the induction property \ref{induction_ppt} (3) applied to $\delta$
  at the order $n-1$.
    In the case where $y\in R(\delta)$, on $\nb_{\bar M}(x,y)$,
   the foliation $G_{n-1}$ coincides with $\GG$
  (induction property \ref{induction_ppt} (2) at the order $n-1$), while $\CC_\delta\subset\GG$ (after Lemma \ref{civilization_lem}
  (1)). The claim is proved.

   In view of (3) of Lemma
  \ref{civilization_lem}, $\CC_\gamma$ is also tangential to $G_{n-1}$ on
  $\nb_{\bar M}(x,y)$. Since
  $P(\gamma)$ is transverse to $\CC_\gamma$,
    the wanted coincidence follows on 
    $$\nb_{\bar M}(Int(P(\gamma))\cap
  \nb_{\bar M}(x,y)$$ (see Tool \ref{pullback_sssec} and Remark
  \ref{tangent2_rmk}).
  \end{proof}

\begin{clm}\label{compatible_clm}\
\begin{enumerate}
\item $G_{n-1}$ coincides with $(e_\alpha\mun\circ\tilde\pr_\alpha)^*(G_\alpha)$
on $$\nb_{\bar M}(R\cup\bar M_{n-1})\cap\nb_{\bar M}(Int(P(\alpha)))$$
\item $G_{n-1}$ coincides with $(e_\alpha\mun\circ\tilde\pr_\beta)^*(G_\alpha)$
on $$\nb_{\bar M}(R\cup\bar M_{n-1})\cap\nb_{\bar M}(Int(P(\beta)))$$
\item $(e_\alpha\mun\circ\tilde\pr_\alpha)^*(G_\alpha)$ coincides
 with $(e_\alpha\mun\circ\tilde\pr_\beta)^*(G_\alpha)$
on $$\nb_{\bar M}(Int(P(\alpha)))\cap\nb_{\bar M}(Int(P(\beta)))$$
\end{enumerate}
\end{clm}
\begin{proof}
(1): One has $$e_\alpha\mun(R\cup\bar M_{n-1})=(\alpha\times\partial\D^q)
\cup(\hat\partial\alpha\times\D^q)\cup(\beta\times A_{1/2})$$
After Lemma \ref{extension_lem} (2) (3) (4), on some open neighborhood of this subset in $\alpha\times\D^q$,
one has $G_\alpha=e_\alpha^*(G_{n-1})$. Hence, the claim (1) amounts to
Lemma \ref{pullback_lem} applied to $\alpha$.

(2): One has $$e_\alpha\mun(R\cup\bar M_{n-1})\cap(\beta\times
D_{1/2})=\partial(\beta\times D_{1/2})$$
After Lemma \ref{extension_lem} (3) (4), on some open neighborhood of this subset in $\alpha\times\D^q$,
one has $G_\alpha=e_\alpha^*(G_{n-1})$. After Lemma \ref{extension_lem} (5),
 this cleft foliation is restrictable
to $\beta\times D_{1/2}$. Hence, the claim (2) amounts to
Lemma \ref{pullback_lem} applied to $\beta$.

(3): after Lemma \ref{extension_lem} (5), on $\nb_{\alpha\times\D^q}
(\beta\times\D^q)$, the cleft foliation $G_\alpha$ coincides with the pullback of
$G_\alpha\vert(\beta\times\D^q)$ through the local retraction of
$\nb_{\alpha\times\D^q}
(\beta\times\D^q)$ onto $\beta\times\D^q$ along the leaves of
$e_\alpha^*(\CC_\beta)$. Hence, the claim (3) follows from
  Corollary \ref{civilization_cor} (3) applied to the pair $\beta\subset\alpha$.
\end{proof}

Globally, we have got a cleft foliation $G_n$ on $\nb_{\bar M}(R\cup\bar M_{n})$ such that
\begin{enumerate}[label=(\roman*)]
\item \label{Gn_item}  $G_n=G_{n-1}$ on $\nb_{\bar M}(R\cup\bar M_{n-1})$;

\item  \label{Gn_item} $G_n=(e_\alpha\mun\circ\tilde\pr_\alpha)^*(G_\alpha)$ on
$\nb_{\bar M}(Int(P(\alpha)))$;

\item \label{Gn3_item} $G_n=(e_\alpha\mun\circ\tilde\pr_\beta)^*(G_\beta)$ on
 $\nb_{\bar M}(Int(P(\beta)))$.

\end{enumerate}
We are left to verify that $G_n$ matches the properties \ref{induction_ppt} (1) through (5).

(1): after Tool \ref{pullback_sssec} (e);

(2): after the very definition \ref{Gn_item} above, and after 
the induction property \ref{induction_ppt} (2)
 at the order $n-1$;

(3): the only two cells to consider are $\alpha$ and $\beta$. The civilization
$\CC_\alpha$ (resp. $\CC_\beta$) is indeed tangential to $G_n$ on
$\nb_{\bar M}(Int(P(\alpha)))$ (resp. $\nb_{\bar M}(Int(P(\beta)))$
by the very definition \ref{Gn_item} (resp. \ref{Gn3_item}), and on 
$\nb_{\bar M}(R(\alpha))$ because $G_n$ coincides with $\GG$ there,
and because of Lemma \ref{civilization_lem} (1);

(4): Let us recollect the construction of the core $\Sigma_n$ from $\Sigma_{n-1}$.
 Recall or agree
 that  $G_{n-1}=(C_{n-1},[c_{n-1}],\GG_{n-1})$,  $G^0=(C^0,[c^0],\GG^0)$,
 $G^1=(C^1,[c^1],\GG^1)$,
 $G_\alpha=(C_\alpha,[c_\alpha],\GG_\alpha)$,  $G_n=(C_n, [c_n],\GG_n)$.
  Consider the cores  $\Sigma_{n-1}:=c_{n-1}\mun(0,0)$,   $\Sigma^0:=(c^0)\mun(0,0)$, 
 $\Sigma^1:=(c^1)\mun(0,0)$,  $\Sigma_\alpha:=c_\alpha\mun(0,0)$,    $\Sigma_n:=c_n\mun(0,0)$.
Put $$\hat\Sigma:=\Sigma_{n-1}\cap(\hat\partial\alpha\times\D^q)$$

By Claim \ref{G0_clm} (\ref{G04_item}), $\Sigma^0$ is the saturation of
$e_\alpha\mun(\hat\Sigma)$
by the flow of $V$. Hence, it is easy to extend
the complex $e_\alpha^*(\Delta_{n-1})$, which is a subdivision of $e_\alpha\mun(\Sigma^0)$,
by a smooth triangulation $\Delta^0$ of
 $\Sigma^0$ collapsing on  $e_\alpha^*(\Delta_{n-1})$.
 
  Also
(Claim \ref{filling_clm}), $\Sigma^1$ is nothing but the disjoint union
of  $\Sigma^0$ with the core $\Sigma^1\cap H$
 of the new fissure in the hole, which is a finite disjoint union of $(i-1)$-spheres
 disjoint from $\hat\partial\alpha\times\D^q$. Let $\Delta_H$ be an arbitrary smooth
 triangulation of $\Sigma^1\cap H$.
 
 Also, recalling the definition of $G^1$, the core $\Sigma_\alpha=\tilde f_ 1\mun(\Sigma^1)$
is isotopic with $\Sigma^1$ in $\alpha\times\D^q$ 
 rel. $\hat\partial\alpha\times\D^q$
(Claim \ref{lift_clm}, (\ref{lift3_item})).

   By the above definition (i), (ii), (iii) of $G_n$, one has
   \begin{equation}\label{core_eqn}
\Sigma_n\cap\bar M_n=(\Sigma_{n-1}\cap\bar M_{n-1})
\cup_{\hat\Sigma}e_\alpha(\Sigma_\alpha)
\end{equation}

 We define a
 partial smooth triangulation $\Delta_n$ of $\Sigma_n$ as the
 union of $\Delta_{n-1}$ with $e_\alpha(\tilde f_1\mun(\Delta_H))$ and $e_\alpha
 (\tilde f_1\mun(\Delta^0))$.
 The property \ref{induction_ppt} (4) is verified for $G_n$.
 
 (5): will follow
  as well from the preceding recollection of the successive constructions of the cores.
 Briefly, by the induction property \ref{induction_ppt} (5) at the order $n-1$, the projection $\bar\pi$
 is one-to-one on every connected component of
  $\Sigma_{n-1}\cap\bar M_{n-1}$, and thus, in particular, of $\hat\Sigma$.
 By \ref{V3_item}, the projection
  $\bar\pi$ is one-to-one on every connected component of $\Sigma^0$.
 On the other hand, in view of
  the construction in the proof of Claim \ref{filling_clm},
 each $(i-1)$-sphere composing $\Sigma^1\cap H$ embeds into $Int(\alpha)$ through 
 $\bar\pi$ (and the image is an embedded
  $(i-1)$-sphere, close to $\partial\beta$ for the Hausdorff distance). Hence,
   $\bar\pi$ is one-to-one on every connected component of $\Sigma^1$.
  Since the isotopy
   $\tilde f_1$ is projectable ((\ref{lift1_item}) of Claim \ref{lift_clm}), the projection
   $\bar\pi$ is one-to-one on every connected component of $\Sigma_\alpha$.
  Finally, after Equation (\ref{core_eqn}), $\bar\pi$
 is one-to-one on every connected component of
  $\Sigma_n\cap\bar M_n$.
 
This completes the induction on $n$, the proof of Proposition \ref{inflation_pro}, and the proof of Theorem \emph{A'}.
\subsection{Topology of the cores}
We now give a complement to Theorems \emph{A} and  \emph{A'} by making explicit the
topologies of the cores of the fissures and of the multifold Reeb components
 resulting from the above proofs (Figure \ref{propagation_fig}).

\begin{figure}
\includegraphics*[scale=0.45, angle=0]{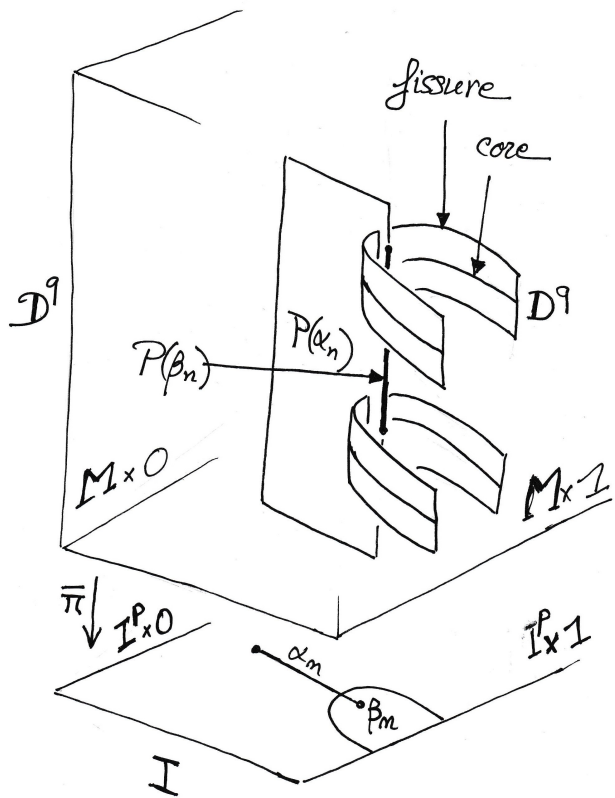}
\caption{Schematic view on a fissure component
which appears while foliating $P(\alpha_n)$ (Claim \ref{filling_clm}),
 on its propagation during the rest of the induction, and on its projection
 into the base $\I^{p+1}$.
 Beware that this schematic low-dimensional drawing can be misleading in several ways:
actually, there appears a fissure while foliating $P(\alpha_n)$ only for 
$i:=\dim(\beta_n)\ge 1$;
in the topology of the core, here the factor
$\S^{i-1}$ has been assimilated to two points;  in the base $\I^{p+1}$, the projection of
the core is actually of codimension $2$, not $1$.}
\label{propagation_fig}
\end{figure}

Recall the notations from  \emph{A'}.
\begin{pro}\label{core_pro} (a)
The core $\Sigma:=c\mun(0,0)$ of
the cleft $\Gamma_q$-structure $\Gamma$
built in the above proof of Theorem  \emph{A'} is diffeomorphic with a disjoint union of products 
$\S^{i-1}\times\D^{p-i}$, where $p:=\dim(M)-q$ and where $1\le i\le p$.

In particular,
the core $\Sigma\cap(M\times 1)=\partial\Sigma$
 of the cleft foliation $\Gamma\vert(M\times 1)$
  is diffeomorphic with a disjoint union of products 
$\S^{i-1}\times\S^{p-i-1}$, where $1\le i\le p-1$.

\end{pro}

\begin{proof} By the above property \ref{induction_ppt} (4) for $n=N$, the
smooth compact manifold $\Sigma=\Sigma_N$
admits a smooth (global) triangulation $\Delta_N$ collapsing on
a system of spheres (Vocabulary \ref{system_voc})
 interior to $\Sigma$. After excising an open
  tubular neighborhood of these spheres,
and after
 an obvious restoration of the triangulation close to the new boundary components,
  one is reduced to the well-known fact that every collapsing cobordism is trivial (see
Proposition \ref{cc_pro} below).
\end{proof}

Now, recall the notations from \emph{A}.
\begin{cor}\label{core_cor}
The core of the multifold Reeb component of the foliation $\gamma'$
built in the above proof
of Theorem \emph{A}  is diffeomorphic with a disjoint union of products 
$\S^{i-1}\times\S^{p-i-1}$, where $1\le i\le p-1$.
\end{cor}

\section{Proof of the Mather-Thurston Theorem as a corollary of \emph{A'}}
\label{MT_sec}

  Let $q\ge 1$. Recalling Notation \ref{diff_ntt},
  fix an element   $\varphi\in\diff(\D^q)^\I$ which is
   a product of commutators:
  $$\varphi=[\alpha_1,\beta_1] \dots
[\alpha_g,\beta_g]$$
with  $\alpha_1,\beta_1, \dots, 
\alpha_g,\beta_g\in\diff(\D^q)^\I$;  and such
that the image $\varphi_1\in\diff(\D^q)$ is not the identity.

\begin{rmk}
In fact, every  $\varphi\in\diff(\D^q)^\I$ is of this form, but we don't need this difficult perfectness result, for which one can find references and a discussion in \cite{mitsumatsu_vogt_16}. Moreover, in view of the above remark
\ref{differentiability_rmk}, recall that it remains unknown if perfectness holds
 in the differentiability class $C^{q+1}$.
\end{rmk}

The interest for us of such a commutators decomposition is that it allows
 an (obvious) suspension construction over a surface.
Namely,
let  $S_g$ be the compact orientable surface bounded by $\S^1$
and of genus $g$, with a basepoint on the boundary. Consider the representation $$\rho:\pi_1(S_g)\to\diff(\D^q)^\I$$
 mapping the canonical free basis to $\alpha_1,\beta_1,
\dots, \alpha_g,\beta_g$; hence $\rho(\partial S_g)=\varphi$:
the suspension $\SS_\varphi$ of $\varphi$ over the circle $\partial S_g$
bounds the codimension-$q$ foliation $\SS_\rho$ on $S_g\times\D^q$ that is the
 suspension of $\rho$.
\medbreak

\begin{proof}[Proof of Theorem  \ref{MT_thm}]
Let $X$, $V$, $\FF_V$, $\gamma$ be as in the hypotheses of Theorem \ref{MT_thm};
 to fix ideas, assume moreover that $X$ is closed.
 Recall that $\bar V:=V\times\I$,
  that $\hat V\subset\bar V$ is the union of $V\times 0$ with
 $\partial V\times\I$, that
 $\hat\pr$ is the restriction to $\hat V\times X$
 of the projection
$V\times\I\times X\to V\times X$; and let $\pi_V:V\times X\to V$ denote
 the first projection.

The hypothesis that $\gamma$
restricts to a foliated product over $\partial V$ amounts to say that
at every point $a\in\partial V\times X$, the differential
 $d_a\gamma$
is a linear retraction of $\tau_a(V\times X)$ onto
 $\nu_a\gamma=\tau_a\FF_V$.
 Hence, the theorem $A'$ provides a \emph{cleft} $\Gamma_q$-structure
on $\bar V\times X$ which coincides with $\hat\pr^*(\gamma)$ over $\nb_{\bar V}(\hat V)$
and which is, over $V\times 1\cong V$, complementary to $\FF_V$
\emph{on the complement of the fissure.}
Our proof of the Mather-Thurston theorem will consist in enlarging
the fissure into a hole and filling the hole by a foliation
complementary to the fibres, at the price of changing the base $\bar V$
into a more general cobordism. To this end,
the main tool
will be the pullback construction that has been described
in paragraph \ref{pullback_sssec} above.
A difficulty
is that  in $\bar V$, the projections of the connected
 components of the fissure will in general intersect
each other. To solve this difficulty,
 we need an inductive process,
 filling at each step only a union of components whose projections are two by two disjoint,
 and performing successive surgeries on the base. During the process, the cores of
 the remaining components will get surgerized too.
 \medbreak

More precisely,
one can paraphrase the conclusions
of Theorem \emph{A'} by the following properties (1), (2), (3), (I), (II), (III), (IV),
 where $V^*:=V$, where $W:=V\times\I$; and where
$\Sigma_1,\dots,\Sigma_\ell$ are the connected components of $\Sigma$.
One has
  
 \begin{enumerate}

    \item An oriented cobordism $(V,W,V^*)$  rel. $\partial V$;
\item On $W\times X$, a cleft $\Gamma_q$-structure $\Gamma=(C,[c],\bar\gamma)$
 whose monodromy is $\varphi$,
  and whose  normal bundle is $\tau\FF_W$;
\item A partition of $\Sigma:=c\mun(0,0)$ as a finite disjoint sum of compact components
   $$\Sigma=\Sigma_1\sqcup\dots\sqcup\Sigma_\ell$$
   (each $\Sigma_i$ is thus a union of connected components of $\Sigma$)
\end{enumerate}
such that
    \begin{enumerate}[label=(\Roman*)]
      \item $C$ is disjoint from $\hat V\times X$,
  and $\bar\gamma$  coincides in restriction to $\hat V\times X$ with $\hat\pr^*(\gamma)$;
  \item $C$ is in standard position with respect to $\FF_W$ (Definition \ref{position_dfn});
     \item $\Gamma\vert(V^*\times X)$  is on $V^*\times X$
      a cleft foliation quasi-complementary to $\FF_{V^*}$
(recall Definitions \ref{holed_foliation_dfn} and \ref{qc_dfn});
   \item The projection $\pi_W$
    is one-to-one in restriction to each of the components $\Sigma_1$, \dots, $\Sigma_\ell$.
 
\end{enumerate}
(Recall that (after
the remarks following the above
 definitions \ref{fissure_dfn} and \ref{holed_dfn})
  in case $\dim(W)=1$, a ``cleft $\Gamma_q$-structure'' has no cleft
at all: $C$ and $\Sigma$ are empty, and the proof of  Theorem \ref{MT_thm}
is complete. From now on, we assume that $\dim(W)\ge 2$.)

Let us describe how such data can be modified so that the 
  length $\ell$ of the partition drops by one.

Fix a smooth map
$\alpha:S_g\to\D^2$ inducing the identity between the boundaries.
Let $H=h(\Sigma\times\D^2\times\D^q)$ be a thin neighborhood of $C$ in $W\times X$
 as in Lemma \ref{complete_lem}.
By (III), provided that we choose $H$ thin enough,
 the equidimensional immersion
$$j:\Sigma_\ell\times\D^2\to W:(\sigma,z)\mapsto\pi_W(h(\sigma,z,0))$$
is an embedding. Let $W'$ (resp. $(V^*)'$) be the $(p+1)$-fold (resp. $p$-fold)
obtained from $W$ (resp. $V^*$) by cutting $j(\Sigma_\ell\times\D^2)$
(resp.  $j(\partial\Sigma_\ell\times\D^2)$) and pasting
$\Sigma_\ell\times S_g$ (resp. $\partial\Sigma_\ell\times S_g$). The triple
$(V,W',(V^*)')$ is
 a new oriented cobordism  rel. $\partial V$.
Let $a:W'\to W$ be the smooth map
defined on $\Sigma_\ell\times S_g\subset W'$ as
$$(\sigma,x)\mapsto j(\sigma,\alpha(x))$$
 and as the identity on the complement.
Consider the
product map $$A:=
a\times\id_X:W'\times X\to W\times X$$

The map $A$
being leafwise etale (Tool \ref{pullback_sssec} (a)) with respect to
 $\FF_{W'}$ and $\FF_W$, according to
 Lemma \ref{generic_lem}, after a generic small
  horizontal perturbation of the fissure $C$, one can arrange that $C$ is pullable through $A$.
  The above properties (I), (II) and (III) still hold (Tool \ref{reparametrizing_sssec});
  we choose the perturbation small enough that (IV) also still holds.

On the one hand,
on $A\mun(H)\cong\Sigma_\ell\times S_g\times\D^q$, define
$\bar\gamma'$ as
 $(\pr_2\times\pr_3)^*(\SS_\rho)$: a genuine foliation
complementary to $\FF_{W'}$ there.
On the other hand, consider the complement $R:=(W\times X)\setminus
Int(H)$ and, on $A\mun(R)$, the cleft $\Gamma_q$-structure
$A\mun(\Gamma\vert R)$.

The two constructions obviously coincide on $\nb_{W'\times X}(A\mun(\partial H))$,
thus defining  on $W'\times X$ a global cleft $\Gamma_q$-structure $\Gamma'$.
  The above properties (I), (II) and (III) hold for $W'$, $(V^*)'$, $\Gamma'$
  instead of $W$, $(V^*)$, $\Gamma$ (after Tool \ref{pullback_sssec} (c) and (e)).
As for (IV), the core of the fissure of $\Gamma'$ is the disjoint union of
$\Sigma'_i:=A\mun(\Sigma_i)$ for  $1\le i\le\ell-1$. Since for each $i$ the projection
 $\pi_W$ is one-to-one on the component $\Sigma_i$,
clearly, $\pi_{W'}$ is one-to-one on $\Sigma'_i$.

After the last step of the inductive process, $\ell=0$; in other
words,
 there is no fissure any more, and Theorem \ref{MT_thm} is proved.
 \end{proof}
 
 \section{Appendix: the ``collapsing cobordism theorem''}\label{cc_sec}
 
 \begin{pro}\label{cc_pro} Let $(V_0,W,V_1)$ be a cobordism between
 two
closed manifolds $V_0$, $V_1$. Assume that $W$ admits a smooth triangulation $K$
collapsing onto $V_0$. Then, $W$ is diffeomorphic with $V_0\times\I$.
\end{pro}

This fact, which Thurston used (\cite{thurston_76}, last lines of
paragraph 7)
 for the same purpose as we do in the present paper, is of
 course covered by the classical theorems of PL topology (I thank Larry Siebenmann
 for pointing this out to me).
  Indeed, by the normal neighborhood theorem
 \cite{hudson_69},
  $K$ is PL-isomorphic with the product $(K\vert V_0)\times\I$.
Then, by the product smoothing theorem (e.g. \cite{hirsch_mazur_74}), $W$ is diffeomorphic with $V_0\times\I$.

If $\dim(W)\ge 6$, alternatively,
a collapsing cobordism being a $s$-cobordism, 
one can apply the $s$-cobordism theorem. This second argument, which does not use the
smoothness of the triangulation, is restricted to the large dimensions.

 Both these arguments involving some rather elaborate results,
  we give a direct, comparatively
  short, and pedestrian proof.
 (We also believe that it is worth giving because the product smoothing theorem
 is
  a particular case of Proposition \ref{cc_pro}: the case where $K$ is a Whitney product
 triangulation).
 
  Our proof consists in building
 a smooth Morse function
on $W$ whose critical points are exactly the barycenters of
 the cells of $K$,
and such that for each (or most)
 of the elementary collapses composing the given global collapse,
 the corresponding
 pair of critical points is cancellable. After the cancellations, one obtains a
noncritical smooth function, and thus a diffeomorphism of $W$ with $V_0\times\I$.
The only subtleties in the construction lie of course in some smoothness topics.

\medbreak
We begin with some elementary tools for handling smooth triangulations.
A simplex of dimension at least $2$ being a smooth manifold with \emph{cornered}
 boundary, the (short) proofs are given.

     For $0\le p\le n$, endow $\R^n$
      with the euclidian coordinates $x_1$, \dots, $x_n$; 
    define the sector $\R^n_p\subset\R^n$ by $x_i\ge 0$ ($n-p+1\le i\le n$);
    and consider for each $n-p+1\le i\le n$
     the hyperface $\eta_i\subset\R^n_p$ defined by $x_i=0$.
 \begin{lem}\label{sector_lem} 
  Every real function $f$ on $\partial{\R^n_p}$ whose restriction
  to each hyperface $\eta_i$ is smooth, extends to a smooth real function on
   ${\R^n_p}$. Moreover, this extension process can be made continuous with respect
   to the smooth topologies on spaces of smooth functions.
\end{lem}
   \begin{proof} There is indeed a simple formula, playing with the projections to the faces.
   Let  $I$ be the finite set $n-p+1$, \dots, $n$; for every subset $J\subset I$,
    let $\eta_J\subset{\R^n_p}$ be  the face defined
by $x_i=0$ for every $i\in J$; and let  $$\pr_J:{\R^n_p}\to \eta_J$$ be the orthogonal projection.
It is enough to verify that on $\partial{\R^n_p}$:
  \begin{equation}\label{J_eqn}
  \sum_{J\subset I}(-1)^{\vert J\vert}f\circ\pr_J=0
  \end{equation}
  (Indeed, this equation amounts to express the term $f$, corresponding to $J=\emptyset$,
  as minus the sum of the other terms, which obviously extend smoothly to $\R^n_p$).
  
Fixing $i\in I$, one verifies (\ref{J_eqn})
  on $\eta_i$ by splitting the set  of subsets $\PP(I)$ into pairs $J$,
$J'$ where $i\notin J$ and $J'=J\cup\{i\}$.
 On $\eta_i$, one has $\pr_J=\pr_{J'}$, hence the terms of (\ref{J_eqn})
 cancel by pairs. 
 \end{proof}

 \begin{lem}\label{fe_lem}
  Given a simplicial complex $K$
  and a subcomplex $L\subset K$, every real simplexwise smooth function $f$ on $L$
  extends to a real simplexwise smooth function $F$ on $K$.
   Moreover, the support of $F$ can be contained in any neighborhood of
   the support of $f$; and
   the extension process can be made continuous with respect
   to the smooth topologies on spaces of smooth functions.
   \end{lem}
   
   \begin{proof} By induction on the number of cells of $K$ not in $L$,
   one is reduced to the case where $K$ is a single cell $\alpha$, and
   $L=\partial\alpha$. Every point $x\in\partial\alpha$ admits in $\alpha$
    an open neighborhood $U_x$
   diffeomorphic to a sector. By Lemma \ref{sector_lem},
    a smooth extension $F_x$ of $f\vert(U_x\cap\partial\alpha)$ exists on $U_x$.
    Pick any mathematical object $*\notin\partial\alpha$;
    put  $X:=\partial\alpha\cup\{*\}$ and $U_*:=Int(\alpha)$; let
     $(u_x)$ ($x\in X$) be a smooth
     partition of the unity on $\alpha$ subordinate to the open cover $(U_x)$  ($x\in X$);
     define
    $$F:=\sum_{x\in\partial\alpha}u_xF_x$$
   \end{proof}

\begin{dfn}\label{isotopy_dfn}
i) A \emph{smooth partial triangulation}
 $K$ of a manifold with boundary $W$ consists of
a finite linear geometric simplicial complex $\vert K\vert$ and
 a topological embedding $h:\vert K\vert
\to W$, such that
\begin{itemize}
\item $h$
 embeds smoothly each cell of $\vert K\vert$ into $W$;
 \item $h\mun(\partial W)$ is a subcomplex of $\vert K\vert$.
 \end{itemize}
 
 ii) If moreover $h(\vert K\vert)=W$, then $K$ is a \emph{smooth triangulation} of $W$.
 
iii) An \emph{isotopy} of the smooth partial triangulation $K$ is
  a continuous family $(h_t)$ ($t\in\I$)
  of topological embeddings $h_t:\vert K\vert
\to W$ such that
\begin{itemize}
\item $h_0=h$;
\item For each cell $\alpha$ of $\vert K\vert$, the restriction $(h_t\vert\alpha)$
 is a smooth isotopy of smooth embeddings
of this cell;
\item $h_t\mun(\partial W)=h\mun(\partial W)$ for every $t\in\I$.
\end{itemize}
iv)  Given an isotopy as in (iii), consider  the simplexwise smooth,
   time-dependant vector field
$v_t:={\partial h_t}/{\partial t}$.
Fix an auxiliary smooth (resp. simplexwise linear) embedding of $W$ (resp. $\vert K\vert$)
in a large-dimensional
 Euclidian space. All Euclidian norms are denoted by $\vert \cdot\vert$.
We define the \emph{$C^1$-norm} of the isotopy as
$$\sup_{x\in\alpha}(\vert v_t(x)\vert+\vert\partial_\alpha v_t(x)\vert)$$
where $\partial_\alpha v_t(x)$ is
 the spatial differential of $v_t\vert\alpha$ at $x$.
\end{dfn}

\begin{lem}\label{ie_lem}
  Given a smooth partial triangulation $K$ of a manifold $W$
  and a subcomplex $L\subset K$, every  isotopy $(h_t)$ of $L$
  which is $C^1$-small enough,
   extends to some 
   isotopy $(H_t)$ of $K$ which we can make arbitrarily $C^1$-small.
    Moreover, the support of $(H_t)$ can be contained in any
   neighborhood of the support of $(h_t)$.
   \end{lem}
   \begin{proof} By induction on the number of cells of $K$ not in $L$,
   one is reduced to the case where $K=L\cup\alpha$ for a cell $\alpha$ such that
   $\partial\alpha=\alpha\cap L$. Since everything takes place in $W$ in a small neighborhood
   of the embedded simplex $\alpha$, we can assume that $W$ is the half space $\R^n_1$.

    The time-dependant vector field $\partial h_t/\partial t$ amounts over
     $\partial\alpha$ to a simplexwise smooth $1$-parameter family of maps
    $v_t:\partial\alpha\to\R^n$. By
    Lemma \ref{fe_lem}, this family extends to a simplexwise smooth
     $1$-parameter family of maps $V_t:\alpha\to\R^n$. In case $h(\alpha)\subset
     \partial W=\R^{n-1}$, one has $v_t(\partial\alpha)\subset\R^{n-1}$, hence one can
     choose $V_t$ to take values in $\R^{n-1}$. We can choose the values $V_t(x)$
     and the spacial differential $\partial_\alpha V_t$ to be arbitrarily close to $0$ on $\alpha$,
     provided that those of $v_t$ are close enough to $0$ on each proper face of $\alpha$.
    We can also choose $V_t$ to be supported in an arbitrary neighborhood
    of the support of $v_t$. Integrating the time-dependant vector field $V_t$,
     one obtains an isotopy of immersions ${h_t^\alpha}:\alpha\to W$
     such that each $h_t^\alpha$ is 
    close to $h\vert\alpha$ \emph{in the $C^1$ topology.} If close enough, then  
     the global map $$H_t:=h_t\cup h_t^\alpha:K\to W$$
      is one-to-one, hence
     a simplexwise smooth global
    {embedding}.
   \end{proof}
   
   Recall that a \emph{local coordinate chart}
    for $W$ at a point $y$ is a smooth diffeomorphism $\phi$
    between an open neighborhood of $y$ in $W$ and an open neighborhood
    of $\phi(y)$ in the half space $\R_1^n$.
    
\begin{dfn}\label{linear_dfn} The smooth partial
 triangulation $K$ of the $n$-manifold $W$ is \emph{linear}
at the point $y=h(x)$ with respect to the
 local coordinate chart $\phi$ if 
  the germ of  $\phi\circ h$ at $x$ is linear in every cell of $\vert K\vert$
   through $x$. We call $K$ \emph{linearizable} at $y$ if there exists such a
   local coordinate chart.
  \end{dfn}
  
  \begin{note}\label{cross_ratio_note}
  Beware that the linearizability
   property \emph{generically fails} provided that the combinatorial link of $x$
  in $K$ is complicated enough --- here, ``generically'' refers to the simplexwise smooth
  embedding $h$. Indeed, the linear representations of quivers come into play.
  
   For a simple example, consider a smooth
     triangulation $K$ of a manifold of dimension $n\ge 3$, and
    a codimension-$2$ cell $\alpha$ of $K$ lying in the boundary of at least
  four codimension-$1$ cells $\eta_i$ of $K$ ($1\le i\le 4$).
  At every point $y\in\alpha$,
  the four lines $\tau_y\eta_i/\tau_y\alpha$ have in the $2$-plane $\tau_yW/\tau_y\alpha$
  a cross-ratio $c(y)\in\R\setminus\{0,1\}$; and the function $c$ on $\alpha$ is in general not constant 
  on any neighborhood of a given point $y_0\in\alpha$. Then, $K$ is not linearizable at $y_0$.
  
  In the same spirit, note that for
  any vector field $\nabla$ of class $C^1$
  which is tangential to each cell of $K$, the derivate $\nabla\cdot c$ must vanish identically
  on $\alpha$. For a generic position of $\eta_1$, \dots, $\eta_4$, the function $c$ is Morse
  on $\alpha$. Then,
  $\nabla$ must be tangential to its level sets, and vanish at the critical points. If $\alpha$
  is contained in at least $n+1$ codimension-$1$ cells, then
  one gets a number $\dim(\alpha)$ of such cross-ratio functions on $\alpha$;
   and for a generic position of these cells,
  $\nabla$ must vanish identically on $\alpha$.
  \end{note}
  
  \begin{lem}\label{linearizing_lem}
   Given a smooth partial triangulation $K=(\vert K\vert, h)$ of a manifold $W$,
   a point $y=h(x)$, and a local coordinate chart $\phi:\nb_W(y)\to\R_1^n$,
    there is an isotopy $(h_t)$ of smooth partial triangulations
   such that $h_0=h$ and that $h_1$ is linear at $y$ with respect to $\phi$.
   
    Moreover, one can choose the isotopy
   to be  supported in an arbitrarily small neighborhood of $x$.

  Given moreover a
   subcomplex $L\subset K$ containing $x$ which
    is already linear at $y$ with respect to $\phi$,
   there is such an isotopy $(h_t)$ which is stationary on $L$.
   
  \end{lem}

  \begin{proof} 
  By induction on the number of
   cells of $\vert K\vert$ not in $\vert L\vert$ and containing $x$. Consider a minimal such
    cell $\alpha$; in particular, $$\partial\alpha\cap St(x)=\alpha\cap\vert L\vert\cap St(x)$$
    where $St(x)$ is the open star of $x$ (the interior of the union of the cells of $\vert K\vert$
    containing $x$). 
  The $1$-jet  at $x$ of the smooth embedding
  $\phi\circ h\vert\alpha$
 extends to a linear embedding $\ell:\alpha\hookrightarrow\R^n$ 
  coinciding with $\phi\circ h$ on $\partial\alpha\cap\nb_{\alpha}(x)$.
   The obvious local isotopy of embeddings $$h_t(z):=\phi\mun((1-t)\phi(h(z))+t\ell(z))
\   (z\in\nb_\alpha(x), t\in\I)$$
 is arbitrarily $C^1$-small (Definition \ref{isotopy_dfn}, iv)
  on a small enough neighborhood of $x$,
  and stationary on $\nb_{\partial\alpha}(x)$.
    By the ordinary isotopy extension property for embeddings,
     the germ of $(h_t)$ at $x$ extends to a global isotopy of embeddings
      $h_t^\alpha:\alpha\hookrightarrow W$,
    stationary on $\partial\alpha$. Clearly, $h_t^\alpha$ can be chosen to be
    supported
    in an arbitrarily small neighborhood of $x$, arbitrarily $C^1$-small,
    and stationary on $\vert L\vert$.
    One thus gets a $C^1$-small isotopy of $L\cup\alpha$ stationary on $L$
    and coinciding with $(h_t^\alpha)$ on $\alpha$, which extends to the rest
    of $K$ with support
    in an arbitrarily small neighborhood of $x$ (Lemma \ref{ie_lem}).
 \end{proof}

We now begin to prove Proposition \ref{cc_pro}.
Let $W$, $V_0$, $V_1$, $K$ be as in this proposition.
First, we somehow normalize $K$ close to the upper boundary
 $V_1$.

\begin{lem}\label{corona_lem} After changing $K$,
one can arrange that moreover, there is a  {smooth} function $w:W\to[0,1]$
such that
\begin{enumerate}
\item $V_i=w\mun(i)$ ($i=0, 1$);
\item Every critical value of $w$ lies in 
the open interval $(0,1/2)$;
\item The smooth hypersurface $w\mun(1/2)$ is a subcomplex of $K$;
\item The subcomplex $K_{1/2}:=K\vert w\mun([0,1/2])$ collapses onto $V_0$.
\end{enumerate}
\end{lem}
\begin{proof} Consider the triangulation $K\vert V_1$ of the manifold $V_1$;
put an arbitrary total order on the set of its vertices; then the Whitney
subdivision $L$ of the cartesian product
$(K\vert V_1)\times\I$ (see \cite{whitney_57} p. 365) is on
$V_1\times\I$ a smooth triangulation, and $L\vert(V_1\times 0)=K\vert V_1$,
and $L$ collapses onto $V_1\times 0$.
The manifold $\hat W:=W\cup_{V_1}(V_1\times\I)$
is diffeomorphic with $W$, and $\hat W$ admits
the smooth
 triangulation $\hat K:=K\cup_{K\vert V_1}L$. Composing the collapses $L\searrow V_1\times 0$
and $K\searrow V_0$, one gets a global collapse $\hat K\searrow V_0$. The function $\hat w$
on $\hat W$
is defined as $\hat w(x,t):=(1+t)/2$ in $V_1\times\I$, and its extension to the rest of
 $\hat W$ is not less
obvious.
\end{proof}

\begin{lem}\label{linearizing_K_lem}  After changing $K$ again,
one can arrange that moreover,  $K$ is linearizable at the barycenter of each of its cells.
\end{lem}

\begin{proof} Consider the hypersurface $H:=w\mun(1/2)$
 and its smooth triangulation
$K\vert H$. For every cell $\alpha$ of $K\vert H$,
apply Lemma \ref{linearizing_lem} at the barycenter of $\alpha$.
 One gets a smooth triangulation of $H$
isotopic to $K\vert H$ and linearizable at each barycenter.
Then, extend this isotopy into an isotopy of $K$ supported
in a small neighborhood of  $H$ (Lemma \ref{ie_lem}).
 Finally, applying Lemma \ref{linearizing_lem}
  at the barycenter of each cell of $K$ with $L=K\vert H$,
 one obtains a triangulation linearizable at the barycenter of each cell.
\end{proof}

Each triangulation bears a canonical, piecewise Euclidian, Riemannian metric. Precisely,
for each $d\ge 0$,  the Euclidian metric on $\R^{d+1}$, restricted to
 the standard $d$-simplex $\Delta^d\subset\R^{d+1}$ (the convex hull of the canonical basis),
  is
  invariant by the symmetric group ${\mathfrak S}_{d+1}$.
   One thus obtains a natural Riemannian metric on
 each linear simplex, and in particular on each cell of $\vert K\vert$.
 Its image under $h$ is a Riemannian metric $g_\alpha$ on each cell $\alpha$ of $K$.
Define on $\alpha$ the quadratic function
  $$q_\alpha:x\mapsto\frac{1}{2}\vert x-\flat(\alpha)\vert^2$$
where $\vert -\vert$ denotes of course the $g_\alpha$-distance.
On every face $\beta\subset\alpha$, one has $g_\alpha\vert\beta=g_\beta$ and (after Pythagora):
\begin{equation}\label{pythagora_eqn}
q_\alpha\vert\beta-q_\beta=\text{constant}
 \end{equation}
Still consider any cell $\alpha$ of $K$.
On a small open neighborhood $U_\alpha$ of $\flat(\alpha)$ in $W$,
let $\phi=(x_1,\dots, x_n)$ be a local coordinate chart linearizing $K$
(Definition \ref{linear_dfn} and Lemma \ref{linearizing_K_lem}).
 Put $d:=\dim(\alpha)$ and $a:=\phi(\flat(\alpha))\in\R_1^n$.
Composing $\phi$ with an affine automorphism of $\R^n$ preserving $\R^n_1$,
  we arrange that moreover:
  \begin{itemize}
   \item $\alpha\cap U_\alpha$ is defined in $U_\alpha$ by the equations
   $x_{i}=a_{i}$ ($d+1\le i\le n$);
   \item
       $g_\alpha\vert(\alpha\cap U_\alpha)=
    \sum_{i=1}^ddx_i^2$.
\end{itemize}
 Endow $U_\alpha$ with the quadratic function
   $$f_\alpha:=d-\frac{1}{2}\sum_{i=1}^d(x_i-a_i)^2+\frac{1}{2}\sum_{i=d+1}^n
   (x_i-a_i)^2$$
    (hence, $f_\alpha=q_\alpha+d$ on $\alpha\cap U_\alpha$)
     and with the linear
   vector field
  $$\nabla_\alpha:=\sum_{i=1}^d(x_i-a_i)\partial/\partial x_i-\sum_{i=d+1}^n(x_i
  -a_i)\partial/\partial x_i$$

 \begin{lem}\label{q_lem}\ 
 \begin{enumerate}
  \item $\nabla_\alpha$ is tangential in $U_\alpha$ to every cell
 of $K$ containing $\alpha$;
\item  For every face $\beta\subset\alpha$,
one has $\nabla_\beta\cdot q_\alpha>0$ on $\nb_\alpha(\flat(\beta))\setminus\flat(\beta)$.
\end{enumerate}
\end{lem}
The verifications are immediate.

\begin{lem}\label{gradient_lem} There are
 a global,  \emph{continuous}
 vector field $\nabla$ on $W$, and real values $0<\tau_0<1/2<\tau_1<1$,
  such that for every cell $\alpha\in K$:
  \begin{enumerate}
  \item $\nabla$ coincides with $\nabla_\alpha$ on some neighborhood of $\flat(\alpha)$
  in $W$;
  \item $\nabla$ is tangential to $\alpha$ and smooth on $\alpha$;
  \item
  $\nabla\cdot q_\alpha>0$ on $Int(\alpha)\setminus\flat(\alpha)$;
    \item $w(\flat(\alpha))\notin(0,\tau_0]\cup(1/2,\tau_1]$;
  \item Every critical value of $w$ is $>\tau_0$;
  \item $\nabla\cdot w<0$ on $w\mun(\tau_i)$ ($i=0,1$).

  \end{enumerate}
  \end{lem}
  \begin{proof}
   On every cell $\alpha$
  of $K$, let the vector field $r_{\alpha}$ be the (ascending)
    $g_\alpha$-gradient of
   $q_\alpha$.
   
    Let us first fix $\tau_0$ and $\tau_1$. Recall that $w\mun(0)=V_0$ and $w\mun(1/2)$
    are subcomplexes of $K$.
    
    Claim:
     \emph{For $i=0,1$, let $\alpha$ be a cell of $K$ contained in $w\mun[i/2,1]$
     but not in $w\mun(i/2)$. Then, 
    $r_\alpha\cdot w<0$ on $\alpha\cap w\mun(i/2)$.}
    
     Indeed, the intersection $\alpha\cap w\mun(i/2)$, if not empty,
      is a proper face of $\alpha$. At any point $x$ on  this face,
     the vector field $r_\alpha$ exits $\alpha$, transversely
     to every hyperface of $\alpha$ containing $x$. On the other hand,
      in the tangent vector space
     $\tau_xW$, the convex cone tangent to $\alpha$ admits $\ker(d_xw)$
     as a supporting hyperplane at point $0$.
    Hence, $r_\alpha\cdot w<0$
   at $x$. The claim is proved.
 
  For $i=0, 1$, we choose $\tau_i>i/2$ so close to $i/2$ that (5) is satisfied and that 
     for every cell $\alpha$ of $K$, (4) is satisfied,
    and (thanks to the claim)
  $r_\alpha\cdot w<0$ on $(w\vert\alpha)\mun(\tau_i)$.
  
  \medbreak
   One builds $\nabla$ cell after cell,
   by induction on the dimensions of the cells. Let $\alpha$ be a cell of $K$;
   assume by induction that a continuous vector field $\nabla$ is already
  defined on $\partial\alpha$; that (1), (2), (3) are satisfied for every proper face
  of $\alpha$ instead of $\alpha$; and that (6) holds at every point of
   $\partial\alpha\cap w\mun(\tau_i)$ ($i=0, 1$).
   
   The extension of $\nabla$ to $\alpha$ will be a local construction:
   we shall first define, for every point
   $x\in\alpha$, a vector field $X_x$ on an open neighborhood $N_x$ of $x$ in $\alpha$.
   The construction of $X_x$
    depends on $x$ being the barycenter of $\alpha$, or the barycenter of
   a proper face, or not a face barycenter.
  
    On $N_{\flat(\alpha)}:=Int(\alpha)$, let $X_{\flat(\alpha)}:=r_\alpha$. In particular,  $X_{\flat(\alpha)}\cdot q_\alpha>0$
    on $N_{\flat(\alpha)}\setminus\flat(\alpha)$.

  For each proper face $\beta\varsubsetneq\alpha$, let $N_{\flat(\beta)}$
  be a small open
  neighborhood of $\flat(\beta)$ in $\alpha$ (not containing any other face barycenter,
   and disjoint from
  $w\mun(\tau_0)$ and $w\mun(\tau_1)$)
   such that
  the linear vector field $X_{\flat(\beta)}:=\nabla_\beta\vert\alpha$, which is
  tangential to $\alpha$ there
  (Lemma \ref{q_lem} (1)), satisfies
    $X_{\flat(\beta)}\cdot q_\alpha>0$
  on $N_{\flat(\beta)}\setminus\flat(\beta)$ (Lemma \ref{q_lem} (2)).

    Now, consider a boundary point $x\in\partial\alpha$ which is not
  the barycenter of any face.
   Call $\beta\subset\partial\alpha$ the smallest face of $\alpha$ through $x$. By induction,
  $\nabla$ is already defined on each hyperface $\eta$ of $\alpha$ through $x$,
  and tangential to $\eta$.
    Let $N_x$ be  in $\alpha$ a small open neighborhood of $x$
   diffeomorphic to a sector.
   By
  Lemma \ref{sector_lem}, 
   there is a smooth vector field $X_x$ on $N_x$
  which coincides with $\nabla$ on every hyperface of $\alpha$ through $x$.
  By the induction hypothesis (3) applied to $\beta$ and by Equation (\ref{pythagora_eqn}),
   one has
   $\nabla\cdot q_\alpha>0$ at $x$. 
   After shrinking $N_x$, it does not contain any cell
   barycenter; and $X_x\cdot q_\alpha>0$ on $N_x$.
   Moreover, in case $w(x)=0$ or $1/2$, by induction, (6) holds at $x$;
   hence, after shrinking $N_x$, one has $X_x\cdot w<0$ on $N_x$.

   Finally, one defines $\nabla$ on $\alpha$ as
   $$\nabla:=\sum_{x\in\partial\alpha\cup\flat(\alpha)}u_xX_x$$
   where $(u_x)$ is a partition of the unity on $\alpha$
    subordinate to the open cover $(N_x)$. The properties (1) through (3) hold on $\alpha$;
    and (6) holds on $\alpha\cap w\mun(\tau_i)$ ($i=0, 1$).
   \end{proof}
   
   For every cell $\alpha$ of $K$,
   the stable manifold $W^s(\flat(\alpha))$ and the unstable manifold  $W^u(\flat(\alpha))$
   of the barycenter $\flat(\alpha)$
   with respect to $\nabla$ are well-defined, since $\nabla$, being simplexwise smooth,
   is Lipschitz.

   \begin{cor}\label{basin_cor}\ 
   \begin{enumerate}
   \item $W^s(\flat(\alpha))$ is contained in the open star of $\alpha$
  (the interior of the union of the cells of $K$ containing $\alpha$);
  \item
    $W^u(\flat(\alpha))=Int(\alpha)$.
    \end{enumerate}
   \end{cor}
   
   This follows immediately from Lemma \ref{gradient_lem} (1), (2), (3).

   \begin{lem}\label{function_lem} There is on $W$ a smooth Morse function $f$ such that
   \begin{enumerate}
   \item $f$ coincides with $f_\alpha$ on a neighborhood of $\flat(\alpha)$ in $W$, for every
   $\alpha\in K$;
   \item $\nabla\cdot f<0$ but at the barycenters of the cells
   of $K$.
   \end{enumerate}
   \end{lem}
   \begin{proof}  Let $\alpha$ be a cell of $K$.
   By induction on $d:=\dim(\alpha)$, assume that a smooth
    function $f_U$ is already defined
   on some small open neighborhood $U$ of $\partial\alpha$ in $W$;
   that for every proper face $\beta\varsubsetneq\alpha$,
    one has $f_U=f_\beta$ on a neighborhood of $\flat(\beta)$ in $U$;
    and that $\nabla\cdot f_U<0$ on $U$, but at the barycenters of the proper faces of
    $\alpha$.
   
    First, let us extend $f_U$ to $\alpha$ itself. To this aim, note that by Corollary
     \ref{basin_cor} (2),
    every local maximum of the continuous function $f_U\vert\partial\alpha$ is the barycenter
    $\flat(\eta)$
    of a hyperface $\eta\subset\alpha$; while $$f_U(\flat(\eta))=f_\eta(\flat(\eta))=d-1$$
     Hence, shrinking $U$ if necessary, one has $f_U<d-1/2$ on $U$.
     
     Again by Corollary \ref{basin_cor} (2), there is on $Int(\alpha)$ a smooth function $g$
     such that 
     \begin{itemize}
     \item $g=f_\alpha$ on a neighborhood of $\flat(\alpha)$;
     \item $\nabla\cdot g<0$ on $Int(\alpha)\setminus\flat(\alpha)$;
     \item $g>d-1/2$.
     \end{itemize}
     
      Again by Corollary \ref{basin_cor} (2),
       one has on $\alpha$ a smooth plateau function $u$ such that
      $u=1$ on a neighborhood in $\alpha$ of $\alpha\setminus(U\cap\alpha)$,
       and $\supp(u)\subset Int(\alpha)$,
      and $\nabla\cdot u\le 0$.
      On $\alpha$, define
       $$F:=(1-u)f_U+ug$$
       Then, on $Int(\alpha)\setminus\flat(\alpha)$:
       $$\nabla\cdot F=(1-u)\nabla\cdot f_U+u\nabla\cdot g+(g-f_U)\nabla\cdot u<0$$
       
       Hence, any smooth extension of $F$
        to $W$ coinciding with $f_\alpha$ close to $\flat(\alpha)$
       and with $f_U$ close to $\partial\alpha$ will also satisfy (2) on a neighborhood of
       $\alpha$.
   \end{proof}
   
  Note that $f$ is self-indexing.
  
   Informally, at this step, the vector field $\nabla$ is in some sense gradient-like for $f$;
   and every pair of critical points $\flat(\alpha_j)$, $\flat(\beta_j)$
   ($1\le j\le J)$ is in some sense in cancellation position with respect to $\nabla$.
    The problem is that $\nabla$ is in general not smooth;
    and in many cases \emph{cannot} be smooth (Note \ref{cross_ratio_note} above).
     In particular, for every cell $\alpha$
    of $K$,
   the stable manifold $W^s(\flat(\alpha))$  is in general not smooth. However, the
   \emph{germ}
    of  $W^s(\flat(\alpha))$ at $\flat(\alpha)$ is smooth (Lemma \ref{gradient_lem} (1)).
   
    Recall (Lemma \ref{corona_lem} (4)) that the subcomplex
     $K_{1/2}:=K\vert w\mun([0,1/2])$ collapses to $V_0$;
    such  a collapse means a filtration of the simplicial complex $K_{1/2}$ by subcomplexes
$(K_j)$, $0\le j\le J$, such that $K_0=K\vert V_0$,  and $K_J=K_{1/2}$,
 and that for each $1\le j\le J$:
\begin{itemize}
\item
 $K_{j-1}\subset K_{j}$;
 \item $K_j$ has exactly two cells $\alpha_j$, $\beta_j$ not in $K_{j-1}$;
 \item $\beta_j\subset\alpha_j$ is a hyperface.
 \end{itemize}

   \begin{lem}\label{v_lem} There are real values $(v_j)$ ($1\le j\le J$) such that
   \begin{enumerate}
   \item $\dim(\beta_j)<v_j<\dim(\alpha_j)$
   \item The compact disk $$D_j^s:=W^s(\flat(\beta_j))
   \cap f\mun(-\infty,v_j]$$ is contained in
   $U_{\beta_j}$, and thus smooth;
   \item $w<\tau_1$ on $D_j^s$;
   \item If $1\le j<j'\le k$ and if $\beta_{j}\subset\alpha_{j'}$, then $v_{j}<v_{j'}$.
   \end{enumerate}
   \end{lem}
   \begin{proof} One defines $v_j$ by descending induction on $j$. At each step,
   any value close enough to $\dim(\beta_j)$ from above works.
   \end{proof}
    Also consider, for $1\le j\le J$, the smooth compact disk
   $$D_j^u:=\alpha_j\cap
    f\mun[v_j,+\infty)$$
    In view of Corollary \ref{basin_cor} (2), of Lemma \ref{gradient_lem} (1) and (2), and
    of Lemma \ref{v_lem} (1),
    in the regular level set $f\mun(v_j)$,
    the two attachment spheres $\partial D_j^s$, $\partial D_j^u$ obviously
    meet transversely in a single point.
    
    Now, apply the elementary Weierstrass approximation theorem.
    The continuous vector field $\nabla$ on $W$ being already smooth
      on a neighborhood of the barycenter of every cell, and on every disk
    $D_j^u$ ($1\le j\le J$), and on a neighborhood of every disk $D_j^s$
    ($1\le j\le J$), we can approximate $\nabla$ on $W$,
    arbitrarily closely in the $C^0$ topology,
     by a smooth vector field $\tilde\nabla$ which coincides with
    $\nabla$ on a neighborhood of the barycenter of every cell, and on every disk
    $D_j^u$  ($1\le j\le J$), and on a neighborhood of every disk $D_j^s$ ($1\le j\le J$).
    We can choose $\tilde\nabla$ to be, like $\nabla$, tangential to $V_0$ and $V_1$;
    hence the flow $\tilde\nabla^t$ is well-defined on $W$ for $t\in\R$.
     We choose $\tilde\nabla$ so close
    to $\nabla$, in the $C^0$ topology,
     that $\tilde\nabla\cdot f<0$ but at the barycenters of the cells,
    and that $\tilde\nabla\cdot w<0$ on $w\mun(\tau_i)$ ($i=0, 1$).
    
   Thus, $\tilde\nabla$ is a genuine smooth, gradient-like vector field for $f$.
   Consider, for every $1\le j\le J$,  the unstable manifold
    $\tilde W^u(\flat(\alpha_j))$ of the critical point $\flat(\alpha_j)$  for $\tilde\nabla$,
    and the stable manifold $\tilde W^s(\flat(\beta_j))$
     of the critical point $\flat(\beta_j)$ for $\tilde\nabla$. By the above definition
      of $\tilde\nabla$,
     one has
     \begin{equation}\label{u_eqn}
      D_j^u\subset\tilde W^u(\flat(\alpha_j))
      \end{equation}
      \begin{equation}\label{s_eqn}
       D_j^s\subset\tilde W^s(\flat(\beta_j))
     \end{equation}
    We can arrange moreover that
    \begin{equation}\label{over_eqn}
    w\vert D_j^u>\tau_0
    \end{equation}
     ( $1\le j\le J$). Indeed,
    after the inclusion (\ref{u_eqn}) and
      Lemma \ref{gradient_lem} (4), for a large enough time $T>0$,
      one has $w\circ\tilde\nabla^{-T}>\tau_0$ on each $D_j^u$. We change $w$ on $W$
      to $w\circ\tilde\nabla^{\theta\circ w}$, where $\theta$ is any smooth
      function on $\I$ such that $\theta(\tau_0)=-T$ and $\theta(\tau_1)=0$.

      Then, consider 
        $W':=w\mun[\tau_0,\tau_1]$, a cobordism between $V'_0:=w\mun(\tau_0)$
        and $V'_1:=w\mun(\tau_1)$,
           diffeomorphic with $W$
           since the critical values of $w$ lie between $\tau_0$ and $\tau_1$.
   We shall prove that $W'$ is diffeomorphic with $V'_0\times\I$.

     The critical points of $f$ in $W'$ are exactly the pairs $\flat(\alpha_j)$,
      $\flat(\beta_j)$ ($1\le j\le J$).
     For each $j$, one has  $D_j^u\subset W'$
      (by the inequation (\ref{over_eqn}), and 
      since $\tilde\nabla\cdot w<0$ on $V'_1$)   and $D_j^s\subset W'$
        (by Lemma \ref{v_lem} (3), and
     since $\tilde\nabla\cdot w<0$ on $V'_0$).
     
     The major consequence of the inclusions
      (\ref{u_eqn}), (\ref{s_eqn}) is of course that
      $\tilde W^u(\flat(\alpha_j))$ and  $\tilde W^s(\flat(\beta_j))$ meet transversely
     on a single gradient line.
     It follows that the pair of critical points $\flat(\alpha_j)$, $\flat(\beta_j)$ can be cancelled
     in Morse's way
     by changing $f$ \emph{in an arbitrarily small neighborhood of
     the bouquet $D_j^u\cup D_j^s$.}
     For this localization, see Lemma 2.3 of \cite{meigniez_17} --- or alternatively,
     choose $\epsilon>0$ so small that no critical value of $f$
     is $\epsilon$-close to $v_j$;
     then apply, in an arbitrarily small neighborhood of
     the bouquet, the standard modification of $f$ that brings both critical values
     to $v_j\pm\epsilon/2$
      without changing the pseudogradient; and finally verify in the proof
     of Theorem 5.4 in \cite{milnor_65} that the modification of $f$ in the cobordism
     $f\mun[v_j-\epsilon,v_j+\epsilon]$ takes place in an arbitrary
     small neighborhood of the union of the stable and unstable manifolds of the two critical points.

     After Corollary \ref{basin_cor} (1) and (2)
      and Lemma \ref{v_lem} (4), the bouquets $D_j^u\cup D_j^s$
     ($1\le j\le J$) are two by two disjoint.
     Hence, one can perform
     simultaneously the $J$ cancellations, yielding  on $W'$ a function $f'$
     without critical points
     which coincides with $f$ on
     neighborhoods of $V'_0$ and $V'_1$.
     
      The end of the proof of Proposition \ref{cc_pro}
      is much classical:
     by means of a partition of the unity,
     one makes on $W'$ a smooth vector field $\nabla'$ such that $\nabla'=\tilde\nabla$
     on neighborhoods of $V'_0$ and $V'_1$, and that $\nabla'\cdot f'<0$ on $W'$.
     Consequently, every orbit of $\nabla'$ must descend from $V'_1$ to $V'_0$.
    After a rescaling of $\nabla'$, the time $t=-1$ of the flow
     $(\nabla')^t$ carries $V'_0$ onto $V'_1$.
     The mapping $(x,t)\mapsto(\nabla')^{-t}(x)$ thus
     yields a diffeomorphism of $V'_0\times\I$ with $W'$.

\bigskip
\noindent
Universit\'e de Bretagne Sud

\noindent
Laboratoire de Math\'ematiques de Bretagne Atlantique

\noindent
LMBA UBS (UMR CNRS 6205)

\noindent B.P. 573

\noindent
F-56019 VANNES CEDEX, France

\noindent
Gael.Meigniez@univ-ubs.fr
\end{document}